\documentclass{article}
\usepackage{amsmath,amsthm}
\usepackage{amssymb}
\usepackage{hyperref}
\usepackage{authblk}
\usepackage{dsfont}
\usepackage{graphicx}

\newtheorem{rmk}{Remark}
\newtheorem{assumption}{Assumption}

\newtheorem{prop}{Proposition}
\newtheorem{corol}{Corollary}
\newtheorem{lm}{Lemma}
\newtheorem{theorem}{Theorem}

\newcommand{\R}{\mathbb{R}}
\newcommand{\bX}{\mathbf{X}}
\newcommand{\bx}{\mathbf{x}}
\newcommand{\bm}{\mathbf{m}}
\newcommand{\bM}{\mathbf{M}}
\newcommand{\cX}{\mathcal{X}}
\newcommand{\bZ}{\mathbf{Z}}
\newcommand{\bk}{\mathbf{k}}
\newcommand{\bh}{\mathbf{h}}
\newcommand{\ba}{\mathbf{a}}
\newcommand{\by}{\mathbf{y}}

\newcommand{\N}{\mathbb{N}}

\newcommand{\E}{\mathrm{E}}
\newcommand{\PP}{\mathbb{P}}
\newcommand{\V}{\mathrm{Var}}

\DeclareMathOperator{\ind}{\perp \!\!\! \perp}

\usepackage[title]{appendix}

\date{}
\begin{document}

\title{Variance reduction for estimation of Shapley effects and adaptation to unknown input distribution} 

\author[1]{Baptiste Broto}
\author[2]{Fran\c{c}ois Bachoc }
\author[1]{Marine Depecker}

\affil[1]{CEA, LIST, Universit\'e Paris-Saclay, F-91120, Palaiseau, France}
\affil[2]{Institut de Math\'ematiques de Toulouse, Universit\'e Paul Sabatier, F-31062 Toulouse, France}

\maketitle

\begin{abstract}
The Shapley effects are global sensitivity indices: they quantify the impact of each input variable on the output variable in a model. In this work, we suggest new estimators of these sensitivity indices. When the input distribution is known, we investigate the already existing estimator defined in \cite{song_shapley_2016} and suggest a new one with a lower variance. Then, when the distribution of the inputs is unknown, we extend these estimators. We provide asymptotic properties of the estimators studied in this article. We also apply one of these estimators to a real data set.
\end{abstract}

\section{Introduction}
Sensitivity indices are important tools in sensitivity analysis. They aim to quantify the impact of the input variables on the output of a model. In this way, they give a better understanding of numerical models and improve their interpretability. For example, the sensitivity indices enable to know if the variation of a specific input variable can lead to an important variation of the output or not. 

In global sensitivity analysis, the input variables $X_1,...,X_p$ are assumed to be random variables. Sobol defined the first sensitivity indices for a general framework, called the Sobol indices, in \cite{sobol_sensitivity_1993}. Many other sensitivity indices have been defined and studied (see \cite{borgonovo_common_2016} for a general review of these indices). Nevertheless, many of these indices suffer from a lack of interpretation when the input variables are dependent. To overcome this lack of interpretation, many variants of the Sobol indices have been suggested for dependent input variables (see for example \cite{jacques_sensitivity_2006}, \cite{mara_variance-based_2012} and \cite{chastaing_indices_2013}).

Recently, Owen defined new sensitivity indices in \cite{owen_sobol_2014} called "Shapley effects" that have beneficial properties and that are easy to interpret, even in the dependent case. The main advantages of these sensitivity indices compared to the Sobol indices (and their variants) are: they remain positive, their sum is equal to one and there is exactly one index for each input (and there are no indices for groups of variables).
The Shapley effects are based on the notion of "Shapley value", that originates from game theory in \cite{shapley_value_1953}. The Shapley value has been widely studied (\cite{colini-baldeschi_variance_2016}, \cite{fatima_linear_2008}) and applied in different fields (see for example \cite{moretti_combining_2008} or \cite{hubert_strategic_2003}). However, only few articles focus on the Shapley effects in sensitivity analysis (see \cite{owen_sobol_2014,song_shapley_2016,owen_shapley_2017,
iooss_shapley_2017,broto_sensitivity_2019,benoumechiara_shapley_2019}).  Song et al. suggested an algorithm to estimate the Shapley effects in \cite{song_shapley_2016} that is implemented in the R package "{\verb|sensitivity|}".

In this paper, we work on the Shapley effects and their estimation. We divide this estimation into two parts. The first part is the estimation of quantities that we call the "conditional elements", on which the Shapley effects depend.  The second part consists in aggregating the estimates of the conditional elements in order to obtain estimates of the Shapley effects. We call this part the $W$-aggregation procedure. We refer to Sections \ref{section_Wu} and \ref{section_structure} for more details on these two parts.

First, we focus on the estimation of the conditional elements with two different estimators: the double Monte-Carlo estimator (used in the algorithm of \cite{song_shapley_2016}) and the Pick-and-Freeze  estimator (see \cite{homma_importance_1996} for the independent case) that we extend to the case where the inputs are dependent. We present the two estimators when it is possible to sample from the conditional distributions of the input vector. Then we suggest a new $W$-aggregation procedure, based on the subsets of $\{1,...,p\}$, to estimate all the Shapley effects (for all the input variables) at the same time. We choose the best parameters to minimize the sum of the variances of all the Shapley effects estimators. The algorithm of \cite{song_shapley_2016} uses a $W$-aggregation procedure based on permutations of $\{1,...,p\}$. We study this $W$-aggregation procedure and explain how it minimizes the variance of the estimates of the Shapley effects. Our suggested $W$-aggregation procedure provides an improved accuracy, compared to the $W$-aggregation procedure in \cite{song_shapley_2016}, using all the estimates of the conditional elements for all the estimates of the Shapley effects. The comparison between the two $W$-aggregation procedures is illustrated with numerical experiments. These experiments also show that the double Monte-Carlo estimator provides better results than the Pick-and-Freeze estimator.

Then, we extend the estimators of the conditional elements (the double Monte-Carlo estimator and the Pick-and-Freeze estimator) to the case where we only observe an i.i.d. sample from the input variables. The extension relies on nearest-neighbour techniques, which are widely used for many non-parametric estimation problems \cite{berrett_efficient_2019,berrett_efficient_2019-1}. To the best of our knowledge, the estimators we suggest are the first that do not require exact samples from the conditional distributions of the input variables. One of our main results is the consistency of these estimators under some mild assumptions, and their rate of convergence under additional regularity assumptions. We then give the consistency of the estimators of the Shapley effects with the two $W$-aggregation procedures and using the double Monte-Carlo estimator or the Pick-and-Freeze estimator. We observe, in numerical experiments, that the estimators of the Shapley effects have a similar accuracy as when it is possible to sample from the conditional distributions. We also apply one of these estimators on meteorological data, more specifically on the output of three different metamodels predicting the ozone concentration in function of nine input variables (with some categorical variables and some continuous variables). This application enables to study the influence of the inputs variables on black-box machine learning procedures.

The paper is organized as follows. In Section \ref{section_notations}, we define the framework of global sensitivity analysis and we recall the definition and some properties of the Shapley effects. In Section \ref{section_Wu}, we assume that the input distribution is known and we present the two methods to estimate the conditional elements. In Section \ref{section_structure}, we suggest a new $W$-aggregation procedure and we study the $W$-aggregation procedure used by the algorithm of \cite{song_shapley_2016}. In Section \ref{section_loi_connue}, we summarize the four estimators of the Shapley effects, give their consistency and we illustrate them with numerical applications. In Section \ref{section_loi_inconnue}, we assume that the input distribution is unknown and that we just observe a sample of the input vector. We give consistent estimators of the conditional elements and thus consistent estimators of the Shapley effects in this case, and we illustrate this with numerical experiments. In Section \ref{section_applireal}, we apply one of our estimators to a real data set. We conclude in Section \ref{section_conclusion}. All the proofs are provided in the appendix.

\section{The Shapley effects}\label{section_notations}

We let $\bX=(X_1,...,X_p)$ be the input random vector on the input domain $\cX=\cX_1\times ... \times \cX_p$ with distribution $\PP_\bX$. We assume that there is an output variable $Y$ in $\R$ defined by
\begin{equation}
Y=f(\bX),
\end{equation}
with $f\in L^2(\PP_\bX)$.
We write $[1:p]$ for the set $\{1,2,...,p\}$. For any non-empty $u\subset [1:p]$, letting $u=\{i_1,...,i_r\}$ with $i_1<i_2<...<i_r$, we define $\bX_u=(X_{i_1},...,X_{i_r}).$ We can now define the conditional elements $(W_u)_{u\subset[1:p]}$ and the Shapley effects $(\eta_i)_{i\in [1:p]}$.

For all $u\subset [1:p]$, we define:
\begin{equation}
V_u:=\V(\E(Y|\bX_u))
\end{equation}
and
\begin{equation}
E_u:=\E(\V(Y|\bX_{-u})),
\end{equation}
where $-u:=[1:p]\setminus u$. We let by convention $\E(Y|\bX_\emptyset)=\E(Y)$ and $\V(Y|\bX_\emptyset)=\V(Y)$. We define the conditional elements $(W_u)_{u\subset[1:p]}$ as being either $(V_u)_{u\subset[1:p]}$ or $(E_u)_{u\subset[1:p]}$. For all $i\in [1:p]$, we define the Shapley effect $\eta_i$ as in \cite{song_shapley_2016} by:
\begin{equation}\label{eq_def_shap}
\eta_i:=\frac{1}{p\V(Y)}\sum_{u\subset -i} \begin{pmatrix}
p-1\\ |u|
\end{pmatrix}^{-1} (W_{u\cup\{i\}}-W_u),
\end{equation}
where we define $-i$ as the subset $[1:p]\setminus \{i\}$ and $|u|$ as the cardinality of $u$.

\begin{rmk}
As explained in \cite{song_shapley_2016}, the Shapley effects do not depend on whether $(W_u)_{u\subset[1:p]}$ denotes $(V_u)_{u\subset[1:p]}$ or $(E_u)_{u\subset[1:p]}$.
\end{rmk}

\begin{rmk}\label{rmk_psi}
The quantities $W_{\emptyset}$ and $W_{[1:p]}$ are equal to $0$ and $\V(Y)$ respectively. The variance of $Y$ is easy to estimate, so we assume without loss of generality that we know the theoretical value $\V(Y)$.
\end{rmk}
We can notice that the Shapley effects are a sum over the subsets $u\subset -i$. Another classical way to compute the Shapley effects is to sum over the permutations of $[1:p]$, see Proposition \ref{prop_perm}. We let $\mathcal{S}_p$ be the set of permutations of $[1:p]$. An element $\sigma \in \mathcal{S}_p$ is a bijective function from $[1:p]$ to $[1:p]$. We let $\sigma^{-1}$ be its inverse function. As in \cite{song_shapley_2016}, for $i\in [1:p]$ and $\sigma \in\mathcal{S}_p$, we let $P_i(\sigma):=\{\sigma(j)|\; j \in [1: \sigma^{-1}(i)-1] \}$.

\begin{prop}\label{prop_perm} [Equation (11) in \cite{song_shapley_2016}, Section 4.1 (see also \cite{castro_polynomial_2009})]
We have
\begin{equation}\label{eq_prop_perm}
\eta_i=\frac{1}{p!\V(Y)} \sum_{\sigma \in \mathcal{S}_p}(W_{P_i(\sigma)\cup \{i\}}- W_{P_i(\sigma))}).
\end{equation}
\end{prop}

Our aim is to estimate the Shapley effects. We have seen two different ways to compute the Shapley effects, given by Equation \eqref{eq_def_shap} (with a sum over the subsets) and Equation \eqref{eq_prop_perm} (with a sum over the permutations). These two equations will correspond to two different $W$-aggregation procedures of the Shapley effects.

\section{Estimation of the conditional elements}\label{section_Wu}

We explain now how to estimate these $(W_u)_{\emptyset \varsubsetneq u \varsubsetneq [1:p]}$ in a restricted setting (recall that $W_\emptyset=0$ and $W_{[1:p]}=\V(Y)$ are known). The restricted setting is the following: as in \cite{song_shapley_2016}, we will assume that for any $\emptyset \varsubsetneq u \varsubsetneq [1:p]$ and $\bx_u \in \mathcal{X}_u:=\prod_{i\in u} \mathcal{X}_i$, it is feasible to generate an i.i.d. sample from the distribution of $\bX_{-u}$ conditionally to $\bX_u=\bx_u$. Moreover, we assume that we have access to the computer code of $f$. 

To estimate $W_u$, we suggest two different estimators. The first one consists in a double Monte-Carlo procedure to estimate $E_u$, and it is the estimator used in the algorithm of \cite{song_shapley_2016}. The other one is the well-known Pick-and-Freeze estimator (see \cite{homma_importance_1996} for the first definition, \cite{gamboa_sensitivity_2014,gamboa_statistical_2016} for theoretical studies) for $V_u$, that we extend to the case where the input variables $(X_i)_{i\in [1:p]}$ are not independent.

Finally, we assume that each evaluation of $f$ is costly, so we define the cost of each estimator $\widehat{W}_u$ as the number of evaluations of $f$.

\subsection{Double Monte-Carlo}\label{section_doubleMC}
A first way to estimate $E_u=\E(\V(Y|\bX_{-u}))$ is using double Monte-Carlo: a first Monte-Carlo step of size $N_I$ for the conditional variance, another one of size $N_u$ for the expectation. Thus, the estimator of $E_u$ suggested in \cite{song_shapley_2016} is
\begin{equation}\label{eq_VuMC_connue}
\widehat{E}_{u,MC}:=\frac{1}{N_u}\sum_{n=1}^{N_u} \frac{1}{N_I-1}\sum_{k=1}^{N_I}\left(f(\bX_{-u}^{(n)},\bX_{u}^{(n,k)})-\overline{f(\bX_{-u}^{(n)})}\right)^2,
\end{equation}
where for $n=1,...,N_u$, $\overline{f(\bX_{-u}^{(n)})}:=N_I^{-1}\sum_{k=1}^{N_I}f(\bX_{-u}^{(n)},\bX_{u}^{(n,k)})$, $(\bX_{-u}^{(n)})_{n \in [1:N_u]}$ is an i.i.d. sample with the distribution of $\bX_{-u}$ and $(\bX_{u}^{(n,k)})_{k \in [1:N_I]}$ conditionally to $\bX_{-u}^{(n)}$ is i.i.d. with the distribution of $\bX_{u}$ conditionally to $\bX_{-u}=\bX_{-u}^{(n)}$. For all $n\in [1:N_u]$, the computation of 
$$
\frac{1}{N_I-1}\sum_{k=1}^{N_I}\left(f(\bX_{-u}^{(n)},\bX_{u}^{(n,k)})-\overline{f(\bX_{-u}^{(n)})}\right)^2
$$
requires the values of $\left(f(\bX_{-u}^{(n)},\bX_{u}^{(n,k)})\right)_{k \in [1:N_I]}$. We will take $N_I=3$, as suggested in \cite{song_shapley_2016}. Thus, the double Monte-Carlo estimator given in Equation \eqref{eq_VuMC_connue} has a cost (number of evaluations of $f$) of $3N_u$.

We remark that for $\bx^{(1)}, \bx^{(2)}\in \mathcal{X}$ and for $\emptyset \varsubsetneq u \varsubsetneq [1:p]$, we let $(\bx_u^{(1)},\bx_{-u}^{(2)})$ be the element $\mathbf{v}\in \mathcal{X}$ such that $\mathbf{v}_u=\bx_u^{(1)}$ and $\mathbf{v}_{-u}=\bx_{-u}^{(2)}$, and we let $f(\bx_u^{(1)},\bx_{-u}^{(2)}):=f(\mathbf{v})$. We use this notation throughout the paper.

\begin{rmk}
The estimator of Equation \eqref{eq_VuMC_connue} is an unbiased estimator of $E_u=\E(\V(Y|\bX_{-u}))$.
\end{rmk}

\subsection{Pick-and-freeze}
We now provide a second estimator of $W_u$: the Pick-and-Freeze estimator for $V_u$. We have
$$
V_u=\V(\E(Y|\bX_u))=\E(\E(Y|\bX_u)^2)-\E(Y)^2.
$$
Remark that $\E(Y)$ is easy to estimate so we assume without loss of generality that we know the value of $\E(Y)$ (for the numerical applications, we will take the empirical mean). It remains to estimate $\E(\E(Y|\bX_u)^2)$, which seems to be complicated. We prove the following proposition that enables to simplify the formulation of this quantity.
\begin{prop}\label{prop_paf}
Let $\bX=(\bX_u,\bX_{-u})$ and $\bX^u=(\bX_u,\bX'_{-u})$ of distribution $\PP_\bX$ such that, a.s. $\PP_{(\bX_{-u},\bX_{-u}')|\bX_u=\bx_u}=\PP_{\bX_{-u}|\bX_u=\bx_u}\otimes \PP_{\bX_{-u}|\bX_u=\bx_u}$. We have
\begin{equation}\label{paf}
\E(\E(Y|\bX_u)^2)=\E(f(\bX)f(\bX^u)).
\end{equation}
\end{prop}
Remark that Proposition \ref{prop_paf} enables to write a double expectation as one single expectation, that we estimate by a simple Monte-Carlo.
Thus, we suggest the Pick-and-Freeze estimator, for $\emptyset \varsubsetneq u \varsubsetneq [1:p]$,
\begin{equation}\label{eq_vuPF_connue}
\widehat{V}_{u,PF}:=\frac{1}{N_u}\sum_{n=1}^{N_u} f \left( \bX_u^{(n)}, \bX_{-u}^{(n,1)} \right) f \left( \bX_u^{(n)}, \bX_{-u}^{(n,2)} \right)- \E(Y)^2,
\end{equation}
where $(\bX_{u}^{(n)})_{n\in [1:N_u]}$ is an i.i.d. sample with the distribution of $\bX_{u}$ and where $\bX_{-u}^{(n,1)}$ and $\bX_{-u}^{(n,2)}$ conditionally to $\bX_{u}^{(n)}$ are independent with the distribution of $\bX_{-u}$ conditionally to $\bX_{u}=\bX_{u}^{(n)}$. This estimator has a cost of $2N_u$.

\section{ \texorpdfstring{$W$}{Lg}-aggregation procedures}\label{section_structure}
As we can see in Equation \eqref{eq_def_shap} or in Equation \eqref{eq_prop_perm}, the Shapley effects are functions of the conditional elements $(W_u)_{u\subset [1:p]}$. In Section \ref{section_Wu}, we have seen how to estimate these conditional elements when it is possible to sample from the conditional distributions of the input vector. In this section, we assume that we have estimators $(\widehat{W}_u)_{u\subset [1:p]}$. From Remark \ref{rmk_psi}, we let $\widehat{W}_{\emptyset}=W_\emptyset=0$ and $\widehat{W}_{[1:p]}=W_{[1:p]}=\V(Y)$. We also add the following assumption that will be needed for the theoretical results that we will prove.
\begin{assumption}\label{hypNu}
For all $\emptyset \varsubsetneq u \varsubsetneq [1:p]$, $\widehat{W}_{u}$ is computed with a cost $\kappa N_u$ by $
\widehat{W}_{u}=\frac{1}{N_u}\sum_{n=1}^{N_u}\widehat{W}_{u}^{(n)}
$ where the $(\widehat{W}_{u}^{(n)})_{n\in [1:N_u]}$ are independent and identically distributed. The $(\widehat{W}_{u})_{u\subset[1:p]}$ are independent. The integer $\kappa\in \N^*$ is the number of evaluations of the computer code $f$ (i.e. the cost) for each $\widehat{W}_u^{(n)}$.
\end{assumption}
Assumption \ref{hypNu} means that we estimate the $(W_u)_{\emptyset \varsubsetneq u \varsubsetneq [1:p]}$ by Monte-Carlo, independently and with different costs $( \kappa N_u)_{\emptyset \varsubsetneq u \varsubsetneq [1:p]}$. The accuracy $N_u$ corresponds to computing $N_u$ independent and identically distributed estimators $\widehat{W}_u^{(1)},...,\widehat{W}_u^{(N_u)}$ that are averaged. We have seen in Section \ref{section_Wu} two estimators that satisfy Assumption \ref{hypNu}: the double Monte-Carlo estimator (with $\kappa=3$) and the Pick-and-Freeze estimator (with $\kappa=2$).

We call "$W$-aggregation procedure" an algorithm that estimates the Shapley effects from the estimates $(\widehat{W}_u)_{\emptyset \varsubsetneq u \varsubsetneq [1:p]}$ and that selects the values of the accuracies $(N_u)_{\emptyset \varsubsetneq u \varsubsetneq [1:p]}$. We first suggest a new $W$-aggregation procedure. Then we obtain a theoretical insight on the $W$-aggregation procedure of \cite{song_shapley_2016}.

\subsection{The subset procedure}
In this section, we suggest a new $W$-aggregation procedure for the Shapley effects. This procedure consists in computing once for all the estimates $\widehat{W}_u$ for all $u\subset[1:p]$, and to store them. Then, we use these estimates to estimate all the Shapley effects.

\subsubsection{The \texorpdfstring{$W$}{Lg}-aggregation procedure}
We suggest to estimate the Shapley effects $(\eta_i)_{i\in [1:p]}$ by using the following $W$-aggregation procedure:

\begin{enumerate}
\item For all $u \subset [1:p]$, compute $\widehat{W}_u$.
\item For all $i\in [1:p]$, estimate $\eta_i$ by
\begin{equation}\label{ShEX}
\widehat{\eta}_i:=\frac{1}{p\V(Y)}\sum_{u\subset -i} \begin{pmatrix}
p-1\\ |u|
\end{pmatrix}^{-1} (\widehat{W}_{u\cup\{i\}}-\widehat{W}_u).
\end{equation}
\end{enumerate}

We call this $W$-aggregation procedure "subset $W$-aggregation procedure".
We can note that each estimate $\widehat{W}_u$ is used for all the estimates $(\widehat{\eta}_i)_{i\in [1:p]}$. It remains to choose the values of the accuracies $(N_u)_{\emptyset \varsubsetneq u \varsubsetneq [1:p]}$.

\subsubsection{Choice of the accuracy of each \texorpdfstring{$\widehat{W}_u$}{Lg}}\label{section_variance}
In this section, we explain how to choose the values of the accuracies $(N_u)_{\emptyset \varsubsetneq u \varsubsetneq [1:p]}$. In the following proposition, we give the best choice of the accuracies $(N_u)_{\emptyset \varsubsetneq u \varsubsetneq [1:p]}$ to minimize $\sum_{i=1}^p \V(\widehat{\eta}_i)$ for a fixed total cost $\kappa \sum_{\emptyset \varsubsetneq u \varsubsetneq [1:p]}N_u$.

\begin{prop}\label{prop_min_variance}
Let a total cost $N_{tot}\in \N$ be fixed.
Under Assumption \ref{hypNu}, if the Shapley effects are estimated with the subset $W$-aggregation procedure, the solution of the relaxed program (i.e. the problem without the constraint of letting the $(N_u)_{\emptyset \varsubsetneq u \varsubsetneq [1:p]}$ be integers)
\begin{equation}\label{eq_otimisation}
\min_{(N_u)_{\emptyset \varsubsetneq u\varsubsetneq [1:p]}\in (0,+\infty)^{2^p-2}}\sum_{i=1}^p \V(\widehat{\eta}_i) \;\;\;\text{\;\; subject to\;\;\;\;} \kappa \sum_{\emptyset \varsubsetneq u\varsubsetneq [1:p]} N_u= N_{tot}
\end{equation}
is $(N_u^{*})_{\emptyset \varsubsetneq u\varsubsetneq [1:p]}$ with for all $\emptyset \varsubsetneq u \varsubsetneq [1:p]$
$$
N_u^{*}=\frac{N_{tot}}{\kappa}\frac{\sqrt{(p-|u|)!|u|!(p-|u|-1)!(|u|-1)! \V (\widehat{W}_{u}^{(1)})}}{\sum_{\emptyset \varsubsetneq v\varsubsetneq [1:p]} \sqrt{(p-|v|)!|v|!(p-|v|-1)!(|v|-1)! \V (\widehat{W}_{v}^{(1)})}}.
$$
\end{prop}
Usually, we do not know the values of $\V(\widehat{W}_{u}^{(1)})$ for $\emptyset \varsubsetneq u \varsubsetneq [1:p]$, but we need them to compute the value of $N_u^{*}$. In practice, we will assume that these values are equal in order to compute $N_u^{*}$.
Furthermore, the sum over the subsets $v$ such that $\emptyset \varsubsetneq v\varsubsetneq [1:p]$ can be too costly to compute. Hence, we make the following approximations in practice:
\begin{equation}\label{eq_heuri}
N_u^{*} \approx \frac{ \frac{N_{tot}}{\kappa}\begin{pmatrix}
p\\ |u|
\end{pmatrix}^{-\frac{1}{2}} \begin{pmatrix}
p\\ |u|-1
\end{pmatrix}^{-\frac{1}{2}}}{\sum_{\emptyset \varsubsetneq v\varsubsetneq [1:p]} \begin{pmatrix}
p\\ |v|
\end{pmatrix}^{-\frac{1}{2}} \begin{pmatrix}
p\\ |v|-1
\end{pmatrix}^{-\frac{1}{2}} }
 \approx \frac{ \frac{N_{tot}}{\kappa}\begin{pmatrix}
p\\ |u|
\end{pmatrix}^{-1} }{\sum_{\emptyset \varsubsetneq v\varsubsetneq [1:p]} \begin{pmatrix}
p\\ |v|
\end{pmatrix}^{-1}  }=\frac{N_{tot}}{\kappa}\frac{\begin{pmatrix}
p\\|u|
\end{pmatrix}^{-1}}{p-1}.
\end{equation}
Hence, when implementing the subset $W$-aggregation procedure, we will choose $N_u^{*}$ as 
\begin{equation}\label{eq_N_u^*}
N_u^*:=\mathrm{Round}\left( N_{tot}\kappa^{-1}\begin{pmatrix}
p\\ |u|
\end{pmatrix}^{-1} (p-1)^{-1} \right)
\end{equation}
for $\emptyset \varsubsetneq u \varsubsetneq [1:p]$, where $\mathrm{Round}$ is the nearest integer function. In this way, for a fixed total cost, we take the accuracies $(N_u)_{\emptyset \varsubsetneq u \varsubsetneq [1:p]}$ near the optimal choice that minimizes $\sum_{i=1}^p\V(\widehat{\eta}_i)$. Hence, the parameter $N_{tot}$ is now the only parameter left to choose. In practice, this parameter is often imposed as a global budget constraint.

\begin{rmk}\label{rmk_approx}
With the approximation discussed above, the real total cost $\kappa \sum_{\emptyset \varsubsetneq u\varsubsetneq [1:p]}N_u$ can be different from the $N_{tot}$ chosen (because of the approximations and the choice of the closest integer). In this case, we suggest to adapt the value of $N_{tot}$ in order to make the total cost $\kappa \sum_{\emptyset \varsubsetneq u\varsubsetneq [1:p]}N_u^{*}$ take the desired value. 
\end{rmk}

\begin{rmk}\label{rmk_Var=}
In order to compute the $(N_u^*)_{\emptyset \varsubsetneq u \varsubsetneq [1:p]}$ in practice, we assume that the values of $\V(\widehat{W}_{u}^{(1)})$, for $\emptyset \varsubsetneq u \varsubsetneq [1:p]$, are equal. We can see on unreported numerical experiments that this choice of $N_u$ gives much  better results than if we choose the same value of $N_u$ for all $\emptyset \varsubsetneq u \varsubsetneq [1:p]$.
However, it seems difficult to obtain theoretical results on the values of $\V(\widehat{W}_{u}^{(1)})$, as they depend on the conditional distributions of $\bX$ in a complicated way.

Hence, this assumption is more a convenient heuristic to compute the best accuracies $(N_u^*)_{\emptyset \varsubsetneq u \varsubsetneq [1:p]}$ than a real property satisfied in many cases. Proposition \ref{prop_min_variance} and the heuristic in Equation \eqref{eq_heuri} justify the choice of $(N_u^*)_{\emptyset \varsubsetneq u \varsubsetneq [1:p]}$ given in Equation \eqref{eq_N_u^*}, and we make this choice even if the assumption of equal values of the $(\V(\widehat{W}_u^{(1)}))_{\emptyset \varsubsetneq u \varsubsetneq [1:p]}$ is not satisfied.
\end{rmk}

\subsubsection{Consistency}
A straightforward consequence of the subset $W$-aggregation procedure and Equation \eqref{ShEX} is that the consistency of $(\widehat{W}_u)_{u\subset [1:p]}$ implies the consistency of $(\widehat{\eta}_i)_{i\in [1:p]}$ (Assumption \ref{hypNu} is not necessary).

\begin{prop}\label{prop_consistance_shapley1}
Assume that for all $\emptyset \varsubsetneq u \varsubsetneq [1:p]$, we have estimators $\widehat{W}_u$ that converge to $W_u$ in probability (resp. almost surely) when $N_u$ goes to $+\infty$, where $\kappa N_u$ is the cost of $\widehat{W}_u$. If we use the subset $W$-aggregation procedure with the choice of $(N_u^{*})_{\emptyset \varsubsetneq u \varsubsetneq [1:p]}$ given by Equation \eqref{eq_N_u^*}, the estimators of the Shapley effects converge to the Shapley effects in probability (resp. almost surely) when $N_{tot}$ goes to $+\infty$ (where $N_{tot}$ is the total cost of the subset $W$-aggregation procedure).
\end{prop}

\subsection{The random-permutation procedure}\label{section_random_perm}
In this section, we present and study the "random-permutation $W$-aggregation procedure" suggested in \cite{song_shapley_2016}.

\subsubsection{The \texorpdfstring{$W$}{Lg}-aggregation procedure}
The $W$-aggregation procedure of the algorithm of \cite{song_shapley_2016} is based on Equation \eqref{eq_prop_perm}. Because of the equation, one could estimate $\eta_i$ by
\begin{equation}\label{songa}
\widehat{\eta}_i=\frac{1}{p!\V(Y)}\sum_{\sigma \in \mathcal{S}_p}\left(\widehat{W}_{P_{i}(\sigma)\cup\{i\}}-\widehat{W}_{P_i(\sigma)}\right),
\end{equation}
for $i\in [1:p]$. In Equation \eqref{songa}, informally, $(\widehat{W}_u)_{\emptyset \varsubsetneq u \varsubsetneq [1:p]}$ are estimators.
However, as the number of permutations is $p!$, there are too many summands and \cite{song_shapley_2016} suggests to replace the sum over all the $p!$ permutations by the sum over $M$ ($M<p!$) random uniformly distributed permutations. Thus, for a fixed $i\in [1:p]$, the estimator of $\eta_i$ suggested in \cite{song_shapley_2016} is
\begin{equation}\label{songb}
\widehat{\eta}_i=\frac{1}{M\V(Y)}\sum_{m=1}^M\left(\widehat{W}_{P_{i}(\sigma_m)\cup\{i\}}(m)- \widehat{W}_{P_i(\sigma_m)}(m)\right),
\end{equation}
where $(\sigma_m)_{m\in [1:M]}$ are independent and uniformly distributed on $\mathcal{S}_p$. If $m ,m'\in [1:M]$ with $m\neq m'$ and $P_{i}(\sigma_m)=P_{i}(\sigma_{m'})=:u$, \cite{song_shapley_2016} estimates twice the same $W_{u}$. To formalize these different estimations, we write $\widehat{W}_u(m)$ the estimation of $W_u$ at step $m$ in Equation \eqref{songb}.

Finally, \cite{song_shapley_2016} reduces the computation cost using the following idea. The authors of \cite{song_shapley_2016} notice that for $1 \leq i<p$, for any permutation $\sigma \in \mathcal{S}_p$ and for $i\in [1:p]$, we have $P_{\sigma(i+1)}(\sigma)=P_{\sigma(i)}(\sigma)\cup \{\sigma(i)\}$. Thus, the algorithm of \cite{song_shapley_2016} uses every estimate $\widehat{W}_{P_{\sigma_m(i)}(\sigma_m)\cup \{ \sigma_m(i)\} }(m)$ for $\widehat{\eta}_{\sigma_m(i)}$ (as an estimator of $W_{P_{\sigma_m(i)}(\sigma_m)\cup \{ \sigma_m(i)\} }$) and for $\widehat{\eta}_{\sigma_m(i+1)}$ (as an estimator of $W_{P_{\sigma_m(i+1)}(\sigma_m)}$). With this improvement, the number of estimations of $W_u$ (for $\emptyset \varsubsetneq u \varsubsetneq [1:p]$) is divided by two when estimating all the Shapley effects $\eta_1,...,\eta_p$. The $W$-aggregation procedure is then

\begin{enumerate}
\item Let $\widehat{\eta}_1=...= \widehat{\eta}_p=0$.
\item For all $m=1,2,...,M$
\begin{enumerate}
\item Generate $\sigma_m$ uniformly distributed on $\mathcal{S}_p$.\label{step_sigmam}
\item Let $prevC=0$.
\item For all $i=1,2,...,p$
\begin{enumerate}
\item Let $u=P_{\sigma_m(i)}(\sigma_m)$.
\item Compute $\widehat{W}_{u\cup\{\sigma_m(i)\}}(m)$.
\item Compute $\widehat{\Delta}=\widehat{W}_{u\cup\{\sigma_m(i)\}}(m)-prevC$.
\item Update $\widehat{\eta}_{\sigma_m(i)}=\widehat{\eta}_{\sigma_m(i)}+\widehat{\Delta}$.
\item Set $prevC=\widehat{W}_{P_{\sigma_m(i+1)}(\sigma_m)}$.
\end{enumerate}
\end{enumerate}
\item Let $\widehat{\eta}_i=\widehat{\eta}_i\slash(\V(Y) M)$ for all $i=1,...,p$.
\end{enumerate}

We write this $W$-aggregation procedure "random-permutation $W$-aggregation procedure".

\begin{rmk}
Recall that in the subset $W$-aggregation procedure, each estimation of $W_u$ was used for the estimation of all the $(\eta_i)_{i\in [1:p]}$ (and not only for two of them). Thus the subset $W$-aggregation procedure seems to be more efficient.
\end{rmk}

\begin{rmk}\label{rmk_exact_permutation_procedure}
When the number of inputs $p$ is small, \cite{song_shapley_2016} suggests to take all the permutations of $[1:p]$ instead of choosing random permutations in Step \ref{step_sigmam} of the random-permutation $W$-aggregation procedure. However, this algorithm requires small values of $p$ and the total cost is a multiple of $p!$ (so there are very restricted possible values). Furthermore, this method still remains very costly due to the computation of $(p-1)!$ conditional variances. For example, in the linear Gaussian framework with $p=10$ (where the computation of the conditional elements is immediate, see \cite{broto_sensitivity_2019}) it spends more than ten minutes computing the Shapley effects.  Hence, the algorithm with all the permutations is not explicitly detailed in \cite{song_shapley_2016}.
\end{rmk}

\subsubsection{Choice of the accuracy of each \texorpdfstring{$\widehat{W}_u$}{Lg}}\label{section_adaptation}

As in Section \ref{section_variance}, we suggest a choice of the accuracies $(N_u)_{\emptyset \varsubsetneq u \varsubsetneq [1:p]}$.

In order to avoid a random total cost, we require for all $\emptyset \varsubsetneq u \varsubsetneq [1:p]$ that the accuracy $N_u$ of the $\left(\widehat{W}_{u}(m)\right)_m$ depends only on $|u|$, and we write $N_{|u|}:=N_u$. In this case, the total cost of the random-permutation $W$-aggregation procedure is equal to $N_{tot}=\kappa M\sum_{k=1}^{p-1}N_k$. Moreover, we assume that the total cost $N_{tot}=\kappa M\sum_{k=1}^{p-1}N_k$ is proportional to $(p-1)$, and thus can be written $N_{tot}=\kappa M N_O(p-1)$ for some fixed $N_O\in \N^*$.
As the permutations $(\sigma_m)_{m\in [1:M]}$  are random, we choose to minimize $\E\left[ \sum_{i=1}^p \V\left( \left. \widehat{\eta}_i \right| (\sigma_m)_{m\in [1:M]} \right)\right] $. 

To compute the optimal values of $(N_u)_{\emptyset \varsubsetneq u \varsubsetneq [1:p]}$, we introduce the following assumption.
\begin{assumption}\label{assum_cost_random}
For all $\emptyset \varsubsetneq u \varsubsetneq [1:p]$ and all $m\in [1:M]$,  $\widehat{W}_{u}(m)$ is computed with a cost $\kappa N_{|u|}$ by $
\widehat{W}_{u}(m)=\frac{1}{N_{|u|}}\sum_{n=1}^{N_{|u|}}\widehat{W}_{u}^{(n)}(m)
$ where the $(\widehat{W}_{u}^{(n)}(m))_{n\in [1:N_u]}$ are independent and identically distributed. The $(\widehat{W}_{u}(m))_{\emptyset \varsubsetneq u \varsubsetneq [1:p],\; m \in [1:M] }$ are independent.
\end{assumption}
When it is possible to sample from the conditional distributions of the input vector, we can generate i.i.d. double Monte-Carlo estimators $(\widehat{E}_{u,MC}(m))_{m\in [1:M]}$ or Pick-and-Freeze estimators $(\widehat{V}_{u,PF}(m))_{m\in [1:M]}$. Hence, they satisfy Assumption \ref{assum_cost_random} by taking $N_u=N_{|u|}$ for all $\emptyset\varsubsetneq u \varsubsetneq [1:p]$.

\begin{prop}\label{prop_min_pu}
Assume that we estimate the Shapley effects with the random-permutation $W$-aggregation procedure under Assumption \ref{hypNu} and that the variances $(\V(\widehat{W}_u^{(1)}(1)))_{\emptyset \varsubsetneq u \varsubsetneq [1:p]}$ are equal. Then, the solution of the problem $$
\min_{(N_k)_{k\in [1:p-1]}\in (0,+\infty)^{p-1}}\E\left[ \sum_{i=1}^p \V\left( \left. \widehat{\eta}_i \right| (\sigma_m)_{m\in [1:M]} \right)\right] \;\;\; \text{ subject to } \;\; \kappa M\sum_{k=1}^{p-1}N_k=\kappa MN_O(p-1)
$$
is $(N_k^{**})_{k\in [1:p-1]}$ with for all $k \in [1:p-1]$,
 $$
 N_k^{**}=N_O.
 $$
\end{prop}

Hence, from now on, with the random permutation $W$-aggregation procedure, we will choose the accuracy $N_u=N_O$ for all subset $u$.

\begin{rmk}
As in Remark \ref{rmk_Var=}, we assume in Proposition \ref{prop_min_pu} that the variances $(\V(\widehat{W}_u^{(1)}(1)))_{\emptyset \varsubsetneq u \varsubsetneq [1:p]}$ are equal. This assumption is not easy to check, but is required technically to prove Proposition \ref{prop_min_pu}.
\end{rmk}

\begin{rmk}
With the exact-permutation $W$-aggregation procedure (see Remark \ref{rmk_exact_permutation_procedure}), $N_k^*=N_O p!$ is the solution of the problem $\sum_{i=1}^p \V\left( \widehat{\eta}_i  \right)$ subject to $\sum_{k=1}^{p-1}N_k=p!N_O(p-1)$.
\end{rmk}

There are now two parameters to choose: the number of permutations $M$ and the accuracy $N_O$ of the estimations of the $(W_u)_{\emptyset \varsubsetneq u \varsubsetneq [1:p]}$. Typically, their product $MN_O$ is imposed by budget constraints. 

\subsubsection{Choice of \texorpdfstring{$N_O$}{Lg}}\label{section_choix_param}
We have seen that for all $\emptyset \varsubsetneq u \varsubsetneq [1:p]$, we choose $N_u=N_{|u|}^{**}=N_O$. In this section, we explain why we should choose $N_O=1$ under Assumption \ref{hypNu} and $M$ as large as possible.

Proposition \ref{prop_choix_param} generalizes the result given in \cite{song_shapley_2016}, Appendix B. Its proof is given in the supplementary material, which is much simpler than the arguments in \cite{song_shapley_2016}.

\begin{assumption}\label{assum_sansbiais}
Assumption \ref{assum_cost_random} holds and for all $\emptyset \varsubsetneq u \varsubsetneq [1:p]$, we have $\E(\widehat{W}_u^{(1)}(1))=W_u$.
\end{assumption}
Assumption \ref{assum_sansbiais} ensures that the estimators have a zero bias. Recall that the double Monte-Carlo estimator and the Pick-and-Freeze estimator have a zero bias. Hence, they satisfy Assumption \ref{assum_sansbiais} by generating i.i.d. $(\widehat{W}_u(m))_{m\in [1:M]}$ and by taking $N_u=N_{|u|}$ for all $\emptyset\varsubsetneq u \varsubsetneq [1:p]$.

\begin{prop}\label{prop_choix_param} Let $i\in [1:p]$ be fixed.
Under Assumption \ref{assum_sansbiais}, in order to minimize, over $N_O$ and $M$, the variance of $\widehat{\eta}_i$ with a fixed cost $\kappa MN_O\times (p-1)=\kappa C\times (p-1)$ (for some $C\in \N^*$), we have to choose $N_O=1$ and $M=C$.
\end{prop}

From now on, we assume that $N_u=N_O=1$ when we use the random-permutation $W$-aggregation procedure and we will let $M$, the number of random permutations, go to infinity. Then, the total cost $N_{tot}$ of the random-permutation $W$-aggregation procedure is equal to $N_{tot}=\kappa M(p-1)$, for estimating the $p$ Shapley effects $\eta_1,\ldots,\eta_p$. Hence, under Assumption \ref{assum_cost_random} or Assumption \ref{assum_sansbiais}, $\widehat{W}_u(m)=\widehat{W}_u^{(1)}(m)$ and has now a cost $\kappa$.

\subsubsection{Consistency}

We give here two sufficient conditions for the consistency of the estimators of the Shapley effects given by the random-permutation $W$-aggregation procedure. We introduce a general assumption.

\begin{assumption}\label{assum_consis_perm}
For all $u$ such that $\emptyset \varsubsetneq u \varsubsetneq[1:p]$, $\left( \widehat{W}_u(m) \right)_{m\in [1:M]}$ have a cost $\kappa$ (since we chose $N_u=1$) and are identically distributed with a distribution that depends on an integer $N$ such that $
\E\left( \widehat{W}_u(1) \right) \underset{N\rightarrow +\infty}{\longrightarrow} W_u.$
Moreover, for all $u$ such that $\emptyset \varsubsetneq u \varsubsetneq[1:p]$, we have
$$
\frac{1}{M^2} \sum_{m,m'=1}^{M} cov\left( \widehat{W}_{u}(m),\widehat{W}_{u}(m') \right) \underset{N,M\rightarrow+\infty}{\longrightarrow}0.
$$
\end{assumption}
Assumption \ref{assum_consis_perm} is more general than Assumption \ref{assum_sansbiais}. Indeed, it enables the estimators to have a bias and a covariance which go to zero. This assumption will be useful to prove the consistency results in Section \ref{section_consis_shap_echan}. Remark that in Assumption \ref{assum_consis_perm}, for all $\emptyset \varsubsetneq u \varsubsetneq[1:p]$, each estimate $\left( \widehat{W}_u(m) \right)_{m\in [1:M]}$ has a cost $\kappa$, as in Assumption \ref{assum_sansbiais} since we fixed $N_u=N_O=1$.

\begin{prop}\label{prop_consistance_random-permutation $W$-aggregation procedure}
Assume that we estimate the Shapley effects using the random-permutation $W$-aggregation procedure. Let $N_{tot}=\kappa M(p-1)$ be the total cost of the random-permutation $W$-aggregation procedure.
\begin{enumerate}
\item Under Assumption \ref{assum_sansbiais}, the estimates of the Shapley effects converge to the Shapley effects in probability when $N_{tot}$ goes to $+\infty$.
\item Under Assumption \ref{assum_consis_perm}, the estimates of the Shapley effects converge to the Shapley effects in probability when $N_{tot}$ and $N$ go to $+\infty$.
\end{enumerate}

\end{prop}

\section{Estimators of the Shapley effects}\label{section_loi_connue}

\subsection{Four consistent estimators of the Shapley effects}

Recall that in Section \ref{section_Wu}, we have seen two estimators of the $(W_u)_{\emptyset \varsubsetneq u \varsubsetneq [1:p]}$: double Monte-Carlo (used in the algorithm of \cite{song_shapley_2016}) and Pick-and-Freeze. In Section \ref{section_structure}, we have studied two $W$-aggregation procedures for the Shapley effects using estimators of the $(W_u)_{\emptyset \varsubsetneq u \varsubsetneq [1:p]}$: the subset $W$-aggregation procedure and the random-permutation $W$-aggregation procedure (used in the algorithm of \cite{song_shapley_2016}).  To sum up, four estimators of the Shapley effects are available:
\begin{itemize}
\item subset $W$-aggregation procedure with double Monte-Carlo;
\item subset $W$-aggregation procedure with Pick-and-Freeze;
\item random-permutation $W$-aggregation procedure with double Monte-Carlo, which is the already existing algorithm of \cite{song_shapley_2016};
\item random-permutation $W$-aggregation procedure with Pick-and-Freeze.
\end{itemize}
With the random-permutation $W$-aggregation procedure, we have seen that we need different estimators $(\widehat{W}_u(m))_{m\in [1:M]}$ of the same $W_u$. In this case, we choose i.i.d. realizations of the estimator of $W_u$. Moreover, we have seen in Section \ref{section_choix_param} that when we use the random-permutation $W$-aggregation procedure, we choose $N_u=N_O=1$.

By Propositions \ref{prop_consistance_shapley1} and \ref{prop_consistance_random-permutation $W$-aggregation procedure}, all these four estimators are consistent when the global budget $N_{tot}$ goes to $+\infty$. Indeed, by Proposition \ref{prop_consistance_shapley1}, the consistency of the $(\widehat{W}_u)_{\emptyset \varsubsetneq u \varsubsetneq [1:p]}$ is sufficient for the consistency with the subset procedure and by Proposition \ref{prop_consistance_random-permutation $W$-aggregation procedure}, unbiased and i.i.d. estimators $(\widehat{W}_u(m))_{m\in [1:M]}$ for all $\emptyset \varsubsetneq u \varsubsetneq [1:p]$ provide the consistency with the random-permutation procedure.

\subsection{Numerical comparison of the different algorithms}\label{section_num_connue}
In this section, we carry out numerical experiments on the different algorithms in the restricted framework (where the exact conditional samples are available). 

To compare these estimators, we use the linear Gaussian framework: $\mathcal{X}=\R^p$, $\bX\sim \mathcal{N}(\boldsymbol\mu,\boldsymbol\Gamma)$ and $Y=\sum_{i=1}^p \beta_i X_i$. In this case, the theoretical values are easily computable (see \cite{owen_shapley_2017,iooss_shapley_2017,broto_sensitivity_2019}). We choose $p=10$, $\beta_i=1$ for all $i\in [1:p]$ and $\boldsymbol\Gamma=\mathbf{A}^T\mathbf{A}$ where $\mathbf{A}$ is a $p\times p$ matrix which components are realisations of $p^2$ i.i.d. Gaussian variables with zero mean and unit variance. To compare these different estimators, we fix a total cost (number of evaluations of $f$) of $N_{tot}=54000$. We compute 1000 realizations of each estimator.

\begin{figure}[t!]
\center
\includegraphics[width=12cm,height=10cm]{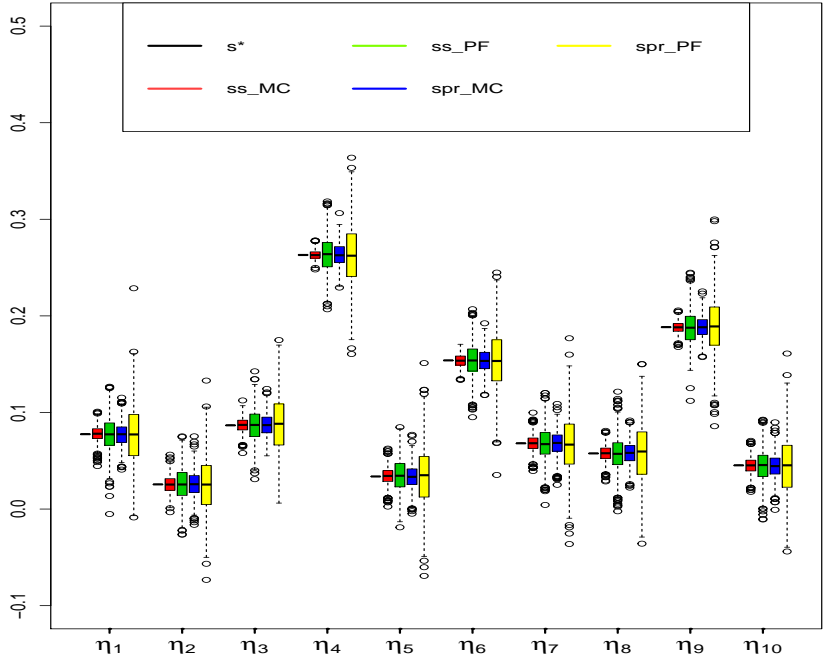}
\caption{Estimation of the Shapley effects in the linear Gaussian framework. In black (s*) we show the theoretical values, in red (ss$\_$MC) the estimates from the subset $W$-aggregation procedure with the double Monte-Carlo estimator, in green (ss$\_$PF) the estimates from the subset $W$-aggregation procedure with the Pick-and-Freeze estimator, in blue (spr$\_$MC) the estimates from the random-permutation $W$-aggregation procedure with the double Monte-Carlo estimator and in yellow (spr$\_$PF) the estimates from the random-permutation $W$-aggregation procedure with the Pick-and-Freeze estimator.}
\label{boxplot_connue}
\end{figure}

In Figure \ref{boxplot_connue}, we plot the theoretical values of the Shapley effects together with the boxplots of the 1000 realizations of each estimator.
\begin{figure}[t!]
\center
\includegraphics[width=7cm,height=7cm]{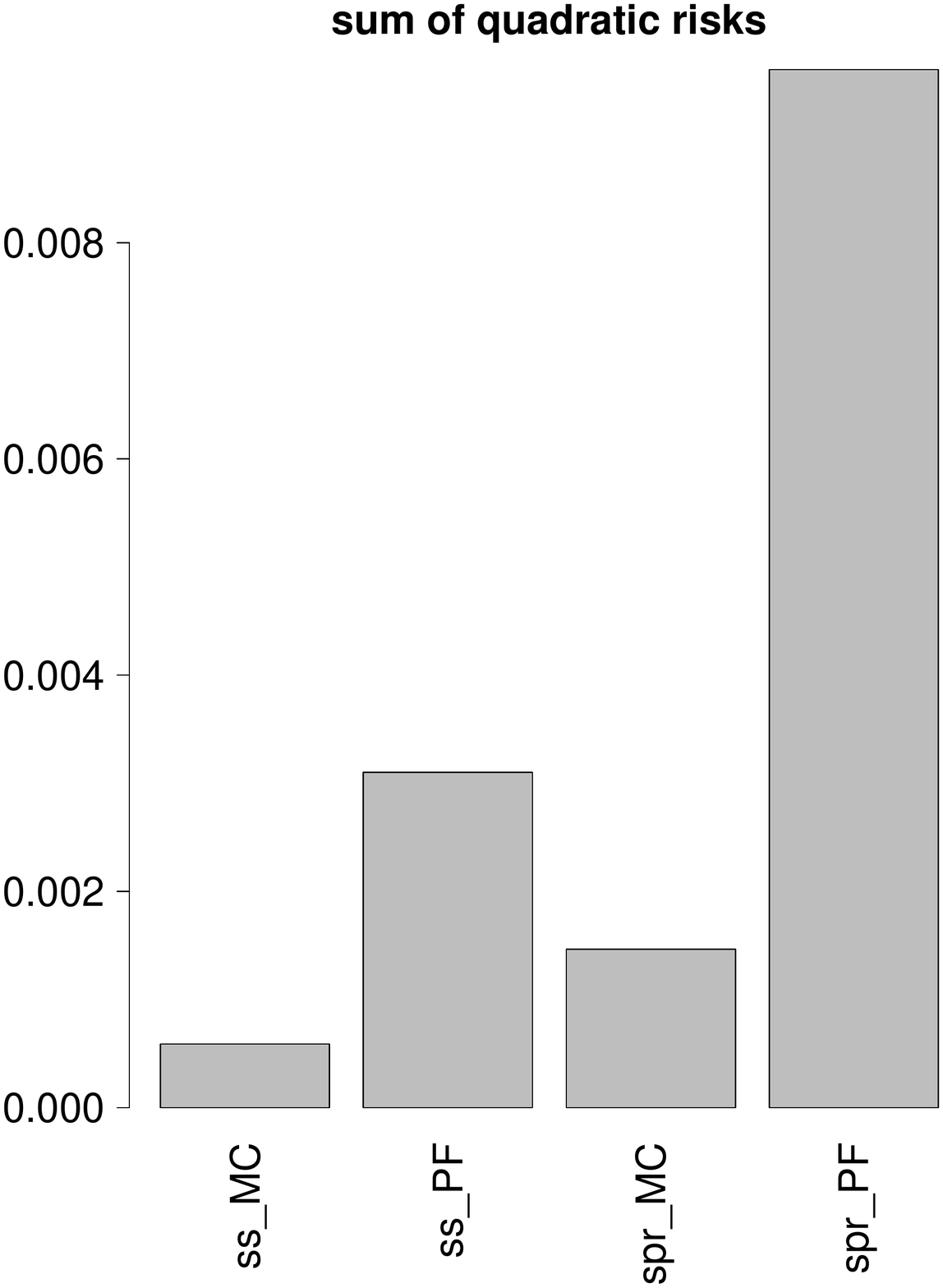}
\caption{Sum over $i\in [1:p]$ of the estimated quadratic risks of the four estimators of the Shapley effects in the linear Gaussian framework.}
\label{barplot_connue}
\end{figure}

In Figure \ref{barplot_connue}, we plot the sum over $i\in [1:p]$ of the quadratic risks: $\sum_{i=1}^p \E \left( (\widehat{\eta}_i-\eta_i)^2 \right)$ (estimated with 1000 realizations) of each estimator. 

We can see that the subset $W$-aggregation procedure gives better results than the random-permutation $W$-aggregation procedure, and that the double Monte-Carlo estimator is better than the Pick-and-Freeze estimator.

\begin{rmk}
It appears that double Monte-Carlo is numerically more efficient than Pick-and-Freeze for estimating the Shapley effects.
Indeed, if we focus only on the estimation of one $W_u$ for a fixed $\emptyset \varsubsetneq u \varsubsetneq [1:p]$, we can see numerically that the Pick-and-Freeze estimator has a larger variance than the double Monte-Carlo estimator. This finding appears to be difficult to confirm theoretically in the general case. Nevertheless, we can obtain such a theoretical confirmation in a simple, specific example. Let
$\bX\sim \mathcal{N}(0,\mathbf{I}_2)$, $Y=X_1+X_2$. Remark that, in this example, the variances of $\widehat{W}_u^{(1)}$, $\emptyset \varsubsetneq u \varsubsetneq [1:p]$, are equal.
In this case, and for $u=\{1\}$, we can easily get $\V(\widehat{V}_{u,PF})=\frac{40}{9} \V(\widehat{E}_{u,MC})$ for the same cost (number of evaluations of $f$), and choosing $N_I=3$ for the double Monte-Carlo estimator. This could be surprising since \cite{janon_asymptotic_2014} proved that some Pick-and-Freeze estimator is asymptotically efficient in the independent case. However, this result and our finding are not contradictory for two reasons: the authors of \cite{janon_asymptotic_2014} estimate the variance of $Y$ so their result does not apply here and the double Monte-Carlo estimator is based on different observations from the Pick-and-Freeze estimator. 
\end{rmk}

To conclude, we improved the already existing algorithm of \cite{song_shapley_2016} (random-permutation $W$-aggregation procedure with double Monte-Carlo) by the estimator given by the subset $W$-aggregation procedure with double Monte-Carlo.

\section{Extension when we observe an i.i.d. sample}\label{section_loi_inconnue}
In Section \ref{section_loi_connue}, we have considered a restricted framework: we assumed that for all $\emptyset \varsubsetneq u \varsubsetneq [1:p]$ and all $\bx_u\in \mathcal{X}_u$, we could generate an i.i.d. sample from the distribution of $\bX_{-u}$ conditionally to $\bX_u=\bx_u$. However, in many cases, we can not generate this sample, as we only observe an i.i.d. sample of $\bX$. In this section, we assume that we only observe an i.i.d. sample $(\bX^{(n)})_{n \in [1:N]}$ of $\bX$ and that we have access to the computer code $f$. We extend the double Monte-Carlo and Pick-and-Freeze estimators in this general case and show their consistency and rates of convergence. We then give the consistency of the implied estimators of the Shapley effects (obtained from the $W$-aggregation procedures studied previously). To the best of our knowledge, these suggested estimators are the first estimators of Shapley effects in this general framework. We conclude giving numerical experiments.

We choose a very general framework to prove the consistency of the estimators. This framework is given in the following assumption.
\begin{assumption}\label{assum_polish}
For all $i\in [1:p]$, $(\cX_i,d_i)$ is a Polish space with metric $d_i$ and $\bX=(X_1,...,X_p)$ has a density $f_\bX$ with respect to a finite measure $\mu=\bigotimes_{i=1}^p \mu_i$ which is bounded and $\PP_X$-almost everywhere continuous.
\end{assumption}
This assumption is really general. Actually, it enables to have some continuous variables (with the Euclidean distance), some categorical variables in countable ordered or unordered sets and some variables in separable Hilbert spaces (for example $L^2(\R^d)$, for some $d \in \N^*$). The fact that $\bX$ has a continuous density $f_\bX$ with respect to a finite measure $\mu=\bigotimes \mu_i$ means that the distribution of $\bX$ is smooth. Assumption \ref{assum_polish} is satisfied in many realistic cases. The assumption of a bounded density which is $\PP_\bX$-almost everywhere continuous may be less realistic in some cases but is needed in the proofs. It would be interesting to alleviate it in future work.

To prove rates of convergence, we will need the following stronger assumption.

\begin{assumption}\label{assum_rate}
The function $f$ is $\mathcal{C}^1$, $\cX$ is compact in $\R^p$, $\bX$ has a density $f_\bX$ with respect to the Lebesgue measure $\lambda_p$ on $\mathcal{X}$ such that $\lambda_p$-a.s. on $\mathcal{X}$, we have $0<C_{\inf}\leq f_\bX \leq C_{\sup}<+\infty$. Furthermore, $f_\bX$ is Lipschitz continuous on $\mathcal{X}$.
\end{assumption}

Assumption \ref{assum_rate} is more restrictive than Assumption \ref{assum_polish}. It requires all the input variables to be continuous and real-valued. Moreover, their values are restricted to a compact set where the density is lower-bounded. Assumption \ref{assum_rate} will be satisfied in some realistic cases (for instance with uniform or truncated Gaussian input random variables). Nevertheless, there also exist realistic cases where the input density is not lower-bounded (for instance with triangular input random variables). We remark that the assumption of a lower-bounded density is common in the field of non-parametric statistics \cite{ghosal_convergence_2001}. Here, it enables us to control the order of magnitude of conditional densities.

\subsection{Estimators of the conditional elements}\label{section_Wu_inconnue}
As far as we know, only \cite{veiga_efficient_2013} suggests a consistent estimator of $W_u$ when we only observe an i.i.d. sample and when the input variables can be dependent, but only for $V_u$ with $|u|=1$.  The estimator suggested in \cite{veiga_efficient_2013} is asymptotically efficient but the fact that $u$ has to be a singleton prevents us to use this estimator for the Shapley effects (because we have to estimate $W_u$ for all $\emptyset \varsubsetneq u \varsubsetneq [1:p]$). We can find another estimator of the $(V_u)_{u\subset [1:p]}$ in \cite{plischke_effective_2010} (but no theoretical results on the convergence are given). Finally, note that \cite{plischke_global_2013} provides an estimator of different sensitivity indices, with convergence proofs.

In this section we introduce two consistent estimators of $(W_u)_{\emptyset \varsubsetneq u \varsubsetneq [1:p]}$ when we observe only an i.i.d. sample of $\bX$, and which are easy to implement. These two estimators follow the principle of the double Monte-Carlo and Pick-and-Freeze estimators, but replacing exact samples from the conditional distributions by approximate ones based on nearest-neighbours methods.

To that end, we have to introduce notation. Let $N\in \N$ and $(\bX^{(n)})_{n\in [1:N]}$ be an i.i.d. sample of $\bX$. If $\emptyset \varsubsetneq v\subsetneq [1:p]$, let us write $k_N^{v}(l,n)$ for the index such that $\bX_{v}^{(k_N^{v}(l,n))}$ is the (or one of the) $n$-th closest element to $\bX_{v}^{(l)}$ in $(\bX_{v}^{(i)})_{i\in [1:N]}$, and such that $(k^{v}_N(l,n))_{n \in [1:N_I]}$ are two by two distinct.

The index $k_N^{v}(l,n)$ could be not uniquely defined if there exist different observations $\bX_{v}^{(i)}$ at equal distance from $\bX_{v}^{(l)}$. In this case, we will choose $k_N^{v}(l,n)$ uniformly over the indices of these observations, with the following independence assumption.
\begin{assumption}\label{assum_indep}
Conditionally to $(\bX_{v}^{(n)})_{n \in [1:N]}$, $k^{v}_N(l,i)$ is randomly and uniformly chosen over the indices of all the $i$-th nearest neighbours of $\bX_{v}^{(l)}$ in $(\bX_{v}^{(n)})_{n\in [1:N]}$ and the $(k^{v}_N(l,i))_{i [1:N_I]}$ are two by two distinct. Furthermore, conditionally to $(\bX_{v}{(n)})_{n\in [1 : N]}$, for all $l \in [1 :N]$, the random vector $(k_N(l,i))_{i\in [1 :N_I]}$ is independent on all the other random variables.
\end{assumption}
To summarize the idea of Assumption \ref{assum_indep}, we can say that the nearest neighbours of $\bX_v^{(l)}$ are chosen uniformly among the possible choices and independently on the other variables. Assumption \ref{assum_indep} actually only formalizes the random choice of the nearest neighbours where there can be equalities of the distances and this choice is easy to implement in practice.

When $\bX_{v}$ is absolutely continuous with respect to the Lebesgue measure, distance equalities can not happen and $k_N^{v}(l,n)$ is uniquely defined. Thus, Assumption \ref{assum_indep} trivially holds in this case. Assumption \ref{assum_indep} is thus specific to the case where some input variables are not continuous.

\subsubsection{Double Monte-Carlo}\label{section_MC_inconnue}

We write $(s(l))_{l \in [1:N_u]}$ a sample of uniformly distributed integers in $[1:N]$ (with or without replacement) independent of the other random variables. Then, we define two slightly different versions of the double Monte-Carlo estimator by
\begin{equation}\label{eq_VuMC_inconnue1}
\widehat{E}_{u,MC}^{mix}=\frac{1}{N_u}\sum_{l=1}^{N_u}\widehat{E}_{u,s(l),MC}^{mix},
\end{equation}
and
\begin{equation}\label{eq_VuMC_inconnue2}
\widehat{E}_{u,MC}^{knn}=\frac{1}{N_u}\sum_{l=1}^{N_u}\widehat{E}_{u,s(l),MC}^{knn},
\end{equation}
with
\begin{equation}\label{eq_Vui_MCmix}
\widehat{E}_{u,s(l),MC}^{mix}=\frac{1}{N_I-1}\sum_{i=1}^{N_I}\left[ f\left(\bX_{-u}^{(s(l))},\bX_{u}^{(k_N^{-u}(s(l),i))}\right)-\frac{1}{N_I}\sum_{h=1}^{N_I} f\left(\bX_{-u}^{(s(l))},\bX_{u}^{(k_N^{-u}(s(l),h))}\right) \right]^2
\end{equation}
and
\begin{equation}
\label{eq_Vui_MCknn}
\widehat{E}_{u,s(l),MC}^{knn}=\frac{1}{N_I-1}\sum_{i=1}^{N_I}\left[ f\left(\bX^{(k_N^{-u}(s(l),i))}\right)-\frac{1}{N_I}\sum_{h=1}^{N_I} f\left(\bX^{(k_N^{-u}(s(l),h))}\right) \right]^2.
\end{equation}

The double Monte-Carlo estimator has two sums: one of size $N_I$ for the conditional variance, one other of size $N_u$ for the expectation. The integer $N_I$ is also the number of nearest neighbours and it is a fixed parameter to choose. For example, we can choose $N_I=3$ (as in the case where the conditional samples are available).

\begin{rmk}
If we observe the sample $(\bX^{(n)})_{n\in [1:N]}$ and if the values of $(f(\bX^{(n)}))_{n\in [1:N]}$ have to be assessed, the cost of the estimators $\widehat{E}_{u,MC}^{mix}$ and $\widehat{E}_{u,MC}^{knn}$ remains the number of evaluations of $f$ (which is $N_I N_u$). If we observe the sample $(\bX^{(n)},f(\bX^{(n)}))_{n\in [1:N]}$, the estimator $\widehat{E}_{u,MC}^{knn}$ does not require evaluations of $f$ but the cost remains proportional to $N_u$ (for the search of the nearest neighbours and for the elementary operations).
\end{rmk}

\begin{rmk}
The integer $N$ is the size of the sample of $\bX$ (that enables us to estimate implicitly its conditional distributions through the nearest neighbours) and the integer $N_u$ is the accuracy of the estimator $\widehat{E}_{u,MC}$ from the estimated distribution of $\bX$. Of course, it would be intuitive to take $N_u=N$ and $(s(l))_{l\in [1:N]}=(l)_{l\in [1:N]}$, but this framework would not be general enough for the subset $W$-aggregation procedure (in which the accuracy $N_u$ of $\widehat{E}_{u,MC}$ depends on $u$) and for the proof of the consistency when using the random-permutation $W$-aggregation procedure in Section \ref{section_consis_shap_echan}. Furthermore, we may typically have to take $N_u$ smaller than $N$.
\end{rmk}

Remark that we give two versions of the double Monte-Carlo estimator. The "mix" version seems more accurate but requires to call the computer code of $f$ at new inputs. For the "knn" version, it is sufficient to have an i.i.d. sample $(\bX^{(n)},f(\bX^{(n)}))_{n\in [1:N]}$.

Now that we defined these two versions of the double Monte-Carlo estimator for an unknown input distribution, we give the consistency of these estimators in Theorem \ref{thrm_consis_MC}. We let $\widehat{E}_{u,MC}$ be given by Equation \eqref{eq_VuMC_inconnue1} or Equation \eqref{eq_VuMC_inconnue2}. In the asymptotic results below, $N_I$ is fixed and $N$ and $N_u$ go to infinity.

\begin{theorem}\label{thrm_consis_MC}
Assume that Assumption \ref{assum_polish} holds and Assumption \ref{assum_indep} holds for $v=-u$. If $f$ is bounded, then $\widehat{E}_{u,MC}$ converges to $E_u$ in probability when $N$ and $N_u$ go to $+\infty$.
\end{theorem}

Furthermore, with additional regularity assumptions, we can give the rate of convergence of these estimators in Theorem \ref{thrm_vitesse_MC} and Corollary\ref{corol_MC}.
 
\begin{theorem}\label{thrm_vitesse_MC}
 Under Assumption \ref{assum_rate}, for all $\varepsilon>0$, $\varepsilon'>0$, there exist fixed constants $C_{\sup}^{(1)}(\varepsilon')$ and $C_{\sup}^{(2)}$ such that
\begin{equation}
\PP\left( \left| \widehat{E}_{u,MC}-E_{u} \right| >\varepsilon \right) \leq \frac{1}{\varepsilon^2}\left( \frac{C_{\sup}^{(1)}(\varepsilon')}{N^{\frac{1}{p-|u|}-\varepsilon'}}+\frac{C_{\sup}^{(2)}}{N_u} \right) .
\end{equation}
\end{theorem}

\begin{corol}\label{corol_MC}
Under Assumption \ref{assum_rate}, choosing $N_u\geq C N^{1\slash(p- |u|)}$ for some fixed $0<C < +\infty$, we have for all $\delta>0$, 
$$
\left| \widehat{E}_{u,MC}-E_{u} \right|=o_p\left( \frac{1}{N^{\frac{1}{2(p-|u|)}-\delta}}\right).
$$
\end{corol}
We remark that for $|u|=p-1$, we nearly obtain a parametric rate of convergence $N^{\frac{1}{2}}$. The rate of convergence decreases when $|u|$ decreases which can be interpreted by the fact that we estimate non-parametrically the function $\bx_{-u}\mapsto\V(f(\bX)|\bX_{-u}=\bx_{-u})$. The estimation problem is higher-dimensional when $|u|$ decreases.

\subsubsection{Pick-and-Freeze}\label{section_PF_inconnue}
We now give similar results for the Pick-and-Freeze estimators. The number $N_I$ of nearest neighbours  that we need for the Pick-and-Freeze estimators is equal to $2$.
Assume that $\E(Y)$ is known and let $(s(l))_{l\in [1:N_u]}$ be as in Section \ref{section_MC_inconnue}. Then, we define two slightly different versions of the Pick-and-Freeze estimator by
\begin{equation}\label{eq_VuPF_inconnue1}
\widehat{V}_{u,PF}^{mix}=\frac{1}{N_u}\sum_{l=1}^{N_u}\widehat{V}_{u,s(l),PF}^{mix},
\end{equation}
and
\begin{equation}\label{eq_VuPF_inconnue2}
\widehat{V}_{u,PF}^{knn}=\frac{1}{N_u}\sum_{l=1}^{N_u}\widehat{V}_{u,s(l),PF}^{knn},
\end{equation}
with
\begin{equation}\label{eq_Vui_PFmix}
\widehat{V}_{u,s(l),PF}^{mix}=f\left((\bX^{(k_N^u(s(l),1))}\right)f\left(\bX_{u}^{(k_N^u(s(l),1))},\bX_{-u}^{(k_N^u(s(l),2))}\right)-\E(Y)^2
\end{equation}
and
\begin{equation}\label{eq_Vui_PFknn}
\widehat{V}_{u,s(l),PF}^{knn}= f(\bX^{(k_N^u(s(l),1))})f(\bX^{(k_N^u(s(l),2))})-\E(Y)^2.
\end{equation}

As for the double Monte-Carlo estimators, we give the consistency of the Pick-and-Freeze estimators in Theorem \ref{thrm_consis_PF} and the rate of convergence in Theorem \ref{thrm_vitesse_PF} and in Corollary\ref{corol_PF}. We let $\widehat{V}_{u,PF}$ be given by Equation \eqref{eq_VuPF_inconnue1} or Equation \eqref{eq_VuPF_inconnue2}.

\begin{theorem}\label{thrm_consis_PF}
Assume that Assumption \ref{assum_polish} holds and Assumption \ref{assum_indep} holds for $v=u$ and $N_I=2$. If $f$ is bounded, then $\widehat{V}_{u,PF}$ converges to $V_u$ in probability when $N$ and $N_u$ go to $+\infty$.
\end{theorem}

\begin{theorem}\label{thrm_vitesse_PF}
Under Assumption \ref{assum_rate}, if $|u|=1$, for all $\varepsilon>0$, $\varepsilon'>0$, 
\begin{equation}
\PP\left( \left| \widehat{V}_{u,PF}-V_{u} \right| >\varepsilon \right) \leq \frac{1}{\varepsilon^2}\left( \frac{C_{\sup}^{(1)}(\varepsilon')}{N^{1-\varepsilon'}}+\frac{C_{\sup}^{(2)}}{N_u} \right),
\end{equation}
and if $|u|>1$, for all $\varepsilon>0$, 
\begin{equation}
\PP\left( \left| \widehat{V}_{u,PF}-V_{u} \right| >\varepsilon \right) \leq \frac{C_{\sup}^{(3)}}{\varepsilon^2}\left( \frac{1}{N^{\frac{1}{|u|}}}+\frac{1}{N_u} \right),
\end{equation}
with fixed constants $C_{\sup}^{(1)}(\varepsilon')<+\infty,\;C_{\sup}^{(2)}<+\infty$ and $C_{\sup}^{(3)}<+\infty$.
\end{theorem}

\begin{corol}\label{corol_PF}
Under Assumption \ref{assum_rate}, choosing $N_u\geq C N^{1\slash |u|}$ for some fixed $0<C < +\infty$, we have
\begin{enumerate}
\item for all $u$ such that $|u|=1$, for all $\delta>0$, 
$$
\left| \widehat{V}_{u,PF}-V_{u} \right|=o_p\left( \frac{1}{N^{\frac{1}{2}-\delta}}\right).
$$
\item for all $u$ such that $|u|>1$, 
$$
\left| \widehat{V}_{u,PF}-V_{u} \right|=O_p\left( \frac{1}{N^{\frac{1}{2|u|}}}\right).
$$
\end{enumerate}
\end{corol}
The interpretation of the rates of convergence is the same as for the double Monte-Carlo estimators.
\subsection{Consistency of the Shapley effect estimators}\label{section_consis_shap_echan}
Now that we have constructed estimators of $W_u$ with an unknown input distribution, we can obtain estimators of the Shapley effects using the subset and random-permutation $W$-aggregation procedures. Note that for each $W$-aggregation procedure, we need to choose the accuracy $N_u$ of the $(\widehat{W}_u)_{\emptyset \varsubsetneq u \varsubsetneq [1:p]}$. Although Assumption \ref{hypNu} does not hold with the estimators $\widehat{E}_{u,MC}$ and $\widehat{V}_{u,PF}$ (the summands of these estimators are not independent), we keep choosing $N_u=N_O=1$ for the random-permutation $W$-aggregation procedure and $N_u$ as the closest integer to $N_{tot} \kappa^{-1}\begin{pmatrix}
p\\ |u|
\end{pmatrix}^{-1} (p-1)^{-1}$ with the subset $W$-aggregation procedure. To unify notation, let $N_I=2$ when the estimators of the conditional elements $(W_u)_{\emptyset \varsubsetneq u \varsubsetneq [1:p]}$ are the Pick-and-freeze estimators (in this way, $N_I$ is the number of nearest neighbours). With the double Monte-Carlo estimators, let $N_I$ be a fixed integer (for example $N_I=3$). 

Finally, recall that for the random-permutation $W$-aggregation procedure, we need different estimators $(\widehat{W}_u(m))_{m \in [1:M]}=(\widehat{W}_u(m)^{(1)})_{m \in [1:M]}$ of $W_u$, with the notation of Assumption \ref{assum_cost_random}. In this case, we choose i.i.d. realizations of $\widehat{W}_u$ conditionally to $(\bX^{(n)})_{n \in [1:N]}$. That is $(\widehat{W}_u(m))_{m \in [1:M]}= \left( \widehat{W}_{u,s(m)}\right)_{m \in [1:M]}$, where $\widehat{W}_{u,s(m)}$ is defined by either Equation \eqref{eq_Vui_MCmix}, Equation \eqref{eq_Vui_MCknn}, Equation \eqref{eq_Vui_PFmix} or Equation \eqref{eq_Vui_PFknn}, and $(s(m))_{m \in [1:M]}$ are independent and uniformly distributed on $[1:N]$. This enables to have different estimators with a small covariance using the same sample $(\bX^{(n)})_{n \in [1:N]}$. Indeed, to prove the consistency in Proposition \ref{prop_consis_shap_inconnue} of the Shapley effects estimator with the random-permutation procedure, we show that Assumption \ref{assum_consis_perm} is satisfied.

\begin{prop}\label{prop_consis_shap_inconnue}
Assume that $Assumption \ref{assum_polish}$ holds and Assumption \ref{assum_indep} holds for all subset $u$, $\emptyset \varsubsetneq u \varsubsetneq [1:p]$.
If $f$ is bounded, then the estimators of the Shapley effects defined by the random-permutation $W$-aggregation procedure or the subset $W$-aggregation procedure combined with $\widehat{W_u}=\widehat{E}_{u,MC}$ (resp. $\widehat{W}_u=\widehat{V}_{u,PF}$) converge to the Shapley effects in probability when $N$ and $N_{tot}$ go to $+\infty$.
\end{prop} 

\begin{rmk}
The Sobol indices are functions of the $(W_u)_{u\subset[1:p]}$. Indeed, we can define the Sobol index of a group of variables $\bX_u$ by either $S_u$ as in \cite{chastaing_indices_2013, broto_sensitivity_2019} or $S_u^{cl}$ as in \cite{iooss_shapley_2017}, where 
$$
S_u:= \frac{1}{\V(Y)}\sum_{v \subset u} (-1)^{|u|-|v|}V_v,\;\;\;\;S_u^{cl}:=\frac{V_u}{\V(Y)},
$$
and where we note that $V_u=\V(Y)-E_{-u}$ by the law of total variance. Thus, we get consistent estimators of the Sobol indices in the general setting of Assumption \ref{assum_polish}. Note that the sum over $u\subset[1:p]$ of the Sobol indices $S_u^{cl}$ is not equal to 1, and when the inputs are dependent, the Sobol index $S_u$ can take negatives values.
\end{rmk}

\subsection{Numerical experiments} In this section, we compute numerically the estimators of the Shapley effects with an unknown input distribution. As in Section \ref{section_num_connue}, we choose the linear Gaussian framework to compute the theoretical values of the Shapley effects.

We take the same parameters as in Section \ref{section_num_connue}. The size $N$ of the observed sample $(\bX^{(n)})_{n \in [1:N]}$ is 10000. Each estimator is computed 200 times. We now have 8 consistent estimators given by:
\begin{itemize}
\item 2 different $W$-aggregation procedures: subset or random-permutation;
\item 2 different estimators of $W_u$: double Monte-Carlo or Pick-and-Freeze;
\item 2 slightly different versions of the estimators of $W_u$: "mix" or "knn".
\end{itemize}

\begin{figure}[h!]
\center
\includegraphics[width=12cm,height=12cm]{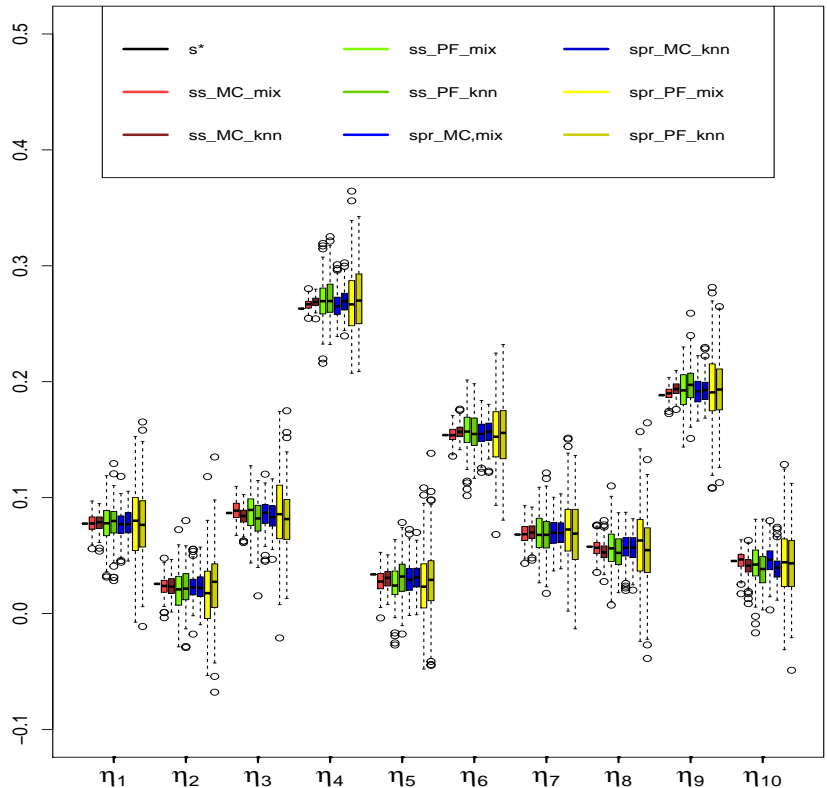}
\caption{Estimation of the Shapley effects in the linear Gaussian framework when we only observe a sample of $\bX$. In black (s*) we show the theoretical results, in red the estimates from the subset $W$-aggregation procedure with the double Monte-Carlo estimator (ss$\_$MC$\_$mix and ss$\_$MC$\_$knn), in green the estimates from the subset $W$-aggregation procedure with the Pick-and-Freeze estimator (ss$\_$PF$\_$mix and ss$\_$PF$\_$knn), in blue the estimates from the random-permutation $W$-aggregation procedure with the double Monte-Carlo estimator (spr$\_$MC$\_$mix and spr$\_$MC$\_$knn) and in yellow the estimates from the random-permutation $W$-aggregation procedure with the Pick-and-Freeze estimator (spr$\_$PF$\_$mix and spr$\_$PF$\_$knn).}
\label{boxplot_inconnue}
\end{figure}

In Figure \ref{boxplot_inconnue}, we plot the theoretical values of the Shapley effects, together with the boxplots of the 200 realizations of each estimator, and with a total cost $N_{tot}=54000$ (we assume here that $f$ is a costly computer code and that for all estimators, the cost is the number of evaluations of $f$).
\begin{figure}[t!]
\center
\includegraphics[width=10cm,height=10cm]{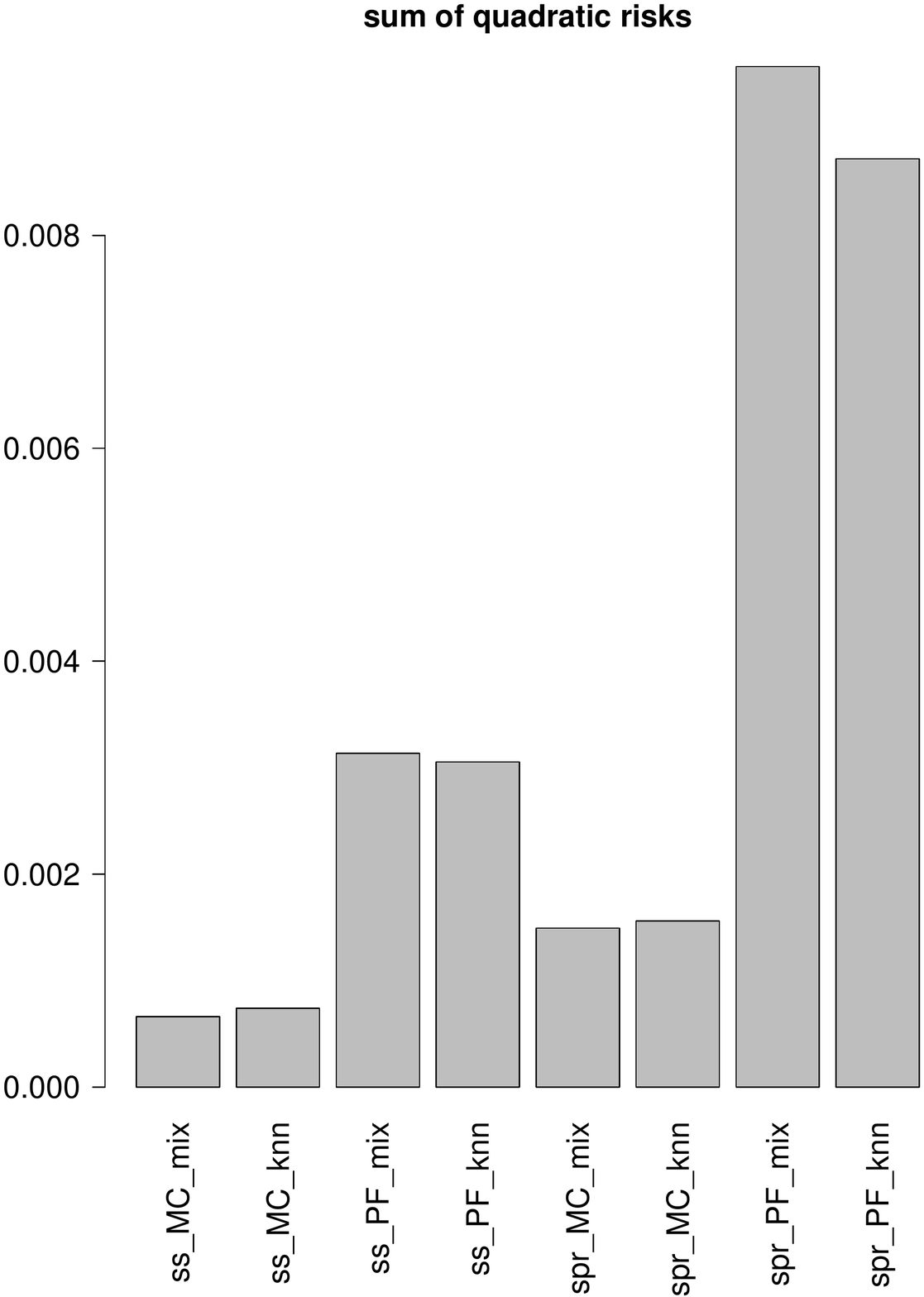}
\caption{Sum over $i$ of the estimated quadratic risks of the eight estimators of the Shapley effects in the linear Gaussian framework when we only observe a sample of $\bX$.}
\label{barplot_inconnue}
\end{figure}

\begin{rmk}
In the linear Gaussian framework, the function $f$ is not bounded and the assumptions of Proposition \ref{prop_consis_shap_inconnue} do not hold. We can thus not guarantee the consistency of the Shapley effects estimators. However, this framework enables to compute the theoretical Shapley effects and we can see numerically that the estimators seem to be consistent.
\end{rmk}

We show the sums over $i\in [1:p]$ of their quadratic risks (estimated with 200 realizations) in Figure \ref{barplot_inconnue}. As in Section \ref{section_num_connue}, the subset $W$-aggregation procedure is better than the random-permutation $W$-aggregation procedure and double Monte-Carlo is better than Pick-and-Freeze. Furthermore, there are no significant differences between the version "mix" and the version "knn". Recall that, in order to compute the estimators with the "mix" version, we need to call the computer code of $f$ at new inputs whereas "knn" only needs an i.i.d. sample $(\bX^{(n)},f(\bX^{(n)}))_{n \in [1:N]}$.\\

We now compare the sums over $i\in [1:p]$ of the estimated quadratic risks of the estimators from the subset $W$-aggregation procedure with double Monte-Carlo when we know the distribution of $\bX$ (results of Section \ref{section_num_connue}) and when we just observe a sample of size 10000 (previous results of this section). These values are equal to $5.9\;10^{-3}$ when we know the distribution of $\bX$, to  $6.6\;10^{-3}$ when we only observe the sample with $\widehat{E}_{u,MC}^{mix}$ and to $7.4\;10^{-3}$ when we only observe the sample with $\widehat{E}_{u,MC}^{knn}$. Thus, in dimension 10, replacing the knowledge of $\bX$ by a sample of size 10000 does not seem to deteriorate significantly our estimates of the Shapley effects.

\section{Application to real data}\label{section_applireal}

In this section, we apply the estimator of the Shapley effects given by the subset $W$-aggregation procedure and the double Monte-Carlo estimator $\widehat{E}_{u,MC}^{knn}$ in Equation \eqref{eq_VuMC_inconnue2} to a real data set. We use the "depSeuil.dat" data, available at \\{\verb http://www.math.univ-toulouse.fr/~besse/Wikistat/data } from \cite{besse_comparaison_2007}. This data set contains 10 variables with 1041 sample observations. The variables are:
\begin{itemize}
    \item JOUR: type of day (holiday: 1, no holiday: 0);
    \item O3obs: observed ozone concentration; 
    \item MOCAGE: ozone concentration predicted by a fluid mechanics model;
    \item TEMPE: temperature predicted by the official meteorology service of France;
    \item RMH2O: humidity ratio;
    \item NO2: nitrogen dioxide concentration;
    \item NO: nitrogen oxide concentration; 
    \item STATION: site of observation (5 different sites);
    \item VentMOD: wind force;
    \item VentANG: wind direction.
\end{itemize}
Here, we focus on the ozone concentration O3obs in function of the nine other variables. 
Hence, let $\tilde{Y}$ be the random variable of the ozone concentration and let $\bX$ be the random vector containing the nine other random variables.
Using the estimator $\widehat{E}_{\emptyset,MC}^{knn}$ of $\E_{\emptyset}=\E(\V(\tilde{Y}|\bX))$ given by Equation \eqref{eq_VuMC_inconnue2}, with $N_I=3$ and $N_{\emptyset}=1000$, we estimate the value of $\V(\E(\tilde{Y}|\bX))\slash \V(\tilde{Y})$ to 0.57, whereas it would be equal to 1 if $\tilde{Y}$ was a function of $\bX$. Thus, we can not assume that the ozone concentration is a function of the nine other random variables.

The theory and methodology of this article holds when $\tilde{Y}$ is a deterministic function of $\bX$. Hence, we create metamodels of the ozone concentration in function of $\bX$, and we write $Y$ the output of the metamodel. In this case, $Y$ is indeed a deterministic function of $\bX$ and we can compute the Shapley effects, which now quantify the impact of the inputs on the metamodel prediction. In practice, we replace the output column by the fitted values given by the metamodel.

To study the impact of the metamodel on the Shapley effects, we estimate the Shapley effects corresponding to three metamodels:
\begin{itemize}
    \item XGBoost, from the R package {\verb xgboost }, with optimized parameter by cross-validation;
    \item generalized linear model (GLM);
    \item Random Forest, from the R package {\verb randomForest }, which optimizes automatically the parameters by out-of-bag.
\end{itemize}

\begin{rmk}
Using the estimator $\widehat{E}_{\emptyset,MC}^{knn}$, we estimate the value of $\V(\E({Y}|\bX))\slash \V({Y})$ to 0.91, 0.93 and 0.90 where $Y$ denotes the output of each of the three metamodels XGBoost, GLM and Random Forest respectively. In contrast, the value of $\V(\E(\tilde{Y}|\bX))\slash \V(\tilde{Y})$ is 0.57 when $\tilde{Y}$ denotes the original observed ozone concentrations. This shows that the predicted values are different from the initial values of the ozone concentration. Moreover, this shows that the metamodels do not overfit the data, since the estimated values of $\V(\E(Y|\bX))\slash \V(Y)$ are close to 1. Indeed, that means that the fitted values of the ozone concentration are much more explained by $\bX$ and have been smoothed by the metamodels. Furthermore, if the metamodels were overfitting the noise contained in the observed ozone concentration values, their outputs could not be predicted well given $\bX$, and the estimate of $\V(\E(Y | \bX)) / \V(Y )$ would then be small when $Y$ is one of the metamodel outputs.

\end{rmk}

For each metamodel, we estimate the Shapley effects with the subset $W$-aggregation procedure and the double Monte-Carlo estimator $\widehat{E}_{u,MC}^{knn}$, with $N_I=3$ and $N_{tot}=50000$ (but the real cost is actually 40176, see Remark \ref{rmk_approx}). For each metamodel, the computation time of all the Shapley effects on a personal computer is around 5 minutes. The results are presented in Figure \ref{fig_Ozone}.\\

\begin{figure}
    \centering
    \includegraphics[scale=0.45]{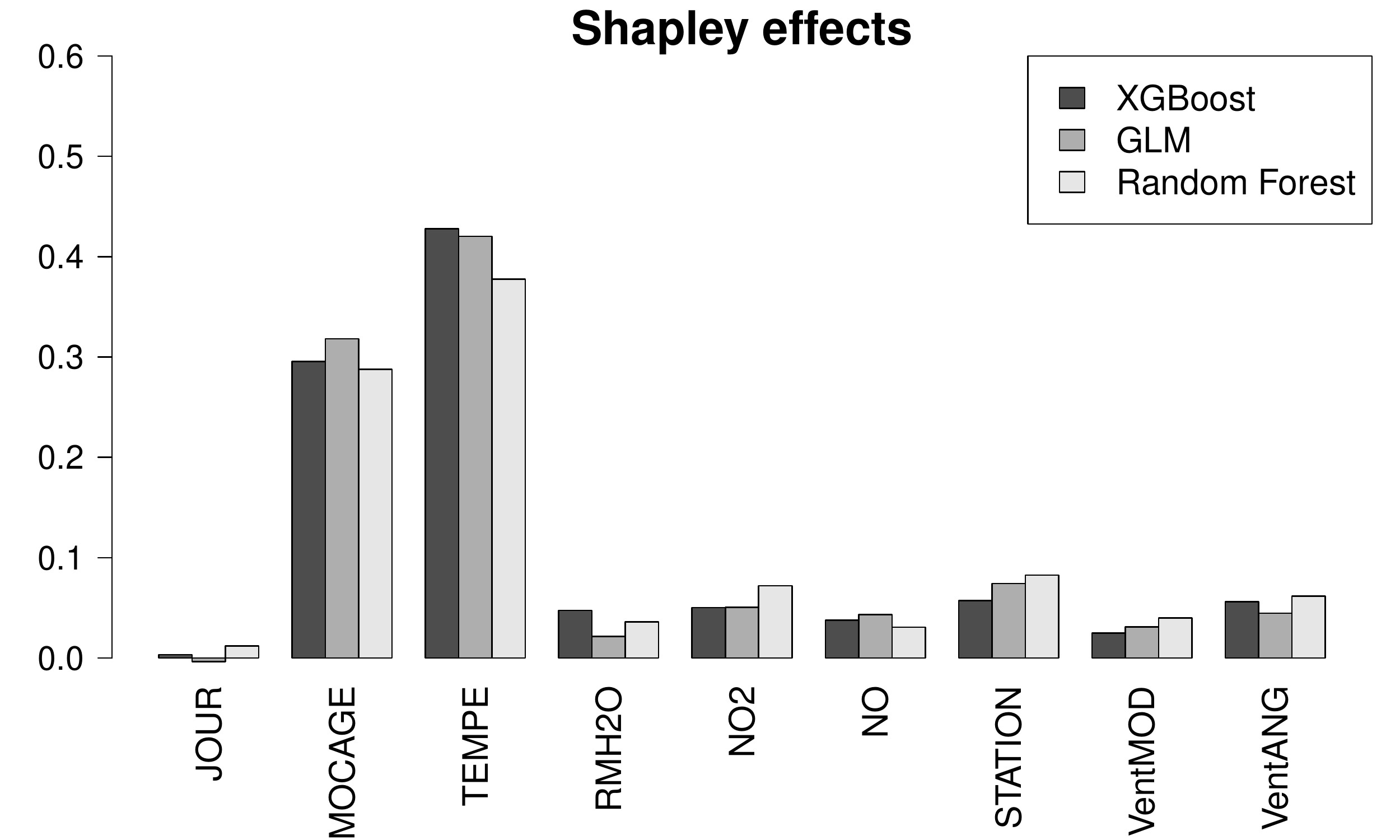}
    \caption{Estimation of the Shapley effects for three metamodels: XGBoost, GLM and Random Forest.}
    \label{fig_Ozone}
\end{figure}

We remark that the three metamodels yield similar Shapley effects. This is reassuring, since observing different behaviours of the metamodels would be a sign of inaccuracy for some of them. Only two variables have a significant impact on the ozone concentration: the predicted ozone concentration (MOCAGE) and the predicted temperature (TEMPE). This comforts the results of \cite{besse_comparaison_2007} as they use regression trees whose two most important variables are the predicted ozone concentration and the predicted temperature. All the other variables have a much smaller impact. The Shapley effect of the predicted temperature is larger than the one of the predicted ozone concentration. It could be from the better accuracy of the predicted temperature (given by the official meteorology service of France) than the predicted ozone concentration (given by a fluid mechanics model). Finally, we remark that the type of the day (holiday or not) has no impact on the ozone concentration. The corresponding Shapley effect is even estimated by a slightly negative value for the GLM, which stems from the small error estimation.

To conclude, the Shapley effect estimator given by the subset $W$-aggregation procedure and the double Monte-Carlo estimator $\widehat{E}_{u,MC}^{knn}$ enables us to estimate the Shapley effects on real data. The estimator only requires a data frame of the inputs-output and handles heterogeneous data, with some categorical inputs and some continuous inputs. Here, the estimator was applied to a metamodel output. This illustrates the interest of the Shapley effects (and of sensitivity analysis) to understand and interpret the predictions of complex black-box machine learning procedures \cite{ribeiro_why_2016,bachoc_entropic_2018}.

This estimator has been implemented in the R package {\verb sensitivity } as the function "shapleySubsetMc".

\section{Conclusion}\label{section_conclusion}
In this article, we focused on the estimation of the Shapley effects. We explained that this estimation is divided into two parts: the estimation of the conditional elements $(W_u)_{\emptyset \varsubsetneq u \varsubsetneq [1:p]}$ and the $W$-aggregation procedure. We suggested the new subset $W$-aggregation procedure and we explained how the already existing random-permutation $W$-aggregation procedure of \cite{song_shapley_2016} minimizes the variance. However, the subset $W$-aggregation procedure is more efficient by using all the estimates of the conditional elements for each Shapley effect estimation. We highlighted this efficiency by numerical experiments. In a second part, we suggested various estimators of $(W_u)_{\emptyset \varsubsetneq u \varsubsetneq [1:p]}$ when the input distribution is unknown and when we only observe an i.i.d. sample of the input variables.  We proved their consistency and gave the rates of convergence. Then, we used these new estimators to estimate the Shapley effects with consistency. We illustrated the efficiency of these estimators with numerical experiments and we tested one estimator on real heterogeneous data.

It is known that the Monte-Carlo algorithms for the estimation of the Sobol indices require many evaluations of $f$ to be accurate (typically several thousands). If the evaluation cost is too high, it could be necessary to replace the function $f$ with a metamodel $\widehat{f}$, such as Kriging \cite{santner_design_2003}. An important field of research in Kriging is adaptive design of experiments \cite{jones_efficient_1998}. It would be interesting to study adaptive design of experiments in order to estimate the Shapley effects \cite{fruth_sequential_2015}, and to develop adaptive algorithms adapted to the estimators that we suggest in this article.

\section*{Acknowledgments}
We are grateful to Vincent Prost for his helpful advises. We acknowledge the financial support of the Cross-Disciplinary Program on Numerical Simulation of CEA, the French Alternative Energies and Atomic Energy Commission.
We would like to thank BPI France for co-financing this work, as part of the PIA (Programme d'Investissements d'Avenir) - Grand D\'{e}fi du Num\'{e}rique 2, supporting the PROBANT project.
We are grateful to Philippe Besse and Olivier Mestre for enabling us to use the Ozone data set.
We are very grateful to the associate editor and two reviewers, for their comments that lead to an improvement of this article.

\appendix

\section{Proofs for the double Monte-Carlo and Pick-and-Freeze estimators: Theorems \ref{thrm_consis_MC}, \ref{thrm_vitesse_MC}, \ref{thrm_consis_PF} and \ref{thrm_vitesse_PF}}
\leavevmode\par
\bigskip

To unify notation, let us write
\begin{eqnarray*}
\Phi^{mix}_{MC}:&(\bx^{(1)},...,\bx^{(N_I)})&\longmapsto \frac{1}{N_I-1}\sum_{k=1}^{N_I}\left( f(\bx^{(1)}_{-u},\bx^{(k)}_{u})-\frac{1}{N_I}\sum_{l=1}^{N_I}f(\bx^{(1)}_{-u},\bx^{(l)}_{u})\right)^2,\\
\Phi^{knn}_{MC}:&(\bx^{(1)},...,\bx^{(N_I)})&\longmapsto \frac{1}{N_I-1}\sum_{k=1}^{N_I}\left( f(\bx^{(k)})-\frac{1}{N_I}\sum_{l=1}^{N_I}f(\bx^{(l)})\right)^2,\\
\Phi_{PF}^{mix}: & (\bx^{(1)},\bx^{(2)}) & \longmapsto f(\bx^{(1)})f(\bx_{u}^{(1)},\bx_{-u}^{(2)})-\E(Y)^2,\\
\Phi_{PF}^{knn}: & (\bx^{(1)},\bx^{(2)}) & \longmapsto f(\bx^{(1)})f(\bx^{(2)})-\E(Y)^2.
\end{eqnarray*}
Remark that all these four functions
are bounded as $f$ is bounded.
When we do not write the exponent $mix$ or $knn$ of $\Phi$ or of the estimators, it means that we refer to both of them ($mix$ and $knn$).  
We write the proofs only for $\widehat{E}_{u,MC}$. For the estimators $\widehat{V}_{u,PF}$, it suffices to replace $\Phi_{MC}$ by $\Phi_{PF}$, $-u$ by $u$ (and vice-versa), $E_u$ by $V_u$, $\V(Y|X_{-u})$ by $\E(Y|X_u)^2-\E(Y)^2$ and $N_I$ by $2$. Hence, we shall only write the complete proofs for Theorems \ref{thrm_consis_MC} and \ref{thrm_vitesse_MC}. To simplify notation, we will write $\widehat{E}_u$ for $\widehat{E}_{u,MC}$, $\widehat{E}_{u,l}$ for $\widehat{E}_{u,l,MC}$ and $\Phi$ for $\Phi_{MC}$. $N_I$ is a fixed integer. We also write $k_N(l,i):=k_N^{-u}(l,i)$, and the dependence on $-u$ is implicit.

\subsection{Proof of consistency: Theorems \ref{thrm_consis_MC} and \ref{thrm_consis_PF}}\label{section_consistance}

Recall that for all $i \in [1:p]$, $(\cX_i,d_i)$ is a Polish space. Then, for all $v \subset[1:p]$, $\cX_v:=\prod_{i\in v}\cX_i$ is a Polish space for the distance $d_v:=\max_{i\in v} d_i$. We will write $B_v(\bx_v,r)$ the open ball in $\cX_v$ of radius $r$ and center $\bx_v$. We also let $\mu_v:=\bigotimes_{i\in v} \mu_i$. Recall that the choice of the $N_I$-nearest neighbours could be not unique. In this case, conditionally to $(\bX_{-u}^{(n)})_{n\leq N}$, the $(k_N(l,i))_{l\in [1:N],i \in [1:N_I]}$ are random variables that we choose in the following way. Conditionally to $(\bX_{-u}^{(n)})_{n\leq N}$, we choose $k_N(l,i)$ uniformly over all the indices of the $i$-th nearest neighbours of $\bX_{-u}^{(l)}$, such that the $(k_N(l,i))_{i\leq N_I}$ are two by two distinct and independent of all the other random variables conditionally to $(\bX_{-u}^{(n)})_{n\leq N}$.

In particular, as we want to prove asymptotic results, we assume (without loss of generality) that we have an infinite i.i.d. sample $(\bX^{(n)})_{n\in \N^*}$, and we assume that for all $N\in \N^*$, conditionally to  $(\bX_{-u}^{(n)})_{n\leq N}$, $(k_N(l,i))_{i\leq N_I}\ind \sigma\left((\bX_{u}^{(n)})_{n\leq N},(\bX^{(n)})_{n> N},(k_{N'}(l',i'))_{(N',l')\neq (N,l),\; i'\in [1:N_I]} \right)$. Hence, for all $N\in \N^*$ and $l\in [1:N]$, conditionally to  $(\bX_{-u}^{(n)})_{n\in \N}$, we have 
$$
(k_N(l,i))_{i\leq N_I}\ind \sigma\left((\bX_u^{(n)})_{n\in \N},(k_{N'}(l',i'))_{(N',l')\neq (N,l),\; i'\in [1:N_I]} \right).
$$

To simplify notation, let us write  $k_N(i):=k_N(1,i)$ (the index of one $i$-th neighbour of $\bX_{-u}^{(1)}$) and $k'_N(i):=k_N(2,i)$ (the index of one $i$-th neighbour of $\bX_{-u}^{(2)}$). Remark that $\bX_{-u}^{(k_N(i))}$ does not depend on $k_N(i)$. Let $\bk:=(k_N(i))_{i\leq N_I,N\in \N^*}$ and $\bk_N:=(k_N(i))_{i\leq N_I}$. We will use the letter $\bh$ for the realizations of the variable $\bk$.

To begin with, let us recall two well-known results that we will use in the following.
\begin{lm}\label{lm_gen_ind}
Let $A$ be a real random variable. If $\mathcal{H}$ is independent of $\sigma(\sigma(A),\mathcal{G})$, then
$$
\E(A|\sigma(\mathcal{G},\mathcal{H}))=\E(A|\mathcal{G}).
$$
\end{lm}

\begin{lm}\label{lm_gen_transfert}
Let $A,B$ be random variables. For all measurable $\phi$,
$$
\mathcal{L}(\phi(A,B)|A=a)=\mathcal{L}(\phi(a,B)|A=a)
$$
and if $B$ is independent of $A$, then
$$
\mathcal{L}(\phi(A,B)|A=a)=\mathcal{L}(\phi(a,B)).
$$
\end{lm}

Now, to prove Theorem \ref{thrm_consis_MC}, we need to prove several intermediate results.
\begin{lm}\label{lm_conv_voisin}
For all $l\in \N^*$,
\begin{equation}
\bX^{(k_N(l))}_{-u}\overset{a.s.}{\underset{N\rightarrow +\infty}{\longrightarrow}}\bX_{-u}^{(1)}.
\end{equation}
\end{lm}

\begin{proof}
First, let us show that for all $\varepsilon>0$, $\PP(d_{-u}(\bX_{-u}^{(1)},\bX_{-u}^{(2)})<\varepsilon)>0$. Indeed, as $\cX_{-u}$ is a Polish space, its support has measure 1. Thus
\begin{eqnarray*}
\PP(d_{-u}(\bX_{-u}^{(1)},\bX_{-u}^{(2)})<\varepsilon)&=&\int_{\cX_{-u}^2}\mathds{1}_{d_{-u}(\bx_{-u},\bx_{-u}')<\varepsilon} d\PP_{\bX_{-u}}\otimes\PP_{\bX_{-u}}(\bx_{-u},\bx_{-u}')\\
&=&\int_{\cX_{-u}} \PP_{\bX_{-u}}(B_{-u}(\bx_{-u},\varepsilon))d\PP_{\bX_{-u}}(\bx_{-u})\\
&=&\int_{supp(\cX_{-u})} \PP_{\bX_{-u}}(B_{-u}(\bx_{-u},\varepsilon))d\PP_{\bX_{-u}}(\bx_{-u})\\
&>&0,
\end{eqnarray*}
because if $\bx_{-u} \in supp(\cX_{-u})$, then $B_{-u}(\bx_{-u},\varepsilon)\not\subset supp(\cX_{-u})^c$ and $\PP_{\bX_{-u}}(B_{-u}(\bx_{-u},\varepsilon))>0$.\\

Next, remark that
$$ 
\bX_{-u}^{(k_N(l))}\overset{a.s.}{\underset{N\rightarrow +\infty}{\longrightarrow}}\bX_{-u}^{(1)}\; \;\Longleftrightarrow \;\; \bX_{-u}^{(k_N(2))}\overset{a.s.}{\underset{N\rightarrow +\infty}{\longrightarrow}}\bX_{-u}^{(1)},
$$
and,
\begin{align*}
\PP\left( \left\{ \bX_{-u}^{(k_N(2))}{\underset{N\rightarrow +\infty}{\longrightarrow}}\bX_{-u}^{(1)}\right\}^c \right)&=\PP\left( \bigcup_{k\geq 1}\bigcap_{n\geq 2}d_{-u}(\bX_{-u}^{(n)},\bX_{-u}^{(1)})\geq \frac{1}{k}\right)\\
&\leq  \sum_{k\geq 1}\PP\left(\bigcap_{n\geq 2}d_{-u}(\bX_{-u}^{(n)},\bX_{-u}^{(1)})\geq \frac{1}{k} \;  \right) \\
&=  \sum_{k\geq 1}\lim_{N\rightarrow +\infty}\PP\left(d_{-u}(\bX_{-u}^{(2)},\bX_{-u}^{(1)})\geq \frac{1}{k} \right)^N\\
&= \sum_{k\geq 1}  \lim_{N\rightarrow +\infty}\left[ 1-\PP\left(d_{-u}(\bX_{-u}^{(2)},\bX_{-u}^{(1)})< \frac{1}{k} \right)\right]^N\\
&=\sum_{k\geq 1}0\\
&=0.
\end{align*}
\end{proof}

\begin{lm}\label{lm_continuite_loi_condi}
There exists a version of
$$
\mathcal{L}(\bX_u|\bX_{-u}=\cdot):(\mathcal{X}_{-u},d_{-u})\longrightarrow (\mathcal{M}_1(\cX_u),\mathcal{T}(weak))
$$
which is continuous $\PP_{\bX_{-u}}$-a.e., where $\mathcal{M}_1(\mathcal{X}_u)$ is the set of probability measures on $\mathcal{X}_u$ and $\mathcal{T}(weak)$ is the topology of weak convergence.
\end{lm}

\begin{proof}
We assumed that there exists a version of $f_{\bX}$ which is bounded and $\PP_{\bX}$-a.e. continuous. Let
$$
f_{\bX_{-u}}(\bx_{-u}):=\int_{\mathcal{X}_u}f_{\bX}(\bx_u,\bx_{-u})d\mu_u(\bx_u),
$$
which is bounded by $\mu_u(\mathcal{X}_u)\|f_\bX\|_{\infty}$ and is a $\PP_{\bX_{-u}}$-a.e. continuous (thanks to the dominated converging Theorem) version of the density of $\bX_{-u}$ with respect to $\mu_{-u}$.
Let $\bx_{-u}\in \cX_{-u}$ such that $f_{\bX_{-u}}(\bx_{-u})\leq \| f_{\bX_{-u}} \|_{\infty}$, $f_{\bX_{-u}}(\bx_{-u})>0$ and such that $f_{\bX_{-u}}$ is continuous at $\bx_{-u}$. We have that
$$
f_{\bX_u|\bX_{-u}=\bx_{-u}}(\bx_u):=\frac{f_\bX(\bx_u,\bx_{-u})}{f_{\bX_{-u}}(\bx_{-u})}
$$
is a version of the density of $\bX_u$ conditionally to $\bX_{-u}=\bx_{-u}$ (defined for almost all $\bx_{-u}$).
Let $(\bx_{-u}^{(n)})$ be a sequence converging to $\bx_{-u}$. There exists $n_0$ such that for all $n\geq n_0$, $f_{\bX_{-u}}(\bx_{-u}^{(n)})>0$. Thus, by continuity of $f$ which respect to $\bx_{-u}$ and of $f_{\bX_{-u}}$, we have $f_{\bX_u|\bX_{-u}=\bx_{-u}}(\bx_u)=\lim_{n\rightarrow+\infty} f_{\bX_u|\bX_{-u}=\bx_{-u}^{(n)}}(\bx_u)$ for almost all $\bx_u$. Then, using the dominated converging Theorem, 
$$
\mathcal{L}(\bX_{u}|\bX_{-u}=\bx_{-u}^{(n)})\overset{weakly}{\underset{N\rightarrow+\infty}{\longrightarrow}}\mathcal{L}(\bX_{u}|\bX_{-u}=\bx_{-u}).
$$
\end{proof}

\begin{rmk}
The assumption "$\bX=(\bX_u,\bX_{-u})$ has a bounded density $f_\bX$ with respect to a finite measure $\mu=\bigotimes_{i=1}^p\mu_i$, which is continuous $\PP_\bX$-a.e." is only used in the proof of Lemma \ref{lm_continuite_loi_condi}. It would be interesting in future work to prove \ref{lm_continuite_loi_condi} with a weaker assumption.
\end{rmk}
\begin{rmk}There exists a different proof of Lemma \ref{lm_continuite_loi_condi} if we assume that $\mu$ is regular. Theorem 8.1 of \cite{tjur_conditional_1974} ensures that the conditional distribution in the sense of Tjur is defined for all $\bx_{-u}$ such that $f_{\bX_{-u}}>0$ (and not only for almost all $\bx_{-u}$) and the continuity of $f_{\bX_u|\bX_{-u}=\bx_{-u}}(\bx_u)$ with respect to $\bx_{-u}$ comes from Theorem 22.1 of \cite{tjur_conditional_1974}.
\end{rmk}

\begin{rmk}
To avoid confusion, we can now define $\mathcal{L}(\bX_u|\bX_{-u}=\bx_{-u})$ as the probability measure of density $\frac{f(\cdot,\bx_{-u})}{f_{\bX_{-u}}(\bx_{-u})}$, which is defined for all (and not "almost all") $\bx_{-u}$ in $\{f_{\bX_{-u}}>0\}$.
\end{rmk}

\begin{prop}\label{prop_asymp_sans_bias}
If
$$
\mathcal{L}(\bX_u|\bX_{-u}=.):(\cX_{-u},d_{-u})\longrightarrow (\mathcal{M}_1(\cX_u),\mathcal{T}(weak))
$$
is continuous (where $\mathcal{T}(weak)$ is the topology of weak convergence) almost everywhere, then, for almost all $\left( (\bx_{-u}^{(n)})_n,\bh \right)$, we have
\begin{equation}\label{eq1}
\E\left(\widehat{E}_{u,1}\left| (\bX_{-u}^{(n)})_n=(\bx_{-u}^{(n)})_n, \bk=\bh   \right. \right)\underset{N\rightarrow+\infty}{\longrightarrow} \V(Y|\bX_{-u}=\bx_{-u}^{(1)})
\end{equation}
and,
\begin{equation}\label{eq2}
\E(\widehat{E}_{u,1}) \underset{N\rightarrow+\infty}{\longrightarrow} E_u .
\end{equation}
\end{prop}

\begin{proof}
Let $\bZ=(\bZ_1,...,\bZ_{N_I}):(\Omega,\mathcal{A})\rightarrow(\cX^{N_I},\mathcal{E}^{\otimes N_I})$ measurable, where $\mathcal{E}$ is the $\sigma$-algebra on $\mathcal{X}$, such that for almost all $\left((\bx_{-u}^{(n)})_n, \bh \right)$, we have 
$$
\mathcal{L}\left(\bZ|(\bX_{-u}^{(n)})_n=(\bx_{-u}^{(n)})_n,\bk=\bh\right)=\bigotimes_{i=1}^{N_I}\mathcal{L}(\bX^{(1)}|\bX^{(1)}_{-u}=\bx^{(1)}_{-u}).
$$
It suffices to show that, for almost all $\left( (\bx_{-u}^{(n)})_n, \bh \right)$,
\begin{equation}\label{conv_KNN_sachant}
(\bX^{(k_N(i))})_{i\leq N_I}
\overset{\mathcal{L}_{|(\bX_{-u}^{(n)})_n=(\bx_{-u}^{(n)})_n,\bk=\bh}}{\underset{N\rightarrow+\infty}{\longrightarrow}}\bZ.
\end{equation}
Indeed, if Equation \eqref{conv_KNN_sachant} is true, then, using that $\Phi$ is bounded,
\begin{eqnarray*}
& & \E\left(\widehat{E}_{u,1}\right.\left| (\bX_{-u}^{(n)})_n=(\bx_{-u}^{(n)})_n, \bk=\bh   \right.) \\
&=& \E\left[\Phi\left( (\bX^{(k_N(i))})_{i\leq N_I} \right)\right.\left| (\bX_{-u}^{(n)})_n=(\bx_{-u}^{(n)})_n, \bk=\bh   \right]\\
& \underset{ N\rightarrow+\infty}{\longrightarrow}&  \E(\Phi(\bZ)\left| (\bX_{-u}^{(n)})_n=(\bx_{-u}^{(n)})_n, \bk=\bh   \right.)\\
&=&\V(Y|\bX_{-u}=\bx_{-u}^{(1)}),
\end{eqnarray*}
by definition of $\bZ$ and of $\Phi$.
Thus, we have Equation \ref{eq1}.
Furthermore, using dominated convergence theorem, integrating on $\left((\bx_{-u}^{(n)})_n,\bh\right)$, we obtain Equation \ref{eq2}.

Thus, it remains to show that conditionally to $(\bX_{-u}^{(n)})_n=(\bx_{-u}^{(n)})_n,\bk=\bh$, the random vector $(\bX^{(k_N(i))})_{i\leq N_I}$ converges in distribution to $\bZ$. We prove this convergence step by step.
\begin{lm}\label{lm_tildeX}
For almost all $(\bx_{-u}^{(n)})_n$,
$$
\mathcal{L}((\bX_u^{(n)})_n|(\bX_{-u}^{(n)})_n=(\bx_{-u}^{(n)})_n)= \bigotimes_{n\geq 1}\mathcal{L}(\bX_u|\bX_{-u}=\bx_{-u}^{(n)}).$$
\end{lm}

\begin{proof}
Let $(\tilde{\bX}_{-u}^{(n)})_n:\Omega\rightarrow \cX_{-u}^\N$ be an i.i.d. sequence of distribution $\mathcal{L}( \bX_{-u})$. Then, we let $(\tilde{\bX}_{u}^{(n)})_n:\Omega\rightarrow \cX_{u}^\N$ be a sequence with conditional distribution 
$$
\mathcal{L}((\tilde{\bX}_u^{(n)})_n|(\tilde{\bX}_{-u}^{(n)})_n=(\bx_{-u}^{(n)})_n)= \bigotimes_{n\geq 1}\mathcal{L}(\bX_u|\bX_{-u}=\bx_{-u}^{(n)}).$$
We just have to prove that $(\tilde{\bX}^{(n)})_n$ is an i.i.d. sample of distribution $\mathcal{L}(\bX)$.

Each $\tilde{\bX}^{(n)}$ has a distribution $\mathcal{L}(\bX)$ because for all bounded measurable $\phi$,
\begin{eqnarray*}
\E(\phi(\tilde{\bX}^{(n)}))&=& \int_\Omega \phi(\tilde{\bX}^{(n)}(\omega))d\PP(\omega)\\
&=& \int_{\cX_u\times \cX_{-u}} \phi(\bx_u,\bx_{-u})d\PP_{(\tilde{\bX}_u,\tilde{\bX}_{-u})}(\bx_u,\bx_{-u})\\
&=&\int_{\cX_{-u}}\left( \int_{\cX_{u}}  \phi(\bx_u,\bx_{-u})d\PP_{\bX_u|\bX_{-u}=\bx_{-u}}(\bx_u)\right) d\PP_{\bX_{-u}}(\bx_{-u})\\
&=&\int_\cX \phi(\bx) d\PP_\bX(\bx).
\end{eqnarray*}
Moreover, $(\tilde{\bX}^{(n)})_n$ are independent because if $n\neq m$, then, for all bounded Borel functions $\phi_1$ and $\phi_2$, we have:
\begin{eqnarray*}
&&\E(\phi_1(\tilde{\bX}^{(n)})\phi_2(\tilde{\bX}^{(m)}))\\
&=&\int_{\cX_{u}^2 \times \cX_{-u}^2}\phi_1(\bx_u^{(n)},\bx_{-u}^{(n)}) \phi_2(\bx_u^{(m)},\bx_{-u}^{(m)} )d\PP_{(\tilde{\bX}_u^{(n)},\tilde{\bX}_u^{(m)},\tilde{\bX}_{-u}^{(n)},\tilde{\bX}_{-u}^{(m)})}(\bx_u^{(n)},\bx_u^{(m)},\bx_{-u}^{(n)},\bx_{-u}^{(m)})\\
&=&\int_{\cX_{-u}^2}\left( \int_{\cX_{u}^2}\phi_1(\bx_u^{(n)},\bx_{-u}^{(n)}) \phi_2(\bx_u^{(m)},\bx_{-u}^{(m)} )d\PP_{(\tilde{\bX}_u^{(n)},\tilde{\bX}_{u}^{(m)})|(\tilde{\bX}_{-u}^{(n)},\tilde{\bX}_{-u}^{(m)})=(\bx_{-u}^{(n)},\bx_{-u}^{(m)})} (\bx_u^{(n)},\bx_u^{(m)}) \right) d\PP_{(\tilde{\bX}_{-u}^{(n)},\tilde{\bX}_{-u}^{(m)})} (\bx_{-u}^{(n)},\bx_{-u}^{(m)})\\
&=&\int_{\cX_{-u}^2}\left( \int_{\cX_{u}^2}\phi_1(\bx_u^{(n)},\bx_{-u}^{(n)}) \phi_2(\bx_u^{(m)},\bx_{-u}^{(m)} )d\PP_{\bX_u|\bX_{-u}=\bx_{-u}^{(n)}}\otimes\PP_{\bX_u|\bX_{-u}=\bx_{-u}^{(m)}} (\bx_u^{(n)},\bx_u^{(m)}) \right) d\PP_{\bX_{-u}}^{\otimes 2} (\bx_{-u}^{(n)},\bx_{-u}^{(m)})\\
&=&\int_{\cX_{-u}^2}\left( \int_{\cX_{u}}\phi_1(\bx_u^{(n)},\bx_{-u}^{(n)}) d\PP_{\bX_u|\bX_{-u}=\bx_{-u}^{(n)}} (\bx_u^{(n)}) \right)\\
& & \left( \int_{\cX_{u}}\phi_2(\bx_u^{(m)},\bx_{-u}^{(m)}) d\PP_{\bX_u|\bX_{-u}=\bx_{-u}^{(m)}} (\bx_u^{(m)}) \right)  d\PP_{\bX_{-u}}^{\otimes 2} (\bx_{-u}^{(n)},\bx_{-u}^{(m)})\\
&=&\int_{\cX_{-u}}\left( \int_{\cX_{u}}\phi_1(\bx_u^{(n)},\bx_{-u}^{(n)}) d\PP_{\tilde{\bX}_u|\tilde{\bX}_{-u}=\bx_{-u}^{(n)}} (\bx_u^{(n)}) \right) d\PP_{\tilde{\bX}_{-u}} (\bx_{-u}^{(n)})\\
& & \left( \int_{\cX_{u}}\phi_2(\bx_u^{(m)},\bx_{-u}^{(m)}) d\PP_{\tilde{\bX}_u|\tilde{\bX}_{-u}=\bx_{-u}^{(m)}} (\bx_u^{(m)}) \right)  d\PP_{\tilde{\bX}_{-u}} (\bx_{-u}^{(m)})\\
&=&\E(\phi_1(\tilde{\bX}^{(n)}))\E(\phi_2(\tilde{\bX}^{(m)})).
\end{eqnarray*} 
The above calculation can be extended to finite products of more than two terms.
That concludes the proof of Lemma \ref{lm_tildeX}.
\end{proof}

\begin{lm}\label{lm_loi_KNN}
For almost all $\left((\bx_{-u}^{(n)})_n,\bh\right)$, we have:
\begin{eqnarray*}
 \mathcal{L}\left((\bX_u^{(k_N(i))})_{i\leq N_I}|(\bX_{-u}^{(n)})_n=(\bx_{-u}^{(n)})_n,\bk=\bh\right)=\bigotimes_{i=1}^{N_I}\mathcal{L}\left( \bX_u|\bX_{-u}=\bx_{-u}^{(h_N(i))} \right).
\end{eqnarray*}
\end{lm}

\begin{proof}
For all bounded Borel function $\phi$, 
\begin{eqnarray*}
& & \E \left(\phi((\bX_u^{(k_N(i))})_{i\leq N_I})|(\bX_{-u}^{(n)})_n=(\bx_{-u}^{(n)})_n,\bk=\bh \right)\\
&=&\E\left(\left. \phi\left((\bX_u^{(k_N(i))})_{i\leq N_I})\right)\right| (\bX_{-u}^{(n)})_n=(\bx_{-u}^{(n)})_n, \left(k_{N'}(i)\right)_{i\leq N_I,N'\in \N^*}=\left(h_{N'}(i)\right)_{i\leq N_I,N'\in \N^*}\right)\\
&=&\E\left(\left. \phi\left((\bX_u^{(k_N(i))})_{i\leq N_I}\right)\right| (\bX_{-u}^{(n)})_n=(\bx_{-u}^{(n)})_n, \left(k_{N}(i)\right)_{i\leq N_I}=\left(h_{N}(i)\right)_{i\leq N_I}\right)\\
&=&\E\left( \phi\left((\bX_u^{(h_N(i))})_{i\leq N_I}\right)|(\bX_{-u}^{(n)})_n=(\bx_{-u}^{(n)})_n\right),
\end{eqnarray*}
using Lemmas \ref{lm_gen_ind} and \ref{lm_gen_transfert} conditionally to $(\bX_{-u}^{(n)})_n=(\bx_{-u}^{(n)})_n$.
Then,
\begin{eqnarray*}
&&\E\left( \phi\left((\bX_u^{(h_N(i))})_{i\leq N_I}\right)|(\bX_{-u}^{(n)})_n=(\bx_{-u}^{(n)})_n\right)\\
&=&\int_{\cX_u^{N_I}}\phi(\bx_u^{(1)},...,\bx_u^{(N_I)})d\PP_{(\bX_u^{(h_N(i))})_{i\leq N_I}|(\bX_{-u}^{(n)})_n=(\bx_{-u}^{(n)})_n}(\bx_u^{(1)},...,\bx_u^{(N_I)})\\
&=&\int_{\cX_u^{N_I}}\phi(\bx_u^{(1)},...,\bx_u^{(N_I)}) d \bigotimes_{i=1}^{N_I}\PP_{\bX_u|\bX_{-u}=\bx_{-u}^{(h_N(i))}}(\bx_u^{(1)},...,\bx_u^{(N_I)}).
\end{eqnarray*}
That concludes the proof of Lemma \ref{lm_loi_KNN}.
\end{proof}

Recall that $\bX_{-u}^{(k_N(i)))}\underset{N\rightarrow+\infty}{\longrightarrow}\bX_{-u}^{(1)}$ $\PP$-a.e., thus, for almost all $\left((\bx_{-u}^{(n)})_n,\bh\right)$,
$$
\bx_{-u}^{(h_N(i))}\overset{}{\underset{N\rightarrow+\infty}{\longrightarrow}}\bx_{-u}^{(1)}.
$$
Thus, using the continuity of the conditional distribution given by Lemma \ref{lm_continuite_loi_condi}, for almost all $\left((\bx_{-u}^{(n)})_n,\bh\right)$, we have,
$$
\mathcal{L}(\bX_{u}|\bX_{-u}=\bx_{-u}^{(h_N(i))})\overset{weakly}{\underset{N\rightarrow+\infty}{\longrightarrow}}\mathcal{L}(\bX_{u}|\bX_{-u}=\bx_{-u}^{(1)}).
$$
Thus, for almost all $\left((\bx_{-u}^{(n)})_n,\bh\right)$,
$$
\bigotimes_{i=1}^{N_I}\mathcal{L}(\bX_{u}|\bX_{-u}=\bx_{-u}^{(h_N(i))})\overset{weakly}{\underset{N\rightarrow+\infty}{\longrightarrow}}\bigotimes_{i=1}^{N_I}\mathcal{L}(\bX_{u}|\bX_{-u}=\bx_{-u}^{(1)})=\mathcal{L}(\bZ_u|\bX_{-u}^{(1)}=\bx_{-u}^{(1)}).
$$
So, using Lemma \ref{lm_loi_KNN}, for almost all $\left((\bx_{-u}^{(n)})_n,\bh\right)$,
\begin{eqnarray*}
\mathcal{L}\left((\bX_u^{(k_N(i))})_{i\leq N_I}|(\bX_{-u}^{(n)})_n=(\bx_{-u}^{(n)})_n,\bk=\bh\right)
\overset{weakly}{\underset{N\rightarrow+\infty}{\longrightarrow}} \mathcal{L}(\bZ_u|\bX_{-u}^{(1)}=\bx_{-u}^{(1)}).
\end{eqnarray*}
So, for almost all $\left((\bx_{-u}^{(n)})_n,\bh\right)$,
\begin{eqnarray*}
\mathcal{L}\left((\bX_u^{(k_N(i))})_{i\leq N_I}|(\bX_{-u}^{(n)})_n=(\bx_{-u}^{(n)})_n,\bk=\bh\right)
\overset{weakly}{\underset{N\rightarrow+\infty}{\longrightarrow}}\mathcal{L}\left(\bZ_u|(\bX_{-u}^{(n)})_n=(\bx_{-u}^{(n)})_n,\bk=\bh\right).
\end{eqnarray*}

Using Slutsky lemma, for almost all $\left((\bx_{-u}^{(n)})_n,\bh\right)$,
\begin{eqnarray*}
\mathcal{L}\left((\bX^{(k_N(i))})_{i\leq N_I}|(\bX_{-u}^{(n)})_n=(\bx_{-u}^{(n)})_n,\bk=\bh\right)
\overset{weakly}{\underset{N\rightarrow+\infty}{\longrightarrow}}\mathcal{L}\left(\bZ|(\bX_{-u}^{(n)})_n=(\bx_{-u}^{(n)})_n,\bk=\bh\right),
\end{eqnarray*}
that concludes the proof of Proposition \ref{prop_asymp_sans_bias}.
\end{proof}

\begin{lm}
The value of $\V(\widehat{E}_{u,1,MC})$ is bounded by $128\|f\|_\infty^4$.
\end{lm}

\begin{proof}
As $f$ is bounded, $\Phi$ is bounded by $\frac{1}{N_I-1}\sum_{k=1}^{N_I}(2\|f\|_\infty)^2=\frac{N_I}{N_I-1}4\|f\|_\infty^2\leq 8\|f\|_\infty^2$ so $\V(\widehat{E}_{u,1})$ is bounded by $2\|\Phi\|_\infty^2\leq 128\|f\|_\infty^4.$
\end{proof}

\begin{prop}\label{prop_covE12}
We have
$$
cov(\widehat{E}_{u,1},\widehat{E}_{u,2})\underset{N\rightarrow+\infty}{\longrightarrow}0.
$$
\end{prop}

\begin{proof}
We use the law of total covariance
\begin{equation}
cov(\widehat{E}_{u,1},\widehat{E}_{u,2})=\E\left(cov\left(\widehat{E}_{u,1},\widehat{E}_{u,2}|\bX_{-u}^{(1)},\bX_{-u}^{(2)}\right)\right) +cov\left(\E(\widehat{E}_{u,1}|\bX_{-u}^{(1)},\bX_{-u}^{(2)}),\E(\widehat{E}_{u,2}|\bX_{-u}^{(1)},\bX_{-u}^{(2)})\right).
\end{equation}
We will show that both terms go to $0$ as $N$ goes to $+\infty$. Let us compute the second term. Using Proposition \ref{prop_asymp_sans_bias},
\begin{eqnarray*}
& & cov\left(\E(\widehat{E}_{u,1}|\bX_{-u}^{(1)},\bX_{-u}^{(2)}),\E(\widehat{E}_{u,2}|\bX_{-u}^{(1)},\bX_{-u}^{(2)})\right)\\
&=&\E\left(\E(\widehat{E}_{u,1}|\bX_{-u}^{(1)},\bX_{-u}^{(2)})\E(\widehat{E}_{u,2}|\bX_{-u}^{(1)},\bX_{-u}^{(2)})\right) - \E(\widehat{E}_{u,1})\E(\widehat{E}_{u,2})\\
& \underset{N\rightarrow +\infty}{\longrightarrow} & \E\left( \V(Y|\bX_{-u}=\bX_{-u}^{(1)})\V(Y|\bX_{-u}=\bX_{-u}^{(2)})\right)-E_u^2\\
&=&0.
\end{eqnarray*}
It remains to prove that $\E\left(cov\left(\widehat{E}_{u,1},\widehat{E}_{u,2}|\bX_{-u}^{(1)},\bX_{-u}^{(2)}\right)\right)$ goes to $0$. By dominated convergence theorem, it suffices to show that for almost all $(\bx_{-u}^{(1)},\bx_{-u}^{(2)})$, 
\begin{equation}\label{cov_conditionnelle12}
cov\left(\widehat{E}_{u,1},\widehat{E}_{u,2}|\bX_{-u}^{(1)}=\bx_{-u}^{(1)},\bX_{-u}^{(2)}=\bx_{-u}^{(2)}\right)\underset{N\rightarrow+\infty}{\longrightarrow}0.
\end{equation}
From now on, we aim to proving Equation \eqref{cov_conditionnelle12}.

First, we want to prove Equation \eqref{cov_conditionnelle12} for $\bx_{-u}^{(1)}\neq \bx_{-u}^{(2)}$. Using dominated convergence theorem and Proposition \ref{prop_asymp_sans_bias}, it will suffice to show that (conditionally to $\bX_{-u}^{(1)}=\bx_{-u}^{(1)},\;\bX_{-u}^{(2)}=\bx_{-u}^{(2)}$), for almost all $((\bx_{-u}^{(n)})_{n\geq 3},\bh,\bh')$,
$$
\E\left(\widehat{E}_{u,1},\widehat{E}_{u,2}|(\bX_{-u}^{(n)})_n=(\bx_{-u}^{(n)})_n,\bk=\bh,\bk'=\bh'\right)\underset{N\rightarrow+\infty}{\longrightarrow}\V\left(Y|\bX_{-u}=\bx_{-u}^{(1)}\right)\V\left(Y|\bX_{-u}=\bx_{-u}^{(2)} \right).
$$ 
Let 
$$
A:=\left\{\left((\bx_{-u}^{(n)})_n,\bh,\bh'\right)|\;\bx_{-u}^{h_N(N_1)}\underset{N\rightarrow+\infty}{\longrightarrow} \bx_{-u}^{(1)},\;\bx_{-u}^{h'_N(N_1)}\underset{N\rightarrow+\infty}{\longrightarrow} \bx_{-u}^{(2)} \right\}.
$$
The set $A$ has probability 1 thanks to Lemma \ref{lm_conv_voisin}. Let $\left((\bx_{-u}^{(n)})_n,\bh,\bh'\right) \in A$ be such that $\bx_{-u}^{(1)}\neq \bx_{-u}^{(2)}$ and let $\delta:=d_{-u}(\bx_{-u}^{(1)},\bx_{-u}^{(2)})\slash 2$. There exists $N_1$ such that for all $N\geq N_1$, 
$$
d_{-u}\left( \bx_{-u}^{(1)},\bx_{-u}^{(h_N(N_I))}\right)< \frac{\delta}{2},\;\;d_{-u}\left( \bx_{-u}^{(2)},\bx_{-u}^{(h_N'(N_I))}\right)< \frac{\delta}{2}.
$$
Thus, for all $N \geq N_1$,
\begin{eqnarray*}
& & E(\widehat{E}_{u,1}\widehat{E}_{u,2}|(\bX_{-u}^{(n)})_n=(\bx_{-u}^{(n)})_n,\bk=\bh,\bk'=\bh')\\
&=&\E\left[\Phi\left(\left.(\bX^{k_N(i)})_{i\leq N_I}\right)\Phi\left((\bX^{k'_N(i)})_{i\leq N_I}\right)\right|(\bX_{-u}^{(n)})_n=(\bx_{-u}^{(n)})_n,\bk=\bh,\bk'=\bh'\right]\\
&=&\E\left[\Phi\left(\left.(\bX^{k_N(i)})_{i\leq N_I}\right)\Phi\left((\bX^{k'_N(i)})_{i\leq N_I}\right)\right|(\bX_{-u}^{(n)})_n=(\bx_{-u}^{(n)})_n,\bk_N=\bh_N,\bk'_N=\bh'_N\right]\\
&=&\E\left[\Phi\left(\left.(\bX^{h_N(i)})_{i\leq N_I}\right)\Phi\left((\bX^{k'_N(i)})_{i\leq N_I}\right)\right|(\bX_{-u}^{(n)})_n=(\bx_{-u}^{(n)})_n,\bk'_N=\bh'_N\right]\\
&=&\E\left[\Phi\left(\left.(\bX^{h_N(i)})_{i\leq N_I}\right)\Phi\left((\bX^{h'_N(i)})_{i\leq N_I}\right)\right|(\bX_{-u}^{(n)})_n=(\bx_{-u}^{(n)})_n\right]\\
&=&\E\left[\Phi\left(\left. (\bx_{-u}^{h_N(i)})_{i\leq N_I},(\bX_{u}^{h_N(i)})_{i\leq N_I}\right)\Phi\left((\bx_{-u}^{h'_N(i)})_{i\leq N_I},(\bX_{u}^{h'_N(i)})_{i\leq N_I}\right)\right|(\bX_{-u}^{(n)})_n=(\bx_{-u}^{(n)})_n\right]\\
&=&\E\left[\Phi\left(\left. (\bx_{-u}^{h_N(i)})_{i\leq N_I},(\bX_{u}^{h_N(i)})_{i\leq N_I}\right)\right|(\bX_{-u}^{(n)})_n=(\bx_{-u}^{(n)})_n\right]\\
& &\E\left[\Phi\left(\left. (\bx_{-u}^{h'_N(i)})_{i\leq N_I},(\bX_{u}^{h'_N(i)})_{i\leq N_I}\right)\right|(\bX_{-u}^{(n)})_n=(\bx_{-u}^{(n)})_n\right]\\
&=&\E\left[\left. \widehat{E}_{u,1}\right|(\bX_{-u}^{(n)})_n=(\bx_{-u}^{(n)})_n,\bk=\bh\right]\E\left[\left.\widehat{E}_{u,2}\right|(\bX_{-u}^{(n)})_n=(\bx_{-u}^{(n)})_n,\bk'=\bh'\right]\\
& \underset{N\rightarrow + \infty}{\longrightarrow}&\V\left[\left. Y\right| \bX_{-u}=\bx_{-u}^{(1)} \right]\V\left[\left. Y\right| \bX_{-u}=\bx_{-u}^{(2)}\right],
\end{eqnarray*}
thanks to Proposition \ref{prop_asymp_sans_bias}.\\

Assume now that $\bX_{-u}^{(1)}=\bX_{-u}^{(2)}=\bx_{-u}$. We can assume without lost of generality that $\PP(\bX_{-u}=\bx_{-u})>0$ because if we write $H:=\{\bx_{-u}|\; \PP(\bX_{-u}=\bx_{-u})=0\}$, we have $\PP(\bX_{-u}^{(1)}=\bX_{-u}^{(2)}\in H)=0$. We have to show that
$$
\E \left(\widehat{E}_{u,1}\widehat{E}_{u,2}|\bX_{-u}^{(1)}=\bX_{-u}^{(2)}=\bx_{-u} \right)-\V(Y|\bX_{-u}=\bx_{-u})^2 \underset{N\rightarrow+\infty}{\longrightarrow}0.
$$
Let $\varepsilon>0$.\\

Let $M_N$ the number of observations which are equal to $\bx_{-u}$,
$$
M_N:=\# \{n\leq N\; :\;\bX_{-u}^{(n)}=\bx_{-u}\},
$$
and let $H_N$ be the number of nearest neighbours (up to $N_I$-nearest) shared by $\bX_{-u}^{(1)}$ and $\bX_{-u}^{(2)}$,
$$
H_N:=\# \left[ \{k_N(i)\; :\;i \leq N_I\}\cap \{k'_N(i)\; :\;i \leq N_I\}\right].
$$
If $M_n=m\geq 2N_I,\;\bX_{-u}^{(1)}=\bx_{-u}=\bX_{-u}^{(2)}$, then the $N_I$-nearest neighbours $\bk_N$ of $\bX_{-u}^{(1)}$ and $\bk_N'$ of $\bX_{-u}^{(2)}$ are independent and are samples of uniformly distributed variables on the same set of cardinal $m$, without replacement. Thus,
\begin{eqnarray*}
& & \PP(H_N=0|M_N=m,\bX_{-u}^{(1)}=\bX_{-u}^{(2)}=\bx_{-u})\\
&=&\frac{\begin{pmatrix}
m-N_I\\N_I
\end{pmatrix}}{\begin{pmatrix}
m\\ N_I
\end{pmatrix}}\\
&=&\frac{(m-2N_I+1)(m-2N_I+2)...(m-N_I)}{(m-N_I+1)(m-N_I+2)...m}
\\
& \underset{m\rightarrow+\infty}{\longrightarrow} &1.
\end{eqnarray*}
Thus, there exists $m_1$ such that 
\begin{equation}\label{eq_alpha_n}
\alpha_{m_1}:=\PP(H_N=0|M_N\geq m_1,\bX_{-u}^{(1)}=\bX_{-u}^{(2)}=\bx_{-u})>1-\frac{\varepsilon}{5\|\Phi\|_\infty^2}.
\end{equation}
So,
\begin{eqnarray*}
& & \E \left(\widehat{E}_{u,1}\widehat{E}_{u,2}|\bX_{-u}^{(1)}=\bX_{-u}^{(2)}=\bx_{-u} \right)\\
&=&\E \left(\widehat{E}_{u,1}\widehat{E}_{u,2}|\bX_{-u}^{(1)}=\bX_{-u}^{(2)}=\bx_{-u} ,M_N< m_1\right)\PP(M_N< m_1|\bX_{-u}^{(1)}=\bX_{-u}^{(2)}=\bx_{-u})\\
&& +\E \left(\widehat{E}_{u,1}\widehat{E}_{u,2}|\bX_{-u}^{(1)}=\bX_{-u}^{(2)}=\bx_{-u} ,M_N\geq m_1\right)\PP(M_N\geq m_1|\bX_{-u}^{(1)}=\bX_{-u}^{(2)}=\bx_{-u}).
\end{eqnarray*}
Let
$$
\beta_N:=\E \left(\widehat{E}_{u,1}\widehat{E}_{u,2}|\bX_{-u}^{(1)}=\bX_{-u}^{(2)}=\bx_{-u} ,M_N< m_1\right)\PP(M_N< m_1|\bX_{-u}^{(1)}=\bX_{-u}^{(2)}=\bx_{-u}).
$$
Conditionally to $\bX_{-u}^{(1)}=\bX_{-u}^{(2)}=\bx_{-u}$, we know that $M_N-2\sim\mathcal{B}\left(N-2,\PP(\bX_{-u}=\bx_{-u}) \right)$, the binomial distribution. Thus, there exists $N_1$ such that for all $N\geq N_1$,
\begin{equation}\label{eq_beta_n}
\PP\left( M_N <m_1 |\bX_{-u}^{(1)}=\bX_{-u}^{(2)}=\bx_{-u}\right)<\frac{\varepsilon}{5 \max(1,\|\Phi\|_\infty^2)},
\end{equation}
and so, for all $N \geq N_1$, $\beta_N< \varepsilon \slash 5$. Furthermore
\begin{eqnarray*}
& & \E \left(\widehat{E}_{u,1}\widehat{E}_{u,2}|\bX_{-u}^{(1)}=\bX_{-u}^{(2)}=\bx_{-u} ,M_N\geq m_1\right)\\
&=&\E \left(\widehat{E}_{u,1}\widehat{E}_{u,2}|\bX_{-u}^{(1)}=\bX_{-u}^{(2)}=\bx_{-u} ,M_N\geq m_1,H_N=0\right) \PP(H_N=0|\bX_{-u}^{(1)}=\bX_{-u}^{(2)}=\bx_{-u} ,M_N\geq m_1)\\
& & +\E \left(\widehat{E}_{u,1}\widehat{E}_{u,2}|\bX_{-u}^{(1)}=\bX_{-u}^{(2)}=\bx_{-u} ,M_N\geq m_1,H_N\geq 1\right) \PP(H_N\geq 1|\bX_{-u}^{(1)}=\bX_{-u}^{(2)}=\bx_{-u} ,M_N\geq m_1).
\end{eqnarray*}
Let
$$
\gamma_N:=\PP\left( M_N \geq m_1 |\bX_{-u}^{(1)}=\bX_{-u}^{(2)}=\bx_{-u}\right).
$$
Moreover, conditionally to $\bX_{-u}^{(1)}=\bX_{-u}^{(2)}=\bx_{-u},M_N\geq m_1$, $H_N=0$ implies that $\widehat{E}_{u,1} \ind \widehat{E}_{u,2}$ thanks to Lemma \ref{lm_conditionnellement_iid_2NI}.
\begin{lm}\label{lm_conditionnellement_iid_2NI}
Conditionally to $\bX_{-u}^{(1)}=\bX_{-u}^{(2)}=\bx_{-u},M_N\geq m_1$, $H_N=0$, the vector $\left( (\bX^{(k_N(i))})_{i\leq N_I},(\bX^{(k'_N(i))})_{i\leq N_I}\right)$ is composed of $2N_I$ i.i.d. random variables of distribution $\bX$ conditionally to $\bX_{-u}=\bx_{-u}$.
\end{lm}
\begin{proof}
We know that, conditionally to $\bX_{-u}^{(1)}=\bX_{-u}^{(2)}=\bx_{-u},M_N\geq m_1$, $H_N=0$, the vector $\left( (\bX_{-u}^{(k_N(i))})_{i\leq N_I},(\bX_{-u}^{(k'_N(i))})_{i\leq N_I}\right)$ is constant equal to $(\bx_{-u})_{i\leq 2N_I}$. It suffices to show that, conditionally to $\bX_{-u}^{(1)}=\bX_{-u}^{(2)}=\bx_{-u},M_N\geq m_1$, $H_N=0$, the vector $\left( (\bX_u^{(k_N(i))})_{i\leq N_I},(\bX_u^{(k'_N(i))})_{i\leq N_I}\right)$ is composed of $2N_I$ i.i.d. random variables of distribution $\bX$ conditionally to $\bX_{-u}=\bx_{-u}$.
Let $((\bx_{-u}^{(n)})_n,\bh_N,\bh'_N)$ such that $\bX_{-u}^{(1)}=\bX_{-u}^{(2)}=\bx_{-u},M_N\geq m_1$ and $H_N=0$. As $M_N\geq m_1\geq N_I$, for all $i\leq N_I$, we have $\bx_{-u}^{(k_N(i))}=\bx_{-u}=\bx_{-u}^{(k_N'(i))}$. As $H_N=0$, then, for all $i$ and $j$ smaller than $N_I$, $h_N(i)\neq h'_N(j)$. Thus, we have for any bounded Borel function $\phi$,
\begin{eqnarray*}
& & \E \left( \left. \phi \left[(\bX_u^{(k_N(i))})_{i\leq N_I},(\bX_u^{(k'_N(i))})_{i\leq N_I} \right] \right| (\bX_{-u}^{(n)})_n=(\bx_{-u}^{(n)})_n,\bk_N=\bh_N,\bk'_N=\bh'_N \right)\\
&=& \E \left( \left. \phi \left[(\bX_u^{(h_N(i))})_{i\leq N_I},(\bX_u^{(k'_N(i))})_{i\leq N_I} \right] \right| (\bX_{-u}^{(n)})_n=(\bx_{-u}^{(n)})_n,\bk'_N=\bh'_N \right)\\
&=& \E \left( \left. \phi \left[(\bX_u^{(h_N(i))})_{i\leq N_I},(\bX_u^{(h'_N(i))})_{i\leq N_I} \right] \right| (\bX_{-u}^{(n)})_n=(\bx_{-u}^{(n)})_n, \right)\\
&=& \E \left( \left. \phi \left[(\bX_u^{(h_N(i))})_{i\leq N_I},(\bX_u^{(h'_N(i))})_{i\leq N_I} \right] \right| (\bX_{-u}^{(h_N(i))})_{i\leq N_I}=(\bx_{-u})_{i\leq N_I},(\bX_{-u}^{(h'_N(i))})_{i\leq N_I}=(\bx_{-u})_{i\leq N_I} \right)\\
&=&\E \left( \left. \phi \left[(\bX_u^{(i)})_{i\leq N_I},(\bX_u^{(i+N_I)})_{i\leq N_I} \right] \right| (\bX_{-u}^{(i)})_{i\leq 2N_I}=(\bx_{-u})_{i\leq 2 N_I} \right).
\end{eqnarray*}
Thus,
\begin{eqnarray*}
& & \E \left( \left. \phi \left[(\bX_u^{(k_N(i))})_{i\leq N_I},(\bX_u^{(k'_N(i))})_{i\leq N_I} \right] \right|\bX_{-u}^{(1)}=\bX_{-u}^{(2)}=\bx_{-u},M_N\geq m_1,H_N=0 \right)\\
& =& \E \left\{\E \left( \left. \phi \left[(\bX_u^{(k_N(i))})_{i\leq N_I},(\bX_u^{(k'_N(i))})_{i\leq N_I} \right] \right|\bX_{-u}^{(1)}=\bX_{-u}^{(2)}=\bx_{-u},M_N\geq m_1,H_N=0, (\bX_{-u}^{(n)})_n,\bk,\bk' \right) \right\}\\
&=& \E \left\{ \E \left( \left. \phi \left[(\bX_u^{(i)})_{i\leq N_I},(\bX_u^{(i+N_I)})_{i\leq N_I} \right] \right| (\bX_{-u}^{(i)})_{i\leq 2N_I}=(\bx_{-u})_{i\leq 2N_I} \right) \right\}\\
&=&\E \left( \left. \phi \left[(\bX_u^{(i)})_{i\leq N_I},(\bX_u^{(i+N_I)})_{i\leq N_I} \right] \right| (\bX_{-u}^{(i)})_{i\leq 2N_I}=(\bx_{-u})_{i\leq 2N_I} \right),
\end{eqnarray*}
that concludes the proof of Lemma \ref{lm_conditionnellement_iid_2NI}.
\end{proof}
Thus
\begin{eqnarray*}
& & \E\left( \widehat{E}_{u,1} \widehat{E}_{u,2}|\bX_{-u}^{(1)}=\bX_{-u}^{(2)}=\bx_{-u},M_N\geq m_1,\;H_N=0\right)\\
&=&\E\left( \widehat{E}_{u,1}|\bX_{-u}^{(1)}=\bX_{-u}^{(2)}=\bx_{-u},M_N\geq m_1,\;H_N=0\right)^2
\end{eqnarray*}
and so, using Proposition \ref{prop_asymp_sans_bias}, there exists $N_2$ such that for all $N\geq N_2$,
\begin{equation}\label{eq_dernier_terme}
\left|\E\left(\widehat{E}_{u,1} \widehat{E}_{u,2}|\bX_{-u}^{(1)}=\bX_{-u}^{(2)}=\bx_{-u},M_N\geq m_1,\;H_N=0\right)-\V(Y|\bX_{-u}=\bx_{-u})^{2} \right|<\frac{\varepsilon}{5}.
\end{equation}
Thus, for all $N\geq \max(N_1,N_2)$,
\begin{eqnarray*}
& &\left|\E \left(\widehat{E}_{u,1} \widehat{E}_{u,2}|\bX_{-u}^{(1)}=\bX_{-u}^{(2)}=\bx_{-u} \right)-\V(Y|\bX_{-u}=\bx_{-u})^2 \right|\\
& \leq & |\beta_N|+ \left|\gamma_N \E\left(\widehat{E}_{u,1} \widehat{E}_{u,2}|\bX_{-u}^{(1)}=\bX_{-u}^{(2)}=\bx_{-u},M_N\geq m_1,\;H_N\geq 1 \right)(1-\alpha_{m_1}) \right| \\
& & +\left|\gamma_N\alpha_{m_1}\E\left(\widehat{E}_{u,1} \widehat{E}_{u,2}|\bX_{-u}^{(1)}=\bX_{-u}^{(2)}=\bx_{-u},M_N\geq m_1,\;H_N=0 \right)-\V(Y|\bX_{-u}=\bx_{-u})^2\right|.
\end{eqnarray*}
The upper-bound is a sum of three terms. The first one is bounded by $\varepsilon\slash 5$ using Equation \ref{eq_beta_n} and the second one is bounded by $\varepsilon\slash 5$ using Equation \ref{eq_alpha_n}. For the last one,  we use that, for all $C\in \R$,
$$
\gamma_N \alpha_{m_1} C=(\gamma_N \alpha_{m_1}-1)C+C.
$$
Thus,
\begin{eqnarray*}
& &\left|\E \left(\widehat{E}_{u,1} \widehat{E}_{u,2}|\bX_{-u}^{(1)}=\bX_{-u}^{(2)}=\bx_{-u} \right)-\V(Y|\bX_{-u}=\bx_{-u})^2 \right|\\
& \leq & \frac{\varepsilon}{5}+\frac{\varepsilon}{5}+\left|\gamma_N \alpha_{m_1}-1 \right|\|\Phi\|_\infty^2+\left|\E\left( \widehat{E}_{u,1} \widehat{E}_{u,2}|\bX_{-u}^{(1)}=\bX_{-u}^{(2)}=\bx_{-u},M_N\geq m_1,\;H_N=0 \right)-\V(Y|\bX_{-u}=\bx_{-u})^2 \right|\\
& \leq & \frac{3\varepsilon}{5}+\left(|\gamma_N-1|\alpha_N+|\alpha_N-1| \right)\|\Phi \|_\infty^2\;\;\text{\;\;using Equation \ref{eq_dernier_terme}}\\
& \leq & \varepsilon,
\end{eqnarray*}
using Equation \ref{eq_beta_n} and Equation \ref{eq_alpha_n}. Finally, we proved that
$$
\E \left(\widehat{E}_{u,1} \widehat{E}_{u,2}|\bX_{-u}^{(1)}=\bX_{-u}^{(2)}=\bx_{-u} \right)-\V(Y|\bX_{-u}=\bx_{-u})^2 \underset{N\rightarrow+\infty}{\longrightarrow}0.
$$
Hence, Equation \eqref{cov_conditionnelle12} is proved and the proof of Proposition \ref{prop_covE12} is concluded.
\end{proof}
 
\begin{prop}\label{prop_conv_proba_biaise}
We have
\begin{equation}
\widehat{E}_{u}-\E\left( \widehat{E}_{u,1} \right)\overset{\PP}{\underset{\substack{N\rightarrow+\infty,\\ N_u\rightarrow+\infty}}{\longrightarrow}}0.
\end{equation}
\end{prop}

\begin{proof}
Let $\varepsilon >0$. By Chebyshev's inequality,
\begin{equation}
\PP\left( \left| \widehat{E}_{u}-\E\left( \widehat{E}_{u} \right) \right|>\varepsilon \right)\leq  \frac{\V(\widehat{E}_{u})}{\varepsilon^2}.
\end{equation}
If $(s(l))_{l\leq N_u}$ is a sample of uniformly distributed variables on $[1:N]$ with replacement, we remark that for all $i \neq j$,
\begin{eqnarray*}
&&  cov\left( \widehat{E}_{u,s(i)},\widehat{E}_{u,s(j)}\right) \\
&=&\E( \widehat{E}_{u,s(i)}\widehat{E}_{u,s(j)})-\E( \widehat{E}_{u,s(i)})\E(\widehat{E}_{u,s(j)})\\
&=& \E( \widehat{E}_{u,s(i)}\widehat{E}_{u,s(j)}|s(i)\neq s(j))\PP(s(i)\neq s(j))\\
&&+ \E( \widehat{E}_{u,s(i)}\widehat{E}_{u,s(j)}|s(i)= s(j))\PP(s(i)= s(j)) -\E( \widehat{E}_{u,s(i)})\E(\widehat{E}_{u,s(j)})\\
&=& \left[\E( \widehat{E}_{u,s(i)}\widehat{E}_{u,s(j)}|s(i)\neq s(j))-\E( \widehat{E}_{u,1})\E(\widehat{E}_{u,2})\right]\PP(s(i)\neq s(j))\\
&&+ \left[\E( \widehat{E}_{u,s(i)}\widehat{E}_{u,s(i)}|s(i)= s(j))-\E( \widehat{E}_{u,1})^2 \right]\PP(s(i)= s(j)) \\
&=& \left[\E( \widehat{E}_{u,1}\widehat{E}_{u,2}|s(i)=1, s(j)=2)-\E( \widehat{E}_{u,1})\E(\widehat{E}_{u,2})\right]\PP(s(i)\neq s(j))\\
&&+ \left[\E( \widehat{E}_{u,1}\widehat{E}_{u,1}|s(i)= s(j)=1)-\E( \widehat{E}_{u,1})^2 \right]\PP(s(i)= s(j)) \\
&=& cov\left( \widehat{E}_{u,1},\widehat{E}_{u,2}\right)\PP(s(i)\neq s(j)) +\V\left( \widehat{E}_{u,1}\right)\PP(s(i)= s(j)),
\end{eqnarray*}
thus
\begin{eqnarray*}
\V(\widehat{E}_{u})&=&\frac{1}{N_u^2}\sum_{i,j=1}^{N_u}cov\left( \widehat{E}_{u,s(i)},\widehat{E}_{u,s(j)}\right)\\
&=& \frac{1}{N_u^2}\sum_{i\neq j=1}^{N_u}cov\left( \widehat{E}_{u,1},\widehat{E}_{u,2}\right)\PP(s(i)\neq s(j))  \\
& & +\frac{1}{N_u^2}\sum_{i\neq j=1}^{N_u} \V\left( \widehat{E}_{u,1}\right)\PP(s(i)= s(j)) +\frac{1}{N_u^2}\sum_{i=1}^{N_u}\V\left( \widehat{E}_{u,s(i)}\right)   \\
& \leq & \frac{1}{N_u^2}\sum_{i\neq j=1}^{N_u}\left|cov\left( \widehat{E}_{u,1},\widehat{E}_{u,2}\right)\right|  \\
& & +\frac{1}{N_u^2}\sum_{i\neq j=1}^{N_u}   \V\left( \widehat{E}_{u,1}\right)\frac{1}{N}  +\frac{1}{N_u^2}\sum_{i=1}^{N_u}\V\left( \widehat{E}_{u,1}\right)   \\
&\leq &\left|cov\left( \widehat{E}_{u,1},\widehat{E}_{u,2}\right)\right|+\V\left( \widehat{E}_{u,1}\right)\left( \frac{1}{N}+\frac{1}{N_u}\right).
\end{eqnarray*}
If $(s(l))_{l\leq N_u}$ is a sample of uniformly distributed variables on $[1:N]$ without replacement, we have
\begin{eqnarray*}
\V(\widehat{E}_{u})&=&\frac{1}{N_u^2}\sum_{i,j=1}^{N_u}cov\left( \widehat{E}_{u,s(i)},\widehat{E}_{u,s(j)}\right)\\
&=& \frac{1}{N_u^2}\sum_{i\neq j=1}^{N_u}cov\left( \widehat{E}_{u,s(i)},\widehat{E}_{u,s(j)}\right)   +\frac{1}{N_u^2}\sum_{i=1}^{N_u}\V\left( \widehat{E}_{u,s(i)}\right)   \\
& = & \frac{N_u-1}{N_u}cov\left( \widehat{E}_{u,1},\widehat{E}_{u,2}\right)+\frac{1}{N_u}\V\left( \widehat{E}_{u,1}\right).
\end{eqnarray*}
In both cases (with or without replacement), thanks to Proposition \ref{prop_covE12}, we have
\begin{equation*}
\PP\left( \left| \widehat{E}_{u}-\E\left( \widehat{E}_{u} \right) \right|>\varepsilon \right)\underset{\substack{N\rightarrow+\infty,\\ N_u\rightarrow+\infty}}{\longrightarrow}0.
\end{equation*}
\end{proof}

Now, to prove Theorem \ref{thrm_consis_MC}, we only have to use Proposition \ref{prop_asymp_sans_bias} (which can be applied thanks to Lemma \ref{lm_continuite_loi_condi}) and Proposition \ref{prop_conv_proba_biaise}.

\subsection{Proof for rate of convergence: Theorems \ref{thrm_vitesse_MC} and \ref{thrm_vitesse_PF}}\label{section_vitesse}

We want to prove Theorems \ref{thrm_vitesse_MC} and \ref{thrm_vitesse_PF} about the rate of convergence of the double Monte-Carlo and Pick-and-Freeze estimators. We have to add some notation. We will write $C_{\sup}$ for a generic non-negative finite constant (depending only on $u$, $f$ and the distribution of $\bX$). The actual value of $C_{\sup}$ is of no interest and can change in the same sequence of equations. Similarly, we will write $C_{\inf}$ a generic strictly positive constant. We will write $C_{\sup}(\varepsilon)$ for a generic non-negative finite constant depending only on $\varepsilon$, $u$, $f$ and the distribution of $\bX$.

Recall that for all $i$, $\cX_i$ is a compact subset of $\R$ and that $f$ is $\mathcal{C}^1$. Moreover recall that $\bX$ has a probability density $f_\bX$ with respect to $\lambda_p$ (the Lebesgue measure on $\R^p$) such that $\lambda_p$-a.e., we have $0<C_{\inf}\leq f_\bX \leq C_{\sup}$, and such that $f_\bX$ is Lipschitz continuous. 

Note that with these assumptions, $\Phi$ is $\mathcal{C}^1$ on the compact set $\cX$ and so Lipschitz continuous. For all $n$, we will write $d$ for the euclidean distance on $\R^n$ (for any value of $n$) and $B(\bx,r)$ for the open ball of radius $r$ and center $\bx$ in $\R^n$. We also let $\mathcal{S}(\bx,r)$ be the sphere of center $\bx$ and radius $r$.\\

Remark that
\begin{eqnarray*}
& & \PP\left(d(\bX_{-u}^{(1)},\bX_{-u}^{(2)})=d(\bX_{-u}^{(1)},\bX_{-u}^{(3)})\right) \\
&=&\int_{\cX_{-u}^2}\PP\left(d(\bx_{-u}^{(1)},\bx_{-u}^{(2)})=d(\bx_{-u}^{(1)},\bX_{-u}^{(3)})\right) d \PP_{\bX_{-u}}^{\otimes 2}(\bx_{-u}^{(1)},\bx_{-u}^{(2)})\\
&\leq & C_{\sup} \int_{\cX_{-u}^2} \lambda_{|-u|}\left( \mathcal{S}(\bx_{-u}^{(1)},d(\bx_{-u}^{(1)},\bx_{-u}^{(2)})) \right)  d \PP_{\bX_{-u}}^{\otimes 2}(\bx_{-u}^{(1)},\bx_{-u}^{(2)})\\
&=&0,
\end{eqnarray*}
because the Lebesgue measure of the sphere $\mathcal{S}(\bx_{-u}^{(1)},d(\bx_{-u}^{(1)},\bx_{-u}^{(2)})) $ is zero. Thus, almost everywhere, for all $l$ and all $i \neq j$,
$$
d\left(\bX_{-u}^{(l)},\bX_{-u}^{(i)}\right)\neq d\left(\bX_{-u}^{(l)},\bX_{-u}^{(j)}\right).
$$
Thus, the indices of the nearest neighbours $(k_N(l,i))_{l,i}$ are constant random variables conditionally to $(\bX_{-u}^{(n)})_n$ or to $(\bX_{-u}^{(n)})_{n\leq N}$. In particular, for all $N$ and $l$, $k_N(l,1)=l$. Thanks to Doob-Dynkin lemma, we can write, abusing notation, $k_N(l,i)(\omega)=k_N(l,i)[(\bX_{-u}^{(n)}(\omega))_n]=k_N(l,i)[(\bX_{-u}^{(n)}(\omega))_{n\leq N}]$. To simplify notation, let us write $k_N(i):=k_N(1,i)$ (the index of one $i$-th neighbour of $\bX_{-u}^{(1)}$) and $k'_N(i):=k_N(2,i)$ (the index of one $i$-th neighbour of $\bX_{-u}^{(2)}$).

\begin{rmk}
We can prove the rate of convergence in a more general framework than the Euclidean space with the Lebesgue measure. It suffices to have a compact set $\cX$ with a dominating finite measure $\mu=\bigotimes \mu_i$ such that for $\mu_i$-almost all $\bx_i\in \cX_i$ and for all $\delta>0$, 
$$
C_{\inf}\delta \leq  \mu_i(B(\bx_i,\delta))= \mu_i(\overline{B}(\bx_i,\delta))\leq C_{\sup} \delta. 
$$
\end{rmk}

We prove Theorems \ref{thrm_vitesse_MC} and \ref{thrm_vitesse_PF} step by step.

\begin{lm}\label{lm_prod}
Assume that $(a_i)_i$ and $(b_i)_i$ are sequences such that for all $i$, $|a_i|\leq M$, $|b_i|\leq M$ and $|a_i-b_i| \leq \varepsilon$. Then, for all $N\in \N^*$
$$
\left|\prod_{i=1}^N a_i-\prod_{i=1}^{N} b_i \right|\leq NM^{N-1}\varepsilon.
$$
\end{lm}

\begin{proof}
By induction.
\end{proof}

\begin{lm}\label{lm_condi1}
If for all $i\leq N$, $d(\bx_{-u}^{(i)},\by_{-u}^{(i)})<\varepsilon$, then, for all $(\ba_{-u}^{(i)})_{i\leq N_I} \in \cX_{-u}^{N_I}$, 
\begin{eqnarray*}
&\left|\E\left[\left. \Phi\left((\ba_{-u}^{(i)})_{i\leq N_I},(\bX_u^{(i)})_{i\leq N_I}\right) \right|(\bX_{-u}^{(i)})_{i\leq N_I}=(\bx_{-u}^{(i)})_{i\leq N_I}\right]\right.&\\
-& \left.  \E\left[\left. \Phi\left((\ba_{-u}^{(i)})_{i\leq N_I},(\bX_u^{(i)})_{i\leq N_I}\right) \right|(\bX_{-u}^{(i)})_{i\leq N_I}=(\by_{-u}^{(i)})_{i\leq N_I}\right]\right|&\leq C_{\sup} \varepsilon.
\end{eqnarray*}
\end{lm}

\begin{proof}

\begin{eqnarray*}
& &\left|\E\left[\left. \Phi\left((\ba_{-u}^{(i)})_{i\leq N_I},(\bX_u^{(i)})_{i\leq N_I}\right) \right|(\bX_{-u}^{(i)})_{i\leq N_I}=(\bx_{-u}^{(i)})_{i\leq N_I}\right]\right.\\
& & - \left.  \E\left[\left. \Phi\left((\ba_{-u}^{(i)})_{i\leq N_I},(\bX_u^{(i)})_{i\leq N_I}\right) \right|(\bX_{-u}^{(i)})_{i\leq N_I}=(\by_{-u}^{(i)})_{i\leq N_I}\right]\right|\\
&=&\left|\int_{\cX_u^{N_I}} \Phi((\ba_{-u}^{(i)})_{i\leq N_I},(\bx_u^{(i)})_{i\leq N_I})\left(f_{(\bX_u^{(i)})_{i\leq N_I}|(\bX_{-u}^{(i)})_{i\leq N_I}=(\bx_{-u}^{(i)})_{i\leq N_I}}((\bx_u^{(i)})_{i\leq N_I})\right. \right.\\
 & & \left. \left. -f_{(\bX_u^{(i)})_{i\leq N_I}|(\bX_{-u}^{(i)})_{i\leq N_I}=(\by_{-u}^{(i)})_{i\leq N_I}}((\bx_u^{(i)})_{i\leq k}) \right) d ((\bx_u^{(i)})_{i\leq N_I})\right|\\
 & \leq &  C_{\sup}\int_{\cX_u^{N_I}}\left|\prod_{i=1}^{N_I}f_{\bX_u|\bX_{-u}=\bx_{-u}^{(i)}}(\bx_u^{(i)}) -  \prod_{i=1}^{N_I}f_{\bX_u|\bX_{-u}=\by_{-u}^{(i)}}(\bx_u^{(i)}) \right|  d((\bx_u^{(i)})_{i\leq N_I}).
\end{eqnarray*}
We know that, 
\begin{eqnarray*}
& & \left| f_{\bX_u|\bX_{-u}=\bx_{-u}}(\bx_u)-f_{\bX_u|\bX_{-u}=\by_{-u}}(\bx_u) \right| \\
& \leq &   \left| \frac{f_\bX(\bx_u,\bx_{-u})}{\int_{\cX_{u}} f_\bX(\bx_u',\bx_{-u})d(\bx_{u}')}-\frac{f_\bX(\bx_u,\by_{-u})}{\int_{\cX_{u}} f_\bX(\bx_u',\by_{-u})d(\bx_{u}')} \right|\\
& \leq & \frac{1}{\int_{\cX_{u}} f_\bX(\bx_u',\bx_{-u})d(\bx_{u}')}\left| f_\bX(\bx_u,\bx_{-u})-f_\bX(\bx_u,\by_{-u}) \right|\\
& & + f_\bX(\bx_u,\by_{-u})\left| \frac{1}{\int_{\cX_{u}} f_\bX(\bx_u',\bx_{-u})d(\bx_{u}')}-\frac{1}{\int_{\cX_{u}} f_\bX(\bx_u',\by_{-u})d(\bx_{u}')} \right|\\
& \leq & C_{\sup}\left| f_\bX(\bx_u,\bx_{-u})-f_\bX(\bx_u,\by_{-u}) \right|  + C_{\sup}\left| f_\bX(\bx_u,\bx_{-u})-f_\bX(\bx_u,\by_{-u}) \right|\\
& \leq & C_{\sup} d(\bx_{-u},\by_{-u}).
\end{eqnarray*}
Thus, for all $i\in [1:N_i]$ and for all $\bx_u^{(i)}$,
$$
\left| f_{\bX_u|\bX_{-u}=\bx_{-u}^{(i)}}(\bx_u^{(i)})-f_{\bX_u|\bX_{-u}=\by_{-u}^{(i)}}(\bx_u^{(i)}) \right| \leq C_{\sup}\varepsilon.
$$
Thus, using Lemma \ref{lm_prod},
\begin{align*}
&\left|\E\left[\left. \Phi\left((\ba_{-u}^{(i)})_{i\leq N_I},(\bX_u^{(i)})_{i\leq N_I}\right) \right|(\bX_{-u}^{(i)})_{i\leq N_I}=(\bx_{-u}^{(i)})_{i\leq N_I}\right]\right.&\\
-& \left.  \E\left[\left. \Phi\left((\ba_{-u}^{(i)})_{i\leq N_I},(\bX_u^{(i)})_{i\leq N_I}\right) \right|(\bX_{-u}^{(i)})_{i\leq N_I}=(\by_{-u}^{(i)})_{i\leq N_I}\right]\right|\leq C_{\sup} \varepsilon.
\end{align*}
\end{proof}

\begin{lm}\label{lm_condi2}
If for all $i$, $d(\bx_{-u}^{(i)},\by_{-u}^{(i)})<\varepsilon$, then
\begin{eqnarray*}
&\left|\E\left[\left. \Phi\left((\bx_{-u}^{(i)})_{i\leq N_I},(\bX_u^{(i)})_{i\leq N_I}\right) \right|(\bX_{-u}^{(i)})_{i\leq N_I}=(\bx_{-u}^{(i)})_{i\leq N_I}\right]\right.&\\
-& \left.  \E\left[\left. \Phi\left((\by_{-u}^{(i)})_{i\leq N_I},(\bX_u^{(i)})_{i\leq N_I}\right) \right|(\bX_{-u}^{(i)})_{i\leq N_I}=(\by_{-u}^{(i)})_{i\leq N_I}\right]\right|&\leq C_{\sup} \varepsilon.
\end{eqnarray*}
\end{lm}

\begin{proof}
\begin{eqnarray*}
& &\left|\E\left[\left. \Phi\left((\bx_{-u}^{(i)})_{i\leq N_I},(\bX_u^{(i)})_{i\leq N_I}\right) \right|(\bX_{-u}^{(i)})_{i\leq N_I}=(\bx_{-u}^{(i)})_{i\leq N_I}\right]\right.\\
&& - \left.  \E\left[\left. \Phi\left((\by_{-u}^{(i)})_{i\leq N_I},(\bX_u^{(i)})_{i\leq N_I}\right) \right|(\bX_{-u}^{(i)})_{i\leq N_I}=(\by_{-u}^{(i)})_{i\leq N_I}\right]\right|\\
 & \leq &\left|\E\left[\left. \Phi\left((\bx_{-u}^{(i)})_{i\leq N_I},(\bX_u^{(i)})_{i\leq N_I}\right) \right|(\bX_{-u}^{(i)})_{i\leq N_I}=(\bx_{-u}^{(i)})_{i\leq N_I}\right]\right.\\
&& - \left.  \E\left[\left. \Phi\left((\bx_{-u}^{(i)})_{i\leq N_I},(\bX_u^{(i)})_{i\leq N_I}\right) \right|(\bX_{-u}^{(i)})_{i\leq N_I}=(\by_{-u}^{(i)})_{i\leq N_I}\right]\right|\\
 & & + \left| \E\left[\Phi\left((\bx_{-u}^{(i)})_{i\leq N_I},(\bX_u^{(i)})_{i\leq N_I}\right)-\Phi\left((\by_{-u}^{(i)})_{i\leq N_I},(\bX_u^{(i)})_{i\leq N_I}\right)|(\bX_{-u}^{(i)})_{i\leq N_I}=(\by_{-u}^{(i)})_{i\leq N_I}\right]\right|\\
 & \leq & C_{\sup}\varepsilon +C_{\sup}\varepsilon,
\end{eqnarray*}
using Lemma \ref{lm_condi1} and using that $\Phi$ is Lipschitz continuous on $\cX$.
\end{proof}

\begin{lm}\label{lm_proba_voisin_loin}
There exists $C_{\sup}<+\infty$ such that for all $a>0$,
\begin{equation}
\PP\left(\left. d\left( \bX_{-u}^{(1)},\bX_{-u}^{(k_N(N_I))}\right) \geq a\right| \bX_{-u}^{(1)}\right) \leq  C_{\sup}N^{N_I}(1-C_{\inf} a^{|-u|})^{N-N_I}.
\end{equation}
\end{lm}

\begin{proof}
Let $K(a):=\# \{n\in [2:N],\;d(\bX_{-u}^{(1)},\bX_{-u}^{(n)})<a\}$. Conditionally to $\bX_{-u}^{(1)}$, $K(a)\sim \mathcal{B}(N-1,p(a,\bX_{-u}^{(1)}))$, writing $p(a,\bX_{-u}^{(1)}):=\PP(d(\bX_{-u}^{(1)},\bX_{-u}^{(2)})<a|\bX_{-u}^{(1)})$. Thus,
\begin{eqnarray*}
&&\PP\left(\left. d\left( \bX_{-u}^{(1)},\bX_{-u}^{(k_N(N_I))}\right) \geq a\right| \bX_{-u}^{(1)}\right)\\
&=& \PP\left(\left.K(a)\leq N_I-1\right| \bX_{-u}^{(1)} \right)\\
&=&\sum_{k=0}^{N_I-1}\begin{pmatrix}
N-1\\ k
\end{pmatrix}p(a,\bX_{-u}^{(1)})^k(1-p(a,\bX_{-u}^{(1)}))^{N-1-k}\\
& \leq & N_I\begin{pmatrix}
N-1\\ N_I-1
\end{pmatrix}(1-p(a,\bX_{-u}^{(1)}))^{N-N_I}\\
& \leq & C_{\sup}N^{N_I}(1-p(a,\bX_{-u}^{(1)}))^{N-N_I}.
\end{eqnarray*}
We know that
\begin{eqnarray*}
p(a,\bX_{-u}^{(1)})&=&\int_{B(\bX_{-u}^{(1)},a)} f_{\bX_{-u}}(\bx_{-u})d\bx_{-u}\\
& \geq & C_{\inf} \lambda_{|-u|}\left(B(\bX_{-u}^{(1)},a)\right)\\
& \geq & C_{\inf} a^{|-u|}.
\end{eqnarray*}
Thus
\begin{equation}
\PP\left(\left. d\left( \bX_{-u}^{(1)},\bX_{-u}^{(k_N(N_I))}\right) \geq a\right| \bX_{-u}^{(1)}\right) \leq  C_{\sup}N^{N_I}(1-C_{\inf} a^{|-u|})^{N-N_I}.
\end{equation}
\end{proof}

\begin{rmk}
For the estimators $\widehat{V}_{u,PF}$, we choose only one nearest neighbour different from $\bX_{u}^{(1)}$ in $\widehat{V}_{u,1,PF}$, which is $\bX_u^{(k_N(2))}$. Thus, in the previous computation, we do not have the  $N^{N_I}$. Remark that this is also true for $\widehat{E}_{u,MC}$ taking $N_I=2$.
\end{rmk}

\begin{lm}\label{lm_Ceps}
For all $\varepsilon>0$, there exists $C_{\sup}(\varepsilon)$ such that
\begin{equation}
\E\left( d\left( \bX_{-u}^{(1)},\bX_{-u}^{(k_N(N_I))}\right) \right)\leq \frac{C_{\sup}(\varepsilon)}{N^{\frac{1}{p-|u|}-\varepsilon}},
\end{equation}
and for all $\bx_{-u}^{(1)}$,
\begin{equation}
\E\left( \left. d\left( \bX_{-u}^{(1)},\bX_{-u}^{(k_N(N_I))}\right) \right| \bX_{-u}^{(1)}=\bx_{-u}^{(1)} \right)\leq \frac{C_{\sup}(\varepsilon)}{N^{\frac{1}{p-|u|}-\varepsilon}}.
\end{equation}
\end{lm}

\begin{proof}
Using Lemma \ref{lm_proba_voisin_loin}, we have
\begin{eqnarray*}
& & \E\left( \left.(N-N_I)^{\frac{1}{|-u|}-\varepsilon} d\left( \bX_{-u}^{(1)},\bX_{-u}^{(k_N(N_I))}\right) \right| \bX_{-u}^{(1)}\right) \\
&=& \int_0^{+\infty} \PP \left( \left. (N-N_I)^{\frac{1}{|-u|}-\varepsilon} d\left( \bX_{-u}^{(1)},\bX_{-u}^{(k_N(N_I))}\right) >t \right| \bX_{-u}^{(1)} \right) dt\\
&\leq & 1+\int_1^{+\infty} \PP\left( \left. d\left( \bX_{-u}^{(1)},\bX_{-u}^{(k_N(N_I))}\right)>t (N-N_I)^{-\frac{1}{|-u|}+\varepsilon}\right| \bX_{-u}^{(1)}\right) dt\\
&=&1+ \frac{1}{|-u|} \int_1^{+\infty}s^{\frac{1}{|-u|}-1} \PP\left( \left. d\left( \bX_{-u}^{(1)},\bX_{-u}^{(k_N(N_I))}\right)>s^\frac{1}{|-u|} (N-N_I)^{-\frac{1}{|-u|}+\varepsilon}\right| \bX_{-u}^{(1)}\right) ds\\
& \leq & 1+  \frac{1}{|-u|} \int_1^{+\infty} C_{\sup}N^{N_I}(1-C_{\inf} s (N-N_I)^{|-u|\varepsilon-1})^{N-N_I} ds,
\end{eqnarray*}
and
\begin{eqnarray*}
(1-C_{\inf} s (N-N_I)^{|-u|\varepsilon-1})^{N-N_I}&=&\exp\left[(N-N_I)\ln \left(1-C_{\inf} s (N-N_I)^{|-u|\varepsilon-1}\right)\right]\\
& \leq & \exp\left[(N-N_I)\left(-C_{\inf} s (N-N_I)^{|-u|\varepsilon-1}\right)\right]\\
&=&  \exp(-C_{\inf} s(N-N_I)^{|-u|\varepsilon}).
\end{eqnarray*}
Thus,
\begin{eqnarray*}
& & \E\left( \left.(N-N_I)^{\frac{1}{|-u|}-\varepsilon} d\left( \bX_{-u}^{(1)},\bX_{-u}^{(k_N(N_I))}\right) \right| \bX_{-u}^{(1)}\right) \\
& \leq & 1+C_{\sup}\int_1^{+\infty}N^{N_I}  \exp(-C_{\inf} s(N-N_I)^{|-u|\varepsilon}) ds\\
& \leq & 1+C_{\sup}\left[ N^{N_I} \exp(-C_{\inf} \frac{1}{2} (N-N_I)^{|-u|\varepsilon}) \right] \int_1^{+\infty}\exp(-C_{\inf} \frac{s}{2}(N-N_I)^{|-u|\varepsilon}) ds\\
& \leq & 1+C_{\sup}(\varepsilon).
\end{eqnarray*}
Indeed, the values $ N^{N_I} \exp(-C_{\inf} \frac{1}{2} (N-N_I)^{|-u|\varepsilon}$ and $\int_1^{+\infty}\exp(-C_{\inf} \frac{s}{2}(N-N_I)^{|-u|\varepsilon}) ds$ go to $0$ when $N$ do $+\infty$. Thus
\begin{equation*}
\E\left( \left. d\left( \bX_{-u}^{(1)},\bX_{-u}^{(k_N(N_I))}\right)\right| \bX_{-u}^{(1)} \right)\leq \frac{1+C_{\sup}(\varepsilon)}{(N-N_I)^{\frac{1}{p-|u|}-\varepsilon}}\leq \frac{C_{\sup}(\varepsilon)}{N^{\frac{1}{p-|u|}-\varepsilon}}.
\end{equation*}
That concludes the proof of Lemma \ref{lm_Ceps}.
\end{proof}

\begin{rmk}\label{rmq}
For the estimators $\widehat{V}_{u,PF}$, we do not have the  $N^{N_I}$. Thus, we can choose $\varepsilon=0$ up to Proposition \ref{prop_vitesse_esperance}.
\end{rmk}

\begin{prop}\label{prop_vitesse_esperance}
For all $\varepsilon >0$, there exists $C_{\sup}(\varepsilon)$ such that
\begin{equation}
\left| \E\left( \widehat{E}_{u}\right)-E_u \right|\leq \frac{C_{\sup}(\varepsilon)}{N^{\frac{1}{p-|u|}-\varepsilon}}
\end{equation}
and for almost all $\bx_{-u}^{(1)}$,
\begin{equation}
\left| \E\left( \widehat{E}_{u,1}|\bX_{-u}^{(1)}=\bx_{-u}^{(1)}\right)-\V(Y|\bX_{-u}=\bx_{-u}^{(1)}) \right|\leq \frac{C_{\sup}(\varepsilon)}{N^{\frac{1}{p-|u|}-\varepsilon}}.
\end{equation}
\end{prop}

\begin{proof}
For almost all $(\bx_{-u}^{(n)})_{n}$, using the definition of the random variable $\bZ$ (in the proof of Proposition \ref{prop_asymp_sans_bias}) and using Lemma \ref{lm_loi_KNN},
\begin{eqnarray*}
& & \bigg| \E \left(\left. \Phi\left( (\bX_{-u}^{(k_N(i)[(\bX_{-u}^{(n)})_{n}])})_{i\leq N_I},(\bX_{u}^{(k_N(i)[(\bX_{-u}^{(n)})_{n}])})_{i\leq N_I} \right) \right| (\bX_{-u}^{(n)})_n=(\bx_{-u}^{(n)})_{n} \right) \\
& &  -  \E \left(\left. \Phi\left( \bZ \right) \right| \bX_{-u}^{(1)} =\bx_{-u}^{(1)} \right)  \bigg| \\
&=& \bigg| \E \left(\left. \Phi\left(  (\bx_{-u}^{(k_N(i)[(\bx_{-u}^{(n)})_{n}])})_{i\leq N_I},(\bX_{u}^{(i)})_{i\leq N_I} \right) \right| (\bX_{-u}^{(i)})_{i\leq N_I}= (\bx_{-u}^{(k_N(i)[(\bx_{-u}^{(n)})_{n}])})_{i\leq N_I} \right)  \\
& &  -  \E \left(\left. \Phi\left( (\bx_{-u}^{(1)})_{i\leq N_I},(\bX_{u}^{(i)})_{i\leq N_I} \right) \right| (\bX_{-u}^{(i)})_{i\leq N_I} =(\bx_{-u}^{(1)})_{i\leq N_I} \right)  \bigg| \\
& \leq & C_{\sup}d\left(\bx_{-u}^{(k_N(N_I)[(\bx_{-u}^{(n)})_{n}])},\bx_{-u}^{(1)}\right),
\end{eqnarray*}
thanks to Lemma \ref{lm_condi2}. Thus, using Lemma \ref{lm_Ceps}, for all $\varepsilon>0$,
\begin{align*}
\left| \E\left( \widehat{E}_{u,1}|\bX_{-u}^{(1)}=\bx_{-u}^{(1)}\right)-\V(Y|\bX_{-u}=\bx_{-u}^{(1)}) \right|& \leq  C_{\sup} \E\left(\left. d\left( \bX_{-u}^{(1)},\bX_{-u}^{(k_N(N_I))}\right) \right| \bX_{-u}^{(1)}=\bx_{-u}^{(1)} \right)\\
& \leq  C_{\sup}\frac{C_{\sup}(\varepsilon)}{N^{\frac{1}{p-|u|}-\varepsilon}}.
\end{align*}
\end{proof}

In the following, to simplify notation, we may write "$\bX_{-u}^{(1,2)}=\bx_{-u}^{(1,2)}$" for "$\bX_{-u}^{(1)}=\bx_{-u}^{(1)}$ and $\bX_{-u}^{(2)}=\bx_{-u}^{(2)}$".
\begin{lm}\label{lm_proba_gn}
For almost all $(\bx_{-u}^{(1)}, \bx_{-u}^{(2)})$ and for all $a\geq 0$, we have
$$
\PP\left( \left.d(\bx_{-u}^{(1)},\bX_{-u}^{(k_N(N_I))})\geq a \right|\bX_{-u}^{(1,2)}=\bx_{-u}^{(1,2)}\right)\leq \PP\left(\left. d(\bx_{-u}^{(1)},\bX_{-u}^{(k_{N-1}(N_I))})\geq a \right|\bX_{-u}^{(1)}=\bx_{-u}^{(1)} \right),
$$
and thus, integrating $a$ on $\R_+$,
$$
\E\left(\left. d\left( \bX_{-u}^{(k_N(N_I))},\bX_{-u}^{(1)} \right) \right|  \bX_{-u}^{(1,2)}=\bx_{-u}^{(1,2)} \right)\leq \E\left(\left. d\left( \bX_{-u}^{(k_{N-1}(N_I))},\bX_{-u}^{(1)} \right) \right|  \bX_{-u}^{(1)}=\bx_{-u}^{(1)}\right).
$$
\end{lm}

\begin{proof}
Let $g_N(i)$ be the index of the $i$-th nearest neighbour of $\bX_{-u}^{(1)}$ in $(\bX_{-u}^{(n)})_{n \in [1:N]\setminus \{2\}}$. For almost all $(\bx_{-u}^{(1)}, \bx_{-u}^{(2)})$, we have
\begin{eqnarray*}
& &\PP\left( \left.d(\bx_{-u}^{(1)},\bX_{-u}^{(k_N(N_I))})\geq a \right|\bX_{-u}^{(1,2)}=\bx_{-u}^{(1,2)}\right)\\
&=&\PP\left( \left.d(\bx_{-u}^{(1)},\bX_{-u}^{(k_N(N_I))})\geq a \right|\bX_{-u}^{(1,2)}=\bx_{-u}^{(1,2)}, d(\bx_{-u}^{(1)},\bx_{-u}^{(2)})> d(\bx_{-u}^{(1)},\bX_{-u}^{(g_N(N_I))})\right) \\
& & \PP \left( \left. d(\bx_{-u}^{(1)},\bx_{-u}^{(2)})> d(\bx_{-u}^{(1)},\bX_{-u}^{(g_N(N_I))}) \right|\bX_{-u}^{(1,2)}=\bx_{-u}^{(1,2)} \right)\\
& & +\PP\left( \left.d(\bx_{-u}^{(1)},\bX_{-u}^{(k_N(N_I))})\geq a \right|\bX_{-u}^{(1,2)}=\bx_{-u}^{(1,2)}, d(\bx_{-u}^{(1)},\bx_{-u}^{(2)})\leq d(\bx_{-u}^{(1)},\bX_{-u}^{(g_N(N_I))})\right) \\
& & \PP \left( \left. d(\bx_{-u}^{(1)},\bx_{-u}^{(2)})\leq d(\bx_{-u}^{(1)},\bX_{-u}^{(g_N(N_I))}) \right|\bX_{-u}^{(1,2)}=\bx_{-u}^{(1,2)} \right).
\end{eqnarray*}
Moreover, conditionally to $\bX_{-u}^{(1,2)}=\bx_{-u}^{(1,2)}$, if $d(\bx_{-u}^{(1)},\bx_{-u}^{(2)})> d(\bx_{-u}^{(1)},\bX_{-u}^{(g_N(N_I))})$, then the $N_I$-nearest neighbours of $\bX_{-u}^{(1)}$ do not change if we do not take into account $\bX_{-u}^{(2)}$. Thus
\begin{eqnarray*}
& & \PP\left( \left.d(\bx_{-u}^{(1)},\bX_{-u}^{(k_N(N_I))})\geq a \right|\bX_{-u}^{(1,2)}=\bx_{-u}^{(1,2)}, d(\bx_{-u}^{(1)},\bx_{-u}^{(2)})> d(\bx_{-u}^{(1)},\bX_{-u}^{(g_N(N_I))})\right) \\
&=& \PP\left( \left.d(\bx_{-u}^{(1)},\bX_{-u}^{(g_N(N_I))})\geq a \right|\bX_{-u}^{(1,2)}=\bx_{-u}^{(1,2)}, d(\bx_{-u}^{(1)},\bx_{-u}^{(2)})> d(\bx_{-u}^{(1)},\bX_{-u}^{(g_N(N_I))})\right) \\
&=& \PP\left( \left.d(\bx_{-u}^{(1)},\bX_{-u}^{(g_N(N_I))})\geq a \right|\bX_{-u}^{(1)}=\bx_{-u}^{(1)}, d(\bx_{-u}^{(1)},\bx_{-u}^{(2)})> d(\bx_{-u}^{(1)},\bX_{-u}^{(g_N(N_I))})\right).
\end{eqnarray*}
Similarly, conditionally to $\bX_{-u}^{(1,2)}=\bx_{-u}^{(1,2)}$, if $d(\bx_{-u}^{(1)},\bx_{-u}^{(2)})\leq d(\bx_{-u}^{(1)},\bX_{-u}^{(g_N(N_I))}$, then $\bx_{-u}^{(2)}$ is one of the $N_I$-nearest neighbours of $\bX_{-u}^{(1)}$. Thus 
\begin{eqnarray*}
& & \PP\left( \left.d(\bx_{-u}^{(1)},\bX_{-u}^{(k_N(N_I))})\geq a \right|\bX_{-u}^{(1,2)}=\bx_{-u}^{(1,2)}, d(\bx_{-u}^{(1)},\bx_{-u}^{(2)})\leq d(\bx_{-u}^{(1)},\bX_{-u}^{(g_N(N_I))})\right) \\
& \leq & \PP\left( \left.d(\bx_{-u}^{(1)},\bX_{-u}^{(g_N(N_I))})\geq a \right|\bX_{-u}^{(1,2)}=\bx_{-u}^{(1,2)}, d(\bx_{-u}^{(1)},\bx_{-u}^{(2)})\leq d(\bx_{-u}^{(1)},\bX_{-u}^{(g_N(N_I))})\right) \\
&=& \PP\left( \left.d(\bx_{-u}^{(1)},\bX_{-u}^{(g_N(N_I))})\geq a \right|\bX_{-u}^{(1)}=\bx_{-u}^{(1)}, d(\bx_{-u}^{(1)},\bx_{-u}^{(2)})\leq d(\bx_{-u}^{(1)},\bX_{-u}^{(g_N(N_I))})\right).
\end{eqnarray*}
Finally,
\begin{eqnarray*}
& &\PP\left( \left.d(\bx_{-u}^{(1)},\bX_{-u}^{(k_N(N_I))})\geq a \right|\bX_{-u}^{(1,2)}=\bx_{-u}^{(1,2)}\right)\\
&\leq &\PP\left( \left.d(\bx_{-u}^{(1)},\bX_{-u}^{(g_N(N_I))})\geq a \right|\bX_{-u}^{(1)}=\bx_{-u}^{(1)}, d(\bx_{-u}^{(1)},\bx_{-u}^{(2)})> d(\bx_{-u}^{(1)},\bX_{-u}^{(g_N(N_I))})\right) \\
& & \PP \left( \left. d(\bx_{-u}^{(1)},\bx_{-u}^{(2)})> d(\bx_{-u}^{(1)},\bX_{-u}^{(g_N(N_I))}) \right|\bX_{-u}^{(1)}=\bx_{-u}^{(1)} \right)\\
& & +\PP\left( \left.d(\bx_{-u}^{(1)},\bX_{-u}^{(g_N(N_I))})\geq a \right|\bX_{-u}^{(1)}=\bx_{-u}^{(1)}, d(\bx_{-u}^{(1)},\bx_{-u}^{(2)})\leq d(\bx_{-u}^{(1)},\bX_{-u}^{(g_N(N_I))})\right) \\
& & \PP \left( \left. d(\bx_{-u}^{(1)},\bx_{-u}^{(2)})\leq d(\bx_{-u}^{(1)},\bX_{-u}^{(g_N(N_I))}) \right|\bX_{-u}^{(1)}=\bx_{-u}^{(1)} \right)\\
&=&\PP\left( \left.d(\bx_{-u}^{(1)},\bX_{-u}^{(g_N(N_I))})\geq a \right|\bX_{-u}^{(1)}=\bx_{-u}^{(1)}\right)\\
&=&\PP\left( \left.d(\bx_{-u}^{(1)},\bX_{-u}^{(k_{N-1}(N_I))})\geq a \right|\bX_{-u}^{(1)}=\bx_{-u}^{(1)}\right),
\end{eqnarray*}
and we proved Lemma \ref{lm_proba_gn}.
\end{proof}

\begin{prop}\label{prop_vitesse_cov}
For all $\varepsilon >0$, there exists $C_{\sup}(\varepsilon)$ such that
\begin{equation}
\left| cov(\widehat{E}_{u,1},\widehat{E}_{u,2})\right| \leq \frac{C_{\sup}(\varepsilon)}{N^{\frac{1}{p-|u|}-\varepsilon}}.
\end{equation}
\end{prop}

\begin{proof}
We use the law of total covariance,
\begin{equation}\label{eq_vitesse_cov_totale}
cov(\widehat{E}_1,\widehat{E}_2)=\E\left[cov\left( \left. \widehat{E}_1,\widehat{E}_2\right| \bX_{-u}^{(1)},\bX_{-u}^{(2)}\right) \right]+cov \left[ \E\left( \widehat{E}_{u,1}|\bX_{-u}^{(1)},\bX_{-u}^{(2)}\right),\E\left( \widehat{E}_{u,2}|\bX_{-u}^{(1)},\bX_{-u}^{(2)}\right) \right].
\end{equation}
\underline{Part 1:} First, we will bound the second term of Equation \ref{eq_vitesse_cov_totale}. Thanks to Lemma \ref{lm_condi2}, we have
\begin{eqnarray*}
& &\left| \E\left( \widehat{E}_{u,1}|\bX_{-u}^{(1)}=\bx_{-u}^{(1)},\bX_{-u}^{(2)}=\bx_{-u}^{(2)} \right)-\V\left(Y|\bX_{-u}=\bx_{-u}^{(1)}) \right) \right| \\
& \leq &  \E\left\{\left| \E \left[ \left. \Phi\left( (\bX^{(k_N(i))})_{i\leq N_I} \right) \right| \bX_{-u}^{(1)}=\bx_{-u}^{(1)},\bX_{-u}^{(2)}=\bx_{-u}^{(2)},(\bX_{-u}^{(n)})_{n\geq 3} \right] -\V\left(Y|\bX_{-u}=\bx_{-u}^{(1)}) \right)\right| \right\}\\
& \leq & C_{\sup} \E\left(\left. d\left( \bX_{-u}^{(1)},\bX_{-u}^{(k_N(N_I))}\right) \right| \bX_{-u}^{(1)}=\bx_{-u}^{(1)},\bX_{-u}^{(2)}=\bx_{-u}^{(2)} \right)     \;\;\text{\;\; using Lemma \ref{lm_condi2}, }\\\\
& \leq & C_{\sup} \E\left(\left. d\left( \bX_{-u}^{(1)},\bX_{-u}^{(k_{N-1}(N_I))}\right) \right| \bX_{-u}^{(1)}=\bx_{-u}^{(1)}\right) \;\;\text{\;\; using Lemma \ref{lm_proba_gn}, }\\
& \leq & \frac{C_{\sup}(\varepsilon)}{(N-1)^{\frac{1}{p-|u|}-\varepsilon}} \;\;\text{\;\; using Lemma \ref{lm_Ceps}, }\\
& \leq & \frac{C_{\sup}(\varepsilon)}{N^{\frac{1}{p-|u|}-\varepsilon}}.
\end{eqnarray*}
Similarly,
$$
\left| \E\left( \widehat{E}_{u,2}|\bX_{-u}^{(1)}=\bx_{-u}^{(1)},\bX_{-u}^{(2)}=\bx_{-u}^{(2)} \right)-\V\left(Y|\bX_{-u}=\bx_{-u}^{(2)}) \right) \right| \leq  \frac{C_{\sup}(\varepsilon)}{N^{\frac{1}{p-|u|}-\varepsilon}}.
$$
Thus, using that $\Phi$ is bounded,
\begin{eqnarray*}
&\left| \E\left( \widehat{E}_{u,1}|\bX_{-u}^{(1)}=\bx_{-u}^{(1)},\bX_{-u}^{(2)}=\bx_{-u}^{(2)} \right)\E\left( \widehat{E}_{u,2}|\bX_{-u}^{(1)}=\bx_{-u}^{(1)},\bX_{-u}^{(2)}=\bx_{-u}^{(2)} \right)\right.&\\
&\left. -\V\left(Y|\bX_{-u}=\bx_{-u}^{(1)}) \right)\V\left(Y|\bX_{-u}=\bx_{-u}^{(2)}) \right) \right| \leq  \frac{C_{\sup}(\varepsilon)}{N^{\frac{1}{p-|u|}-\varepsilon}}.&
\end{eqnarray*}
Moreover, using Proposition \ref{prop_vitesse_esperance}, we have
\begin{eqnarray*}
&\left| \E\left( \widehat{E}_{u,1}|\bX_{-u}^{(1)}=\bx_{-u}^{(1)} \right)\E\left( \widehat{E}_{u,2}|\bX_{-u}^{(2)}=\bx_{-u}^{(2)} \right)\right.&\\
&\left. -\V\left(Y|\bX_{-u}=\bx_{-u}^{(1)}) \right)\V\left(Y|\bX_{-u}=\bx_{-u}^{(2)}) \right) \right| \leq  \frac{C_{\sup}(\varepsilon)}{N^{\frac{1}{p-|u|}-\varepsilon}}.&
\end{eqnarray*}
Thus,
\begin{eqnarray*}
\left| \E\left( \widehat{E}_{u,1}|\bX_{-u}^{(1)}=\bx_{-u}^{(1)},\bX_{-u}^{(2)}=\bx_{-u}^{(2)} \right)\E\left( \widehat{E}_{u,2}|\bX_{-u}^{(1)}=\bx_{-u}^{(1)},\bX_{-u}^{(2)}=\bx_{-u}^{(2)} \right)\right.&\\
\left. - \E\left( \widehat{E}_{u,1}|\bX_{-u}^{(1)}=\bx_{-u}^{(1)} \right)\E\left( \widehat{E}_{u,2}|\bX_{-u}^{(2)}=\bx_{-u}^{(2)} \right) \right| &\leq  \frac{C_{\sup}(\varepsilon)}{N^{\frac{1}{p-|u|}-\varepsilon}}.
\end{eqnarray*}
Finally,
\begin{eqnarray*}
& & \left| cov \left[ \E\left( \widehat{E}_{u,1}|\bX_{-u}^{(1)},\bX_{-u}^{(2)}\right),\E\left( \widehat{E}_{u,2}|\bX_{-u}^{(1)},\bX_{-u}^{(2)}\right) \right]  \right|\\
&=& \left| \E\left[  \E\left( \widehat{E}_{u,1}|\bX_{-u}^{(1)},\bX_{-u}^{(2)} \right)\E\left( \widehat{E}_{u,2}|\bX_{-u}^{(1)},\bX_{-u}^{(2)}\right) \right] - \E \left[ \E\left( \widehat{E}_{u,1}|\bX_{-u}^{(1)}\right)\E\left( \widehat{E}_{u,2}|\bX_{-u}^{(2)} \right)   \right]  \right|\\
& \leq & \E\left[\left|  \E\left( \widehat{E}_{u,1}|\bX_{-u}^{(1)},\bX_{-u}^{(2)} \right)\E\left( \widehat{E}_{u,2}|\bX_{-u}^{(1)},\bX_{-u}^{(2)}\right) -  \E\left( \widehat{E}_{u,1}|\bX_{-u}^{(1)}\right)\E\left( \widehat{E}_{u,2}|\bX_{-u}^{(2)} \right) \right|  \right]\\
& \leq & \frac{C_{\sup}(\varepsilon)}{N^{\frac{1}{p-|u|}-\varepsilon}}.
\end{eqnarray*}

\begin{rmk}
In this Part 1, we can choose $\varepsilon=0$ for the estimators $\widehat{V}_{u,PF}$ or for $\widehat{E}_{u,MC}$ if we take $N_I=2$.
\end{rmk}

\underline{Part 2:} Let $\varepsilon>0$. We will bound the first term of Equation \ref{eq_vitesse_cov_totale}: $\E\left[cov\left( \left.\widehat{E}_1,\widehat{E}_2\right| \bX_{-u}^{(1)},\bX_{-u}^{(2)}\right) \right]$.
We want to prove that
$$
\left| \int_{\cX_{-u}^2}\E\left(\widehat{E}_{u,1}\widehat{E}_{u,2}|\bX_{-u}^{(1,2)}=\bx_{-u}^{(1,2)}\right) - \E(\widehat{E}_{u,1}|\bX_{-u}^{(1,2)}=\bx_{-u}^{(1,2)})\E(\widehat{E}_{u,2}|\bX_{-u}^{(1,2)}=\bx_{-u}^{(1,2)})d \PP_{\bX_{-u}}^{ \otimes 2}(\bx_{-u}^{(1)},\bx_{-u}^{(2)})\right| \leq  \frac{C_{\sup}(\varepsilon)}{N^{1-\varepsilon}}.
$$

Let us write 
$$
l(\bx_{-u}^{(1)},\bx_{-u}^{(2)}):=\min\left(d(\bx_{-u}^{(1)},\bx_{-u}^{(2)})\slash 2,\;\frac{1}{N^{\frac{1}{|-u|}-\delta}} \right)
$$
where $\delta=\varepsilon\slash(4|-u|)$, and
\begin{eqnarray*}
G(\bx_{-u}^{(1)},\bx_{-u}^{(2)}):=\bigg\{(\bx_{-u}^{(n)})_{n\in [3:N]}| &  d(\bx_{-u}^{(1)},\bx_{-u}^{(k_{N}(N_I)[(\bx_{-u}^{(n)})_{n\leq N}])})< l(\bx_{-u}^{(1)},\bx_{-u}^{(2)}) ,  \\
& d(\bx_{-u}^{(2)},\bx_{-u}^{(k_{N}'(N_I)[(\bx_{-u}^{(n)})_{n\leq N}])})< l(\bx_{-u}^{(1)},\bx_{-u}^{(2)})  \bigg\}.
\end{eqnarray*}

\underline{Part 2.A:} We prove the following lemmas.
\begin{lm}\label{lm_voisin_xu1_xu2}
For all $\varepsilon>0$, there exists $C_{\sup}(\varepsilon)$ such that,
\begin{equation}
\int_{\cX_{-u}^2}\PP\left(\left. d(\bX_{-u}^{(1)},\bX_{-u}^{(k_{N-1}(N_I))})\geq d(\bx_{-u}^{(1)},\bx_{-u}^{(2)})\slash 2   \right| \bX_{-u}^{(1)}=\bx_{-u}^{(1)} \right)d\PP_{\bX_{-u}}^{\otimes 2}(\bx_{-u}^{(1)},\bx_{-u}^{(2)})\leq \frac{C_{\sup}(\varepsilon)}{N^{1-\varepsilon}}.
\end{equation}
\end{lm}
\begin{proof}
We divide $\cX_{-u}^{2}$ in $F_1:=\{(\bx_{-u}^{(1)},\bx_{-u}^{(2)}) \in \cX_{-u}^2,\;d(\bx_{-u}^{(1)},\bx_{-u}^{(2)}) <(N-N_I-1)^{\frac{-1+\varepsilon}{|-u|} }\}$ and $F_2:=\{(\bx_{-u}^{(1)},\bx_{-u}^{(2)}) \in \cX_{-u}^2,\;d(\bx_{-u}^{(1)},\bx_{-u}^{(2)}) \geq (N-N_I-1)^{\frac{-1+\varepsilon}{|-u|}} \}$.
\begin{eqnarray*}
& & \int_{F_1}\PP\left(\left. d(\bX_{-u}^{(1)},\bX_{-u}^{(k_{N-1}(N_I))})\geq d(\bx_{-u}^{(1)},\bx_{-u}^{(2)})\slash 2   \right| \bX_{-u}^{(1)}=\bx_{-u}^{(1)} \right)d\PP_{\bX_{-u}}^{\otimes 2}(\bx_{-u}^{(1)},\bx_{-u}^{(2)})\\
& \leq & C_{\sup}\lambda_{|-u|}^{\otimes 2}(F_1)\\
& \leq & C_{\sup}\int_{\cX_{-u}}\lambda_{|-u|}\left(B\left[\bx_{-u},(N-N_I-1)^{\frac{-1+\varepsilon}{|-u|}} \right] \right) d\bx_{-u}\\
& \leq & C_{\sup}\int_{\cX_{-u}}(N-N_I-1)^{\frac{-1+\varepsilon}{|-u|}|-u|}dx_{-u}\\
& \leq & C_{\sup}(N-N_I-1)^{-1+\varepsilon }\\
& \leq & \frac{C_{\sup}}{N^{1-\varepsilon}}.
\end{eqnarray*}
Furthermore, using Lemma \ref{lm_proba_voisin_loin}, we have
\begin{align*}
&  \int_{F_2}\PP\left(\left. d(\bX_{-u}^{(1)},\bX_{-u}^{(k_{N-1}(N_I))})\geq d(\bx_{-u}^{(1)},\bx_{-u}^{(2)})\slash 2   \right| \bX_{-u}^{(1)}=\bx_{-u}^{(1)} \right)d\PP_{\bX_{-u}}^{\otimes 2}(\bx_{-u}^{(1)},\bx_{-u}^{(2)})\\
 \leq & \int_{F_2}  C_{\sup}(N-1)^{N_I}(1-C_{\inf} d(\bx_{-u}^{(1)},\bx_{-u}^{(2)})^{|-u|})^{N-1-N_I}d\PP_{\bX_{-u}}^{\otimes 2}(\bx_{-u}^{(1)},\bx_{-u}^{(2)})\\
 \leq & \lambda_{|-u|}(\cX_{-u})^2C_{\sup}(N-1)^{N_I}(1-C_{\inf}(N-N_I-1)^{\frac{-1+\varepsilon}{|-u|}|-u|})^{N-1-N_I}\\
 \leq &C_{\sup} (N-1)^{N_I} (1-C_{\inf}(N-N_I-1)^{-1+\varepsilon})^{N-1-N_I}\\
 \leq &C_{\sup} (N-1)^{N_I} \exp\left[ (N-1-N_I)\ln\left(1-C_{\inf}(N-N_I-1)^{-1+\varepsilon}\right) \right]\\
 \leq &C_{\sup} (N-1)^{N_I}\exp\left[ -C_{\inf}(N-N_I-1)^{\varepsilon} +o((N-N_I-1)^{\varepsilon})\right]\\
\leq & \frac{C_{\sup}(\varepsilon)}{N^{1-\varepsilon}}.
\end{align*}
\end{proof}

\begin{rmk}
In Lemma \ref{lm_voisin_xu1_xu2}, we need $\varepsilon>0$ even for the Pick-and-Freeze estimators. That explains the rate of convergence when $|u|=1$ for the Pick-and-Freeze estimators.
\end{rmk}

\begin{lm}\label{lm_P(Gc)}
For all $\varepsilon>0$, there exists $C_{\sup}(\varepsilon)$ such that,
\begin{equation}
\int_{\cX_{-u}^2}\PP_{\bX_{-u}}^{\otimes (N-2)}(G(\bx_{-u}^{(1)},\bx_{-u}^{(2)})^c) d \PP_{\bX_{-u}}^{ \otimes 2}(\bx_{-u}^{(1)},\bx_{-u}^{(2)}) \leq \frac{C_{\sup}(\varepsilon)}{N^{1-\varepsilon}}.
\end{equation}
\end{lm}

\begin{proof}
Using Lemma \ref{lm_proba_voisin_loin}, we have
\begin{equation*}
\PP\left( d(\bX_{-u}^{(k_{N-1}(N_I))},\bx_{-u}^{(1)})\geq N^{-\frac{1}{|-u|}+\delta}|\bX_{-u}^{(1)} \right)\leq   C_{\sup}(N-1)^{N_I}(1-C_{\inf} N^{-1+\delta |-u|})^{N-1-N_I} ,
\end{equation*}
so
\begin{equation}\label{eq_proba_pas_boule}
\PP\left( d(\bX_{-u}^{(k_{N-1}(N_I))},\bx_{-u}^{(1)}) \geq N^{-\frac{1}{|-u|}+\delta}|\bX_{-u}^{(1)} \right) \leq  \frac{C_{\sup}(\varepsilon)}{N}.
\end{equation}
Thus, we have
\begin{eqnarray*}
& & \int_{\cX_{-u}^2}\PP_{\bX_{-u}}^{\otimes (N-2)}(G(\bx_{-u}^{(1)},\bx_{-u}^{(2)})^c) d \PP_{\bX_{-u}}^{ \otimes 2}(\bx_{-u}^{(1)},\bx_{-u}^{(2)})\\
& \leq & \int_{\cX_{-u}^2} \PP\left(\left. d(\bX_{-u}^{(1)},\bX_{-u}^{(k_N(N_I))})\geq d(\bx_{-u}^{(1)},\bx_{-u}^{(2)})\slash 2   \right| \bX_{-u}^{(1,2)}=\bx_{-u}^{(1,2)} \right)  d \PP_{\bX_{-u}}^{ \otimes 2}(\bx_{-u}^{(1)},\bx_{-u}^{(2)})\\
&  &+ \int_{\cX_{-u}^2} \PP\left(\left. d(\bX_{-u}^{(2)},\bX_{-u}^{(k'_N(N_I))})\geq d(\bx_{-u}^{(1)},\bx_{-u}^{(2)})\slash 2   \right| \bX_{-u}^{(1,2)}=\bx_{-u}^{(1,2)} \right)  d \PP_{\bX_{-u}}^{ \otimes 2}(\bx_{-u}^{(1)},\bx_{-u}^{(2)})\\ 
& &+ \int_{\cX_{-u}^2} \PP\left(\left. d(\bX_{-u}^{(1)},\bX_{-u}^{(k_N(N_I))})\geq N^{-\frac{1}{|-u|}+\delta} \right| \bX_{-u}^{(1,2)}=\bx_{-u}^{(1,2)} \right)  d \PP_{\bX_{-u}}^{ \otimes 2}(\bx_{-u}^{(1)},\bx_{-u}^{(2)})\\
&  &+ \int_{\cX_{-u}^2} \PP\left(\left. d(\bX_{-u}^{(2)},\bX_{-u}^{(k'_N(N_I))})\geq N^{-\frac{1}{|-u|}+\delta}   \right| \bX_{-u}^{(1,2)}=\bx_{-u}^{(1,2)} \right)  d \PP_{\bX_{-u}}^{ \otimes 2}(\bx_{-u}^{(1)},\bx_{-u}^{(2)})\\ 
& \leq & \int_{\cX_{-u}^2} \PP\left(\left. d(\bX_{-u}^{(1)},\bX_{-u}^{(k_{N-1}(N_I))})\geq d(\bx_{-u}^{(1)},\bx_{-u}^{(2)})\slash 2   \right| \bX_{-u}^{(1)}=\bx_{-u}^{(1)} \right)  d \PP_{\bX_{-u}}^{ \otimes 2}(\bx_{-u}^{(1)},\bx_{-u}^{(2)})\\
& & +\int_{\cX_{-u}^2} \PP\left(\left. d(\bX_{-u}^{(2)},\bX_{-u}^{(k'_{N-1}(N_I))})\geq d(\bx_{-u}^{(1)},\bx_{-u}^{(2)})\slash 2   \right| \bX_{-u}^{(2)}=\bx_{-u}^{(2)} \right)  d \PP_{\bX_{-u}}^{ \otimes 2}(\bx_{-u}^{(1)},\bx_{-u}^{(2)})\\
& & + \int_{\cX_{-u}^2} \PP\left(\left. d(\bX_{-u}^{(1)},\bX_{-u}^{(k_{N-1}(N_I))})\geq N^{-\frac{1}{|-u|}+\delta}   \right| \bX_{-u}^{(1)}=\bx_{-u}^{(1)} \right)  d \PP_{\bX_{-u}}^{ \otimes 2}(\bx_{-u}^{(1)},\bx_{-u}^{(2)})\\
& & +\int_{\cX_{-u}^2} \PP\left(\left. d(\bX_{-u}^{(2)},\bX_{-u}^{(k'_{N-1}(N_I))})\geq N^{-\frac{1}{|-u|}+\delta}   \right| \bX_{-u}^{(2)}=\bx_{-u}^{(2)} \right)  d \PP_{\bX_{-u}}^{ \otimes 2}(\bx_{-u}^{(1)},\bx_{-u}^{(2)}),
\end{eqnarray*}
and we conclude the proof of Lemma \ref{lm_P(Gc)} using Lemma \ref{lm_voisin_xu1_xu2} and Equation \ref{eq_proba_pas_boule}.
\end{proof}

For $i=1,2$, let $B_i$ be the ball of center $\bx_{-u}^{(i)}$ and of radius $l(\bx_{-u}^{(1)},\bx_{-u}^{(2)})$, let $p_i$ be the probability of $B_i$ and $N_i$ be the number of observations $(\bX_{-u}^{(n)})_{n\in [3:N]}$ in the ball $B_i$. Remark that
$$
p_i\leq \frac{C_{\sup}}{N^{1-\delta|-u|}}.
$$
We have the two following lemmas.

\begin{lm}\label{loi_Ni}
Conditionally to $\bX_{-u}^{(1,2)}=\bx_{-u}^{(1,2)}$, the random variable $N_i$ is binomial $\mathcal{B}(N-2,p_i)$.\\
Conditionally to $\bX_{-u}^{(1,2)}=\bx_{-u}^{(1,2)}, N_j=n_j$, the random variable $N_i$ is binomial $\mathcal{B}(N-2-n_j,p_i(1-p_j)^{-1})$.
\end{lm}

\begin{proof}
For the first assertion, we use that the $(\bX_{-u}^{(n)})_n$ are i.i.d. For the second assertion, we compute 
$\PP(N_i=n_i|\bX_{-u}^{(1,2)}=\bx_{-u}^{(1,2)},N_j=n_j)$ with Bayes' theorem. 
\end{proof}

\begin{lm}\label{lm_francois}
If $N_i=n_i$, let $\bX_{-u}^{(\bM_i)}$ be the random vector composed of the $n_i$ observations in $B_i$ of $(\bX_{-u}^{(n)})_{n \in [3:N]}$ and $\bM_i \in [3:N]^{n_i}$ the vector containing the corresponding indices. We have:
\begin{eqnarray*}
& & \mathcal{L}\left(\bX^{(\bM_1)},\bX^{(\bM_2)}|\bX_{-u}^{(1,2)}=\bx_{-u}^{(1,2)},N_1=n_1,N_2=n_2\right)\\
&=&\mathcal{L}\left(\bX^{(\bM_1)}|\bX_{-u}^{(1,2)}=\bx_{-u}^{(1,2)},N_1=n_1\right)\otimes \mathcal{L}\left(\bX^{(\bM_2)}|\bX_{-u}^{(1,2)}=\bx_{-u}^{(1,2)},N_2=n_2\right).
\end{eqnarray*}
\end{lm}

\begin{proof}
For any bounded Borel functions $\phi_1,\;\phi_2$, we have
\begin{eqnarray*}
& & \E\left(\phi_1(\bX^{(\bM_1)}) \phi_2(\bX^{(\bM_2)})|\bX_{-u}^{(1,2)}=\bx_{-u}^{(1,2)},N_1=n_1,N_2=n_2 \right)\\
&=&\frac{ \E\left(\phi_1(\bX^{(\bM_1)}) \phi_2(\bX^{(\bM_2)})\mathds{1}_{N_1=n_1}\mathds{1}_{N_2=n_2} |\bX_{-u}^{(1,2)}=\bx_{-u}^{(1,2)}\right)}{\PP(N_1=n_1,N_2=n_2|\bX_{-u}^{(1,2)}=\bx_{-u}^{(1,2)})}.
\end{eqnarray*}
Let 
$$
\mathcal{P}([3:N],n_1):=\{(k_1,\cdots,k_{n_1})\in [3:N]^{n_1}\; : \; k_i< k_j \; \text{for}\; i,j\in[1:n_1],\; i< j\}
$$
be the set of all possible two-by-two distinct elements in $[3:N]$.
To simplify notation, we also consider an element of $\mathcal{P}([3:N],n_1)$ with the subset of $[3:N]$ that contains its indices. We have
\begin{eqnarray*}
& & \E\left(\phi_1(\bX^{(\bM_1)}) \phi_2(\bX^{(\bM_2)})\mathds{1}_{N_1=n_1}\mathds{1}_{N_2=n_2} \middle|\bX_{-u}^{(1,2)}=\bx_{-u}^{(1,2)}\right)\\
&=& \sum_{\bm_1 \in \mathcal{P}([3:N],n_1)} \sum_{\bm_2 \in \mathcal{P}([3:N]\setminus \bm_1, n_2)}\E\Big(\phi_1(\bX^{(\bm_1)}) \phi_2(\bX^{(\bm_2)})\mathds{1}_{\bX_{-u}^{(\bm_1)}\in B_1^{n_1}}\mathds{1}_{\bX_{-u}^{(\bm_2)}\in B_2^{n_2}}\\
& & \times \prod_{i \in [3:N]\setminus(\bm_1 \cup \bm_2)}\mathds{1}_{\bX_{-u}^{(i)}\notin B_1\cup B_2} \Big|\bX_{-u}^{(1,2)}=\bx_{-u}^{(1,2)}\Big)
\end{eqnarray*}
Now, using the independence of $(\bX^{(n)})_n$ and summing over $\bm_1$ and $\bm_2$, we have, for any value of $\bm_1 \in \mathcal{P}([3:N],n_1)$ and $\bm_2 \in  \mathcal{P}([3:N],n_2)$,
\begin{eqnarray*}
& & \E\left(\phi_1(\bX^{(\bM_1)}) \phi_2(\bX^{(\bM_2)})|\bX_{-u}^{(1,2)}=\bx_{-u}^{(1,2)},N_1=n_1,N_2=n_2 \right)\\
&=& \begin{pmatrix}N-2 \\ n_1 \end{pmatrix} \begin{pmatrix}N-2-n_1 \\ n_2 \end{pmatrix} (1-p_1-p_2)^{N-2-n_1-n_2} \\
& & \frac{\E\left( \phi_1(\bX^{(\bm_1)})\mathds{1}_{\bX_{-u}^{(\bm_1)}\in B_1^{n_1}} \middle|  \bX_{-u}^{(1,2)}=\bx_{-u}^{(1,2)} \right)\E\left( \phi_2(\bX^{(\bm_2)})\mathds{1}_{\bX_{-u}^{(\bm_2)}\in B_2^{n_2}} \middle|  \bX_{-u}^{(1,2)}=\bx_{-u}^{(1,2)} \right) }{\begin{pmatrix}N-2 \\ n_1 \end{pmatrix} \begin{pmatrix}N-2-n_1 \\ n_2 \end{pmatrix} p_1^{n_1} p_2^{n_2} (1-p_1-p_2)^{N-2-n_1-n_2}}\\
&=& \frac{\E\left( \phi_1(\bX^{(\bm_1)})\mathds{1}_{\bX_{-u}^{(\bm_1)}\in B_1^{n_1}} \middle|  \bX_{-u}^{(1,2)}=\bx_{-u}^{(1,2)} \right)}{ p_1^{n_1}} \frac{\E\left( \phi_2(\bX^{(\bm_2)})\mathds{1}_{\bX_{-u}^{(\bm_2)}\in B_2^{n_2}} \middle|  \bX_{-u}^{(1,2)}=\bx_{-u}^{(1,2)} \right) }{ p_2^{n_2}}\\
&=& \frac{\E\left( \phi_1(\bX^{(\bm_1)})\mathds{1}_{\bX_{-u}^{(\bm_1)}\in B_1^{n_1}}  \prod\limits_{i \in [3:N]\setminus \bm_1}\mathds{1}_{\bX_{-u}^{(i)}\notin B_1} \middle|  \bX_{-u}^{(1,2)}=\bx_{-u}^{(1,2)} \right)}{ p_1^{n_1}(1-p_1)^{n_1}}\\
& & \frac{\E\left( \phi_2(\bX^{(\bm_2)})\mathds{1}_{\bX_{-u}^{(\bm_2)}\in B_2^{n_2}}  \prod\limits_{i \in [3:N]\setminus \bm_2}\mathds{1}_{\bX_{-u}^{(i)}\notin B_2} \middle|  \bX_{-u}^{(1,2)}=\bx_{-u}^{(1,2)} \right)}{ p_2^{n_2}(1-p_2)^{n_2}}\\
&=&\E\left(\phi_1(\bX^{(\bM_1)}) \middle|\bX_{-u}^{(1,2)}=\bx_{-u}^{(1,2)},N_1=n_1 \right)\E\left(\phi_2(\bX^{(\bM_2)}) \middle|\bX_{-u}^{(1,2)}=\bx_{-u}^{(1,2)},N_2=n_2 \right).
\end{eqnarray*}
That concludes the proof of Lemma \ref{lm_francois}.
\end{proof}

\underline{Part 2.B:} We aim to proving that
$$
\left| \int_{\cX_{-u}^2}\E\left(\widehat{E}_{u,1}\widehat{E}_{u,2}|\bX_{-u}^{(1,2)}=\bx_{-u}^{(1,2)}\right) - \E(\widehat{E}_{u,1}|\bX_{-u}^{(1,2)}=\bx_{-u}^{(1,2)})\E(\widehat{E}_{u,2}|\bX_{-u}^{(1,2)}=\bx_{-u}^{(1,2)})d \PP_{\bX_{-u}}^{ \otimes 2}(\bx_{-u}^{(1)},\bx_{-u}^{(2)})\right| \leq  \frac{C_{\sup}(\varepsilon)}{N^{1-\varepsilon}}.
$$
To simplify notation, let $\bX^{(\bk_N)}:=(\bX^{(k_N(i))})_{i\leq N_I}$ and $\bX^{(\bk'_N)}:=(\bX^{(k'_N(i))})_{i\leq N_I}$.
We have
\begin{eqnarray*}
& & \E\left(\Phi(\bX^{(\bk_N)})\Phi(\bX^{(\bk_N')})|\bX_{-u}^{(1,2)}=\bx_{-u}^{(1,2)}\right)\\
&=&\sum_{n_1,n_2=0}^{N-2}\E\left(\Phi(\bX^{(\bk_N)})|N_1=n_1,\bX_{-u}^{(1,2)}=\bx_{-u}^{(1,2)}\right) \E\left(\Phi(\bX^{(\bk_N')})|N_2=n_2,\bX_{-u}^{(1,2)}=\bx_{-u}^{(1,2)}\right)\\
& & \times \PP(N_1=n_1,N_2=n_2|\bX_{-u}^{(1,2)}=\bx_{-u}^{(1,2)}).
\end{eqnarray*}
On the other hand, we have
\begin{eqnarray*}
& & \E\left(\Phi(\bX^{(\bk_N)})|\bX_{-u}^{(1,2)}=\bx_{-u}^{(1,2)}\right)\E\left( \Phi(\bX^{(\bk_N')})|\bX_{-u}^{(1,2)}=\bx_{-u}^{(1,2)}\right)\\
&=&\sum_{n_1,n_2=0}^{N-2}\E\left(\Phi(\bX^{(\bk_N)})|N_1=n_1,\bX_{-u}^{(1,2)}=\bx_{-u}^{(1,2)}\right) \E\left(\Phi(\bX^{(\bk_N')})|N_2=n_2,\bX_{-u}^{(1,2)}=\bx_{-u}^{(1,2)}\right)\\
& &\times \PP(N_1=n_1|\bX_{-u}^{(1,2)}=\bx_{-u}^{(1,2)})\PP(N_2=n_2|\bX_{-u}^{(1,2)}=\bx_{-u}^{(1,2)}).
\end{eqnarray*}
Thus, using that $\Phi$ is bounded and using Lemma \ref{lm_P(Gc)}, it suffices to show that
\begin{eqnarray*}
&\sum_{n_1,n_2=N_I-1}^{N-2}&\big|  \PP(N_1=n_1,N_2=n_2|\bX_{-u}^{(1,2)}=\bx_{-u}^{(1,2)})\\
&&-\PP(N_1=n_1|\bX_{-u}^{(1,2)}=\bx_{-u}^{(1,2)})\PP(N_2=n_2|\bX_{-u}^{(1,2)}=\bx_{-u}^{(1,2)}) \big|\leq \frac{C_{\sup}(\varepsilon)}{N^{1-\varepsilon}}.
\end{eqnarray*}
Let $K_N:=\lfloor N^{\alpha} \rfloor$, where $\alpha=\varepsilon\slash3$.
We divide the previous sum into two sums:
\begin{eqnarray*}
A(\bx_{-u}^{(1)},\bx_{-u}^{(2)}):=&\sum_{n_1,n_2=N_I-2}^{K_N}&\big|  \PP(N_1=n_1,N_2=n_2|\bX_{-u}^{(1,2)}=\bx_{-u}^{(1,2)})\\
&&-\PP(N_1=n_1|\bX_{-u}^{(1,2)}=\bx_{-u}^{(1,2)})\PP(N_2=n_2|\bX_{-u}^{(1,2)}=\bx_{-u}^{(1,2)}) \big|,\\
B(\bx_{-u}^{(1)},\bx_{-u}^{(2)}):=&\displaystyle\sum_{\substack{n_1,n_2=N_I-1,\\ n_1> K_N \text{ or } n_2 > K_N}}^{N-2}&\big|  \PP(N_1=n_1,N_2=n_2|\bX_{-u}^{(1,2)}=\bx_{-u}^{(1,2)})\\
&&-\PP(N_1=n_1|\bX_{-u}^{(1,2)}=\bx_{-u}^{(1,2)})\PP(N_2=n_2 | \bX_{-u}^{(1,2)}=\bx_{-u}^{(1,2)}) \big|.
\end{eqnarray*}
Let us bound these two terms.

First, we have
\begin{eqnarray*}
A(\bx_{-u}^{(1)},\bx_{-u}^{(2)})=&\sum_{n_1,n_2=N_I-1}^{K_N}&\PP(N_1=n_1|\bX_{-u}^{(1,2)}=\bx_{-u}^{(1,2)})\PP(N_2=n_2|N_1=n_1,\bX_{-u}^{(1,2)}=\bx_{-u}^{(1,2)})\\
& & \times \left| 1-\frac{\PP(N_2=n_2|\bX_{-u}^{(1,2)}=\bx_{-u}^{(1,2)})}{\PP(N_2=n_2|N_1=n_1,\bX_{-u}^{(1,2)}=\bx_{-u}^{(1,2)})}\right|.
\end{eqnarray*}
Thus, it suffices to bound 
$$
 \left| 1-\frac{\PP(N_2=n_2|\bX_{-u}^{(1,2)}=\bx_{-u}^{(1,2)})}{\PP(N_2=n_2|N_1=n_1,\bX_{-u}^{(1,2)}=\bx_{-u}^{(1,2)})}\right| \leq \frac{C_{\sup}(\varepsilon)}{N^{1-\varepsilon}}.
$$
Thus, it suffices to show
$$
\left| \log\left(\frac{\PP(N_2=n_2|\bX_{-u}^{(1,2)}=\bx_{-u}^{(1,2)})}{\PP(N_2=n_2|N_1=n_1,\bX_{-u}^{(1,2)}=\bx_{-u}^{(1,2)})}\right) \right| \leq \frac{C_{\sup}(\varepsilon)}{N^{1-\varepsilon}}.
$$
To simplify notation, let $T=N-2$.
Thanks to Lemma \ref{loi_Ni}, we have,
\begin{eqnarray*}
& & \log\left(\frac{\PP(N_2=n_2|\bX_{-u}^{(1,2)}=\bx_{-u}^{(1,2)})}{\PP(N_2=n_2|N_1=n_1,\bX_{-u}^{(1,2)}=\bx_{-u}^{(1,2)})}\right)\\
&=& \log\left(\frac{T(T-1)...(T-n_1+1)}{(T-n_2)(T-n_2-1)...(T-n_2-n_1+1)}\frac{(1-p_1)^{T-n_1}(1-p_2)^{T-n_2}}{(1-p_1-p_2)^{T-n_1-n_2}}\right)\\
&=&\log \left(1(1-\frac{1}{T})...(1-\frac{n_1-1}{T})\right)-\log \left((1-\frac{n_2}{T})(1-\frac{n_2+1}{T})...(1-\frac{n_2+n_1-1}{T})\right)\\
& & (T-n_1)\log(1-p_1)+(T-n_2)\log(1-p_2)-(T-n_1-n_2)\log(1-p_1-p_2)\\
&=&-\frac{n_1(n_1-1)}{2T}+n_1O(\frac{n_1^2}{T^2})+\frac{n_1(n_1+2n_2-1)}{2T}+n_1O(\frac{(n_1+n_2)^2}{T^2})\\
& & -(T-n_2)p_2+(T-n_2)O(p_2^2)-(T-n_1)p_1+(T-n_1)O(p_1^2)\\
& & +(T-n_1-n_2)(p_1+p_2)+(T-n_1-n_2)O((p_1+p_2)^2)\\
&=&\frac{n_1n_2}{T}+O(\frac{n_1^3}{T})+O(\frac{n_1(n_1+n_2)^2}{T^2})-n_2p_1-n_1p_2\\
& & +(T-n_2)O(p_1^2)+(T-n_1)O(p_2^2)+(T-n_1-n_2)O((p_1+p_2)^2).
\end{eqnarray*}
We know that 
$$
K_N p_i\leq  \frac{C_{\sup}}{N^{1-\delta|-u|-\alpha}}\leq  \frac{C_{\sup}}{N^{1-\varepsilon}} .
$$
So, for all $n_1\leq K_N$ and all $n_2 \leq K_N$,
$$
\left| \log\left(\frac{\PP(N_2=n_2|\bX_{-u}^{(1,2)}=\bx_{-u}^{(1,2)})}{\PP(N_2=n_2|N_1=n_1,\bX_{-u}^{(1,2)}=\bx_{-u}^{(1,2)})}\right) \right| \leq \frac{C_{\sup}(\varepsilon)}{N^{1-\varepsilon}}.
$$
Thus, we have shown that we have
$$
A(\bx_{-u}^{(1)},\bx_{-u}^{(2)})\leq \frac{C_{\sup}}{N^{1-\varepsilon}}.
$$
Now, let us bound $B(\bx_{-u}^{(1)},\bx_{-u}^{(2)})$. Remark that $\{(n_1,n_2)\in [N_I-1:N-2]|\;n_1> K_N \text{ or } n_2 > K_N \}$ is a subset of
$$
 \left([K_N+1:N-2]\times [N_I-1:N-2] \right) \cup \left( [N_I-1:N-2]\times [K_N+1:N-2]\right).
$$
Thus, it suffices to bound
\begin{eqnarray*}
&\sum_{n_1=K_N+1}^{N-2} \sum_{n_2=N_I-1}^{N-2}&\big|  \PP(N_1=n_1,N_2=n_2|\bX_{-u}^{(1,2)}=\bx_{-u}^{(1,2)})\\
&&-\PP(N_1=n_1|\bX_{-u}^{(1,2)}=\bx_{-u}^{(1,2)})\PP(N_2=n_2|\bX_{-u}^{(1,2)}=\bx_{-u}^{(1,2)}) \big| \\
&=&\sum_{n_1=K_N+1}^{N-2} \PP(N_1=n_1|\bX_{-u}^{(1,2)}=\bx_{-u}^{(1,2)})\\
 & & \sum_{n_2=N_I-1}^{N-2} \big|\PP(N_2=n_2|N_1=n_1,\bX_{-u}^{(1,2)}=\bx_{-u}^{(1,2)}) -\PP(N_2=n_2|\bX_{-u}^{(1,2)}=\bx_{-u}^{(1,2)}) \big|.
\end{eqnarray*}
Thus, it suffices to bound
$$
\sum_{n_1=K_N+1}^{N-2} \PP(N_1=n_1|\bX_{-u}^{(1,2)}=\bx_{-u}^{(1,2)}).
$$
Let $T:=N-2$. We know that $N_1$ has a binomial distribution with parameters $T$ and $p_1$. Thus, 
$$
\E(N_1)=p_1 T \leq C_{\sup}N^{\delta|-u|}\leq C_{\sup} N^{\frac{\varepsilon}{4}}.
$$
Thus, there exists $N_\varepsilon$ such that for $N\geq N_\varepsilon$, we have that, $\E(N_1)\leq K_T+1$. Thus, for $N$ large enough and for all $n_1> K_T$ and, we have
$$
\PP(N_1=n_1|\bX_{-u}^{(1,2)}=\bx_{-u}^{(1,2)})\leq \PP(N_1=K_T+1|\bX_{-u}^{(1,2)}=\bx_{-u}^{(1,2)}).
$$ 
Thus, for $N\geq N_\varepsilon$,
\begin{eqnarray*}
& & \sum_{n_1=K_N+1}^{N-2} \PP(N_1=n_1|\bX_{-u}^{(1,2)}=\bx_{-u}^{(1,2)})\\
& \leq & (T-K_T)\PP(N_1=K_T+1|\bX_{-u}^{(1,2)}=\bx_{-u}^{(1,2)})\\
&=&(T-K_T)\frac{T!}{(T-K_T-1)!(K_T+1)!}p_1^{K_T+1}(1-p_1)^{T-K_T+1}\\
& \leq &  (T-K_T)\frac{T! }{(T-K_T-1)!(K_T+1)!}p_1^{K_T+1}\\
& \leq & C_{\sup} \frac{(T-K_T)\sqrt{2\pi T} \left( \frac{T}{e}\right)^T \left( \frac{C_{\sup}}{T^{1-\delta|-u|}}\right)^{K_T+1}}{\sqrt{2\pi (K_T+1)}\left( \frac{K_T+1}{e} \right)^{(K_T+1)}\sqrt{2\pi (T-K_T-1)}\left( \frac{T-K_T-1}{e} \right)^{(T-K_T-1)}}\\
& \leq &  C_{\sup} \frac{(T-K_T)\sqrt{T}T^T C_{\sup}^{K_T+1}}{\sqrt{(K_T+1)(T-K_T-1)}(K_T+1)^{K_T+1}(T-K_T-1)^{T-K_T-1}T^{(1-\delta|-u|)(K_T+1)}}\\
& \leq &  C_{\sup} (T-K_T)^{K_T+\frac{3}{2}-T}(K_T+1)^{-K_T-\frac{3}{2}}T^{T-\frac{1}{2}+\delta |-u|(K_T+1)-K_T}C_{\sup}^{K_T+1}.
\end{eqnarray*}
Using the Taylor expansion of $x\mapsto \log(1-x)$ at $0$, we can see that 
$$
(T-K_T)^{-T}T^T\leq C_{\sup} \exp(K_T)\leq C_{\sup}^{K_T}.
$$ 
Moreover, we have
$$
(K_T+1)T^{1-\delta |-u|}\geq T^{\frac{\varepsilon}{3}}T^{1-\frac{\varepsilon}{4}}= T^{1+\frac{\varepsilon}{12}},
$$
and so
\begin{eqnarray*}
(T-K_T)^{K_T}(K_T+1)^{-K_T}T^{-K_T(1-\delta |-u|)}C_{\sup}^{K_T} &\leq &\exp\left( K_T \log \left[ C_{\sup} \frac{T-K_T}{T^{1+\frac{\varepsilon}{12}}}\right] \right)\\
& \leq & C_{\sup}(\varepsilon)e^{-K_T}.
\end{eqnarray*}
Thus, we have
\begin{eqnarray*}
& & \sum_{n_1=K_N+1}^{N-2} \PP(N_1=n_1|\bX_{-u}^{(1,2)}=\bx_{-u}^{(1,2)})\\
& \leq & C_{\sup}(\varepsilon)e^{-K_T}(T-K_T)^{\frac{3}{2}}(K_T+1)^{-\frac{3}{2}}T^{-\frac{1}{2}+\delta|-u|}\\
& \leq & \frac{C_{\sup}(\varepsilon)}{T}\\
& \leq & \frac{C_{\sup}(\varepsilon)}{N}.
\end{eqnarray*}
Finally, we have
$$
A(\bx_{-u}^{(1)},\bx_{-u}^{(2)})\leq \frac{C_{\sup}}{N^{1-\varepsilon}},\text{\;\;\; and \;\;\;}B(\bx_{-u}^{(1)},\bx_{-u}^{(2)})\leq \frac{C_{\sup}(\varepsilon)}{N}.
$$
Thus
\begin{eqnarray*}
&\sum_{n_1,n_2=N_I}^{N}&\big|  \PP(N_1=n_1,N_2=n_2|\bX_{-u}^{(1,2)}=\bx_{-u}^{(1,2)})\\
&&-\PP(N_1=n_1|\bX_{-u}^{(1,2)}=\bx_{-u}^{(1,2)})\PP(N_2=n_2|\bX_{-u}^{(1,2)}=\bx_{-u}^{(1,2)}) \big|\leq \frac{C_{\sup}(\varepsilon)}{N^{1-\varepsilon}}.
\end{eqnarray*}
So, we have proved Proposition \ref{prop_vitesse_cov}.
\end{proof}

We conclude by the proof of Theorem \ref{thrm_vitesse_MC}.
\begin{proof}
\begin{eqnarray*}
& & \PP\left( \left| \widehat{E}_{u}-E_u \right| >\varepsilon \right)\\
& \leq & \PP\left( \left| \widehat{E}_{u}-\E(\widehat{E}_{u}) \right| > \frac{\varepsilon}{2} \right)+ \PP\left( \left|\E(\widehat{E}_{u})-E_u \right| > \frac{\varepsilon}{2} \right).
\end{eqnarray*}
Then, we use the proof of Proposition \ref{prop_conv_proba_biaise}. If $(s(l))_{l\leq N_u}$ is a sample of uniformly distributed variables on $[1:N]$ with replacement, then for all $\varepsilon>0$,
\begin{eqnarray*}
\PP\left( \left| \widehat{E}_{u}-\E(\widehat{E}_{u}) \right| > \frac{\varepsilon}{2} \right) & \leq &  \frac{4}{\varepsilon^2}\left(\left|cov\left( \widehat{E}_{u,1},\widehat{E}_{u,2}\right)\right|+\V\left( \widehat{E}_{u,1}\right)\left( \frac{1}{N}+\frac{1}{N_u}\right) \right)\\
& \leq & \frac{1}{\varepsilon^2}\left( \frac{C_{\sup}(\varepsilon')}{N^{\frac{1}{p-|u|}-\varepsilon'}}+\frac{C_{\sup}}{N_u} \right),
\end{eqnarray*}
for all $\varepsilon'>0$, thanks to Proposition \ref{prop_vitesse_cov}.
 If $(s(l))_{l\leq N_u}$ is a sample of uniformly distributed variables on $[1:N]$ without replacement, then for all $\varepsilon>0$,
\begin{eqnarray*}
\PP\left( \left| \widehat{E}_{u}-\E(\widehat{E}_{u}) \right| > \frac{\varepsilon}{2} \right) & \leq & \frac{4}{\varepsilon^2}\left(\frac{N_u-1}{N_u}cov\left( \widehat{E}_{u,1},\widehat{E}_{u,2}\right)+\frac{1}{N_u}\V\left( \widehat{E}_{u,1}\right) \right)\\
& \leq & \frac{1}{\varepsilon^2}\left( \frac{C_{\sup}(\varepsilon')}{N^{\frac{1}{p-|u|}-\varepsilon'}}+\frac{C_{\sup}}{N_u} \right),
\end{eqnarray*}
for all $\varepsilon'>0$, thanks to Proposition \ref{prop_vitesse_cov}.
Moreover, for all $\varepsilon>0$,
\begin{eqnarray*}
 \PP\left( \left|\widehat{E}_{u}-E_u \right| > \frac{\varepsilon}{2} \right) & \leq & \frac{2}{\varepsilon}\left|\E(\widehat{E}_{u})-E_u \right|\\
 & \leq & \frac{C_{\sup}(\varepsilon')}{\varepsilon N^{\frac{1}{p-|u|}-\varepsilon'}},
\end{eqnarray*}
for all $\varepsilon'>0$, thanks to Proposition \ref{prop_vitesse_esperance}.
Finally, for all $\varepsilon>0$, $\varepsilon'>0$, we have
\begin{equation*}
 \PP\left( \left| \widehat{E}_{u}-E_u \right| >\varepsilon \right) \leq \frac{1}{\varepsilon^2}\left( \frac{C_{\sup}(\varepsilon')}{N^{\frac{1}{p-|u|}-\varepsilon'}}+\frac{C_{\sup}}{N_u} \right).
\end{equation*}
That concludes the proof.
\end{proof}

\section{Other proofs}
\leavevmode\par
\bigskip

\textbf{Proof of Proposition \ref{prop_paf}}
\begin{proof}
\begin{eqnarray*}
& & \E(f(\bX)f(\bX^{u}))\\
&=&\E(\E(f(\bX)f(\bX^{u})|\bX_u))\\
&=&\E\left(\int_{\cX_{-u}^2}f(\bX_u,\bx_{-u})f(\bX_u,\bx_{-u}')d\PP_{\bX_{-u}|\bX_u}\otimes \PP_{\bX_{-u}|\bX_u}(\bx_{-u},\bx_{-u}')   \right)\\
&=&\E\left(\int_{\cX_{-u}}f(\bX_u,\bx_{-u})d\PP_{\bX_{-u}|\bX_u}(\bx_{-u}) \int_{\cX_{-u}}f(\bX_u,\bx_{-u}')d\PP_{\bX_{-u}|\bX_u}(\bx_{-u}')   \right)\\
&=& \E\left(\E(f(\bX)|\bX_u)^2\right).
\end{eqnarray*}
That concludes the proof of Proposition \ref{prop_paf}.
\end{proof}

\textbf{Proof of Proposition \ref{prop_min_variance}}
\begin{proof}
Let
$$
A_{i,u}:=\left\{ \begin{array}{ll}
-\frac{1}{p}\begin{pmatrix}
p-1\\ |u|
\end{pmatrix}^{-1} & \text{if }i \notin u\\
\frac{1}{p}\begin{pmatrix}
p-1\\ |u|-1
\end{pmatrix}^{-1} & \text{if }i \in u.
\end{array}\right.
$$
Under Assumption \ref{hypNu}, we have

\begin{eqnarray*}
\V(Y)^2\sum_{i=1}^p \V(\widehat{\eta}_i)&=&\sum_{i=1}^p\sum_{\emptyset \varsubsetneq u\varsubsetneq [1:p]} A_{i,u}^2\V(\widehat{W}_u)\\
&=&\sum_{\emptyset \varsubsetneq u\varsubsetneq [1:p]}\V(\widehat{W}_u)\sum_{i=1}^p A_{i,u}^2\\
&=&\sum_{\emptyset \varsubsetneq u\varsubsetneq [1:p]}\frac{\V(\widehat{W}_u^{(1)})}{N_u}\sum_{i=1}^p A_{i,u}^2.
\end{eqnarray*}
Moreover,
\begin{eqnarray*}
\sum_{i=1}^p A_{i,u}^2&=&\sum_{i \in -u } \frac{1}{p^2} \begin{pmatrix}
p-1\\ |u|
\end{pmatrix}^{-2}+\sum_{i \in u } \frac{1}{p^2} \begin{pmatrix}
p-1\\ |u|-1
\end{pmatrix}^{-2}\\
&=& \frac{1}{p!^2}\left( (p-|u|)|u|!^2(p-|u|-1)!^2+|u|(|u|-1)!^2(p-|u|)!^2 \right)\\
&=&\frac{(p-|u|)!|u|!}{p!^2}(p-|u|-1)!(|u|-1)!(|u|+p-|u|)\\
&=&\frac{(p-|u|)!|u|!}{p!}\frac{(p-|u|-1)!(|u|-1)!}{(p-1)!}\\
&=:&C(|u|,p).
\end{eqnarray*}

Thus, we want to minimize 
$$
\sum_{\emptyset \varsubsetneq u\varsubsetneq [1:p]}\frac{\V(\widehat{W}_u^{(1)})}{N_u}C(|u|,p)
$$
subject to
$$
\sum_{\emptyset \varsubsetneq u\varsubsetneq [1:p]}N_u=\frac{N_{tot}}{\kappa}.
$$
Let $U=(\R_+^*)^{2^p-2}$. If $\bx \in U$, we index the components of $\bx$ by the subsets $\emptyset \varsubsetneq u \varsubsetneq [1:p]$ and we write $\bx=(x_u)_{\emptyset \varsubsetneq u \varsubsetneq [1:p]}$. Let $h$ be the $C^1$ function on $U$ defined by 
$h(\bx)=\sum_{\emptyset \varsubsetneq u\varsubsetneq [1:p]}\frac{C(|u|,p)\V(\widehat{W}_{u}^{(1)})}{x_u}$,
let $g$ be the $C^1$ function on $U$ defined by $g(\bx)=(\sum_{\emptyset \varsubsetneq u\varsubsetneq [1:p]}x_u)-N_{tot}\slash \kappa$ and let $A=g^{-1}(\{0\})$.  Using the method of Lagrange multipliers, if $h_{|A}$ has a local minimum in $\textbf{a}$, there exists $c$ such that $Dh(\textbf{a})=cDg(\textbf{a})$, i.e. $\nabla h(\textbf{a})=\nabla g(\textbf{a})$ i.e. 
$$
\textbf{a}=\frac{N_{tot}}{\kappa \sum_{\emptyset \varsubsetneq v\varsubsetneq [1:p]} \sqrt{C(|v|,p) \V(\widehat{W}_{v}^{(1)})}}\left(\sqrt{C(|u|,p) \V (\widehat{W}_{u}^{(1)}})\right)_{\emptyset \varsubsetneq u\varsubsetneq [1:p]}.
$$
Moreover, note that $h$ is strictly convex and the set $A$ is convex, thus $h_{|A}$ is strictly convex. Thus $\textbf{a}$ is the strict global minimum point of $h_{|A}$.
\end{proof}

\textbf{Proof of Proposition \ref{prop_min_pu}}
\begin{proof}
Let us write $V:=\V(\widehat{W}_u^{(1)})$ that does not depend on $u$ by assumption. To simplify notation, let $N_0=N_p=+\infty$. In this way, we have, for all $u\subset [1:p]$, $\V(\widehat{W}_u(m))=V \slash N_{|u|}$.

We have
\begin{eqnarray*}
\V\left( \left. \widehat{\eta}_i \right| (\sigma_m)_{m\leq M} \right) &=&  \frac{1}{p^2\V(Y)^2}\sum_{u \subset -i}\frac{1}{M^2}\sum_{m=1}^M\left[\V\left(\widehat{W}_{u\cup \{i\}}(m)\right)+\V\left(\widehat{W}_{u}(m)\right)\right]\mathds{1}_{ P_i(\sigma_m)=u}\\
&=&\frac{V}{p^2\V(Y)^2}\sum_{u \subset -i}\frac{1}{M^2}\sum_{m=1}^M\left[ \frac{1}{N_{|u \cup \{i\}|}}+\frac{1}{N_u} \right]\mathds{1}_{ P_i(\sigma_m)=u}.
\end{eqnarray*}
Thus,
\begin{eqnarray*}
\E\left[ \V\left( \left. \widehat{\eta}_i \right| (\sigma_m)_{m\leq M} \right)\right]&=&\frac{V}{p^2\V(Y)^2}\sum_{u \subset -i}\frac{1}{M^2}\sum_{m=1}^M\left[ \frac{1}{N_{|u \cup \{i\}|}}+\frac{1}{N_u} \right]\PP( P_i(\sigma_m)=u)\\
&=&\frac{V}{p^2\V(Y)^2}\sum_{u \subset -i}\frac{1}{M^2}\sum_{m=1}^M \frac{1}{p}\begin{pmatrix}p-1\\ |u| \end{pmatrix}^{-1}\left[ \frac{1}{N_{|u \cup \{i\}|}}+\frac{1}{N_{|u|}} \right]\\
&=&\frac{V}{p^2\V(Y)^2} \sum_{u\subset[1:p]} a_{i,u} \frac{1}{N_{|u|}},
\end{eqnarray*}
where
$$
a_{i,u}:=\left\{ \begin{array}{cc}
\frac{1}{p} \begin{pmatrix}p-1\\ |u| \end{pmatrix}^{-1}    & \text{if } i\notin u  \\
  \begin{pmatrix}p-1\\ |u|-1 \end{pmatrix}^{-1}    & \text{if } i\in u.
\end{array}\right.
$$
Remark that $\sum_{i=1}^pa_{i,u}=2\begin{pmatrix} p\\ |u| \end{pmatrix}^{-1}$. Then,
\begin{eqnarray*}
\E\left[ \sum_{i=1}^p\V\left( \left. \widehat{\eta}_i \right| (\sigma_m)_{m\leq M} \right)\right]&=&\sum_{i=1}^p \frac{V}{p^2\V(Y)^2} \sum_{u\subset[1:p]} a_{i,u} \frac{1}{N_{|u|}}\\
&=& \frac{V}{p^2\V(Y)^2}   \sum_{u\subset[1:p]} \frac{1}{N_{|u|}}  \sum_{i=1}^p  a_{i,u}\\
&=& \frac{2V}{p^2\V(Y)^2}  \sum_{u\subset[1:p]}  \frac{1}{N_{|u|}} \begin{pmatrix} p\\ |u| \end{pmatrix}^{-1}\\
&=& \frac{2V}{p^2\V(Y)^2} \sum_{k=1}^{p-1} \frac{1}{N_k}
\end{eqnarray*}
We get the relaxed problem
\begin{eqnarray*}
\min_{(N_k)_{k\in [1:p-1]}}\frac{2V}{p^2\V(Y)^2} \sum_{k=1}^{p-1} \frac{1}{N_k}
\end{eqnarray*}
subject to $M\sum_{k=1}^{p-1}N_k=MN_O(p-1)$. Let $U=(\R_+^*)^{p-1}$. Let $h$ be the $C^1$ function on $U$ defined by $h(\bx)=\frac{2V}{p^2\V(Y)^2} \sum_{k=1}^{p-1} \frac{1}{x_k}$, $g$ be the $C^1$ function on $U$ defined by $g(\bx)=M\sum_{k=1}^{p-1}x_k-MN_O(p-1)$. Finally, let $A=g^{-1}(\{0\})$. Using the method of Lagrange multipliers, if $h_{|A}$ has a local minimum in $\ba$, there exists $c$ such that $Dh(\ba)=cDg(\ba)$, i.e. $\nabla h(\ba)=\nabla g(\ba)$ i.e. $\forall u, \;-\frac{1}{a_u^2}=c'$ i.e. $a_u=c''$. To sum up, if $h_{|A}$ has a local minimum, it is in $\ba$ defined by
$$
a_u=N_OMp_u.
$$
Moreover, note that $h$ is strictly convex and the set $A$ is convex, thus $h_{|A}$ is strictly convex. Thus $\ba$ is the strict global minimum point of $h_{|A}$. Thus, $\ba$ is the global minimum on the constraint problem (where the inputs are integers).
\end{proof}

\textbf{Proof of Proposition \ref{prop_choix_param}}
This proof totally arises from the appendix of \cite{song_shapley_2016}. The computations are the same.
\begin{proof}
Under Assumption \ref{assum_sansbiais}, we have
\begin{eqnarray*}
\V(\widehat{\eta}_i)&=&\frac{1}{M\V(Y)^2}\left( \V\left( \widehat{W}_{P_{i}(\sigma_1)\cup\{i\}} \right)+ \V\left( \widehat{W}_{P_{i}(\sigma_1)} \right) \right)\\
&=&\frac{1}{M\V(Y)^2}\bigg( \V(\E (\widehat{W}_{P_{i}(\sigma_1)\cup\{i\}} |\sigma_1))+\E(\V(\widehat{W}_{P_{i}(\sigma_1)\cup\{i\}} |\sigma_1))\\
& &+ \V(\E( \widehat{W}_{P_{i}(\sigma_1)}|\sigma_1))+\E(\V( \widehat{W}_{P_{i}(\sigma_1)}|\sigma_1)) \bigg)\\
&=&\frac{1}{C\V(Y)^2}\bigg( N_O\V(W_{P_{i}(\sigma_1)\cup\{i\}})+N_O\V(W_{P_{i}(\sigma_1)})\\ & & +\E(\V(\widehat{W}_{P_{i}(\sigma_1)\cup\{i\}}^{(1)} |\sigma_1))+\E(\V(\widehat{W}_{P_{i}(\sigma_1)}^{(1)} |\sigma_1)) \bigg).
\end{eqnarray*}
Thus, the minimum is with $N_O=1$.
\end{proof}

\textbf{Proof of Proposition \ref{prop_consistance_random-permutation $W$-aggregation procedure}}
\begin{proof} We only prove the second item. The first one is easier and uses the same idea. Let $i \in [1:p]$. Remark that
\begin{eqnarray*}
\widehat{\eta}_i&=&\frac{1}{M\V(Y)}\sum_{m=1}^M\left(\widehat{W}_{P_{i}(\sigma_m)\cup\{i\}}(m)- \widehat{W}_{P_i(\sigma_m)}(m)\right)\\
& =&\frac{1}{p\V(Y)}\sum_{u\subset -i}\begin{pmatrix}
p-1\\ |u|
\end{pmatrix}^{-1} \left(\tilde{W}_{u\cup\{i\},i }- \tilde{W}_{u,i} \right)
\end{eqnarray*}
with
$$
\tilde{W}_{u,i}:=\begin{pmatrix}
p-1\\ |u|
\end{pmatrix} \frac{p}{M}\sum_{m|\;P_i(\sigma_m)=u}\widehat{W}_{u}(m)\;\;\;\text{and}\;\;\;
\tilde{W}_{u\cup\{i\},i}:=\begin{pmatrix}
p-1\\ |u|
\end{pmatrix} \frac{p}{M}\sum_{m|\;P_i(\sigma_m)=u}\widehat{W}_{u\cup\{i\}}(m),
$$
where we sum over all the integers $m \in [1:M]$ such that $P_i(\sigma_m)=u$.
Thus, for all $u$,
$$
\tilde{W}_{u,i} \sim  \begin{pmatrix}
p-1\\ |u\setminus \{i\}|
\end{pmatrix} \frac{p}{M} \tilde{N}_{u,i,M} \widehat{W}_{u}^{\tilde{N}_{u,i,M}},
$$
where
$$
\widehat{W}_{u}^{\tilde{N}_{u,i,M}}:= \frac{1}{\tilde{N}_{u,i,M}}\sum_{k=1}^{\tilde{N}_{u,i,M}}\widehat{W}_{u}(k),
$$
and $\tilde{N}_{u,i,M}=\tilde{N}_{u\cup\{i\},i,M}\sim \mathcal{B}(M,\frac{|u|!(p-1-|u|)!}{p!})$ (the binomial distribution).
Now, remark that $M$ goes to $+\infty$ when $N_{tot}$ goes to $+\infty$ (recall that $N_{tot}= \kappa M(p-1)$). Hence,
$$
\begin{pmatrix}
p-1\\ |u\setminus \{i\}|
\end{pmatrix} \frac{p}{M} \tilde{N}_{u,i,M} \overset{\PP}{\underset{N_{tot}\rightarrow +\infty}{\longrightarrow}}1.
$$
It suffices to show that for all $u\subset[1:p]$, the estimator $\omega\mapsto\widehat{W}_{u}^{\tilde{N}_{u,i,M}(\omega)}(\omega)$ converges to $W_u$ in probability when $N$ and $N_{tot}$ go to $+\infty$ and we could conclude by
\begin{eqnarray*}
\widehat{\eta}_i & = & \frac{1}{p\V(Y)}\sum_{u\subset -i}\begin{pmatrix}
p-1\\ |u|
\end{pmatrix}^{-1} \left(\tilde{W}_{u\cup\{i\},i }- \tilde{W}_{u,i} \right)\\
& \overset{\PP}{\underset{\substack{ N_{tot}\rightarrow +\infty\\ N \rightarrow +\infty}}{\longrightarrow}} & \frac{1}{p\V(Y)}\sum_{u\subset -i}\begin{pmatrix}
p-1\\ |u|
\end{pmatrix}^{-1} \left(W_{u\cup\{i\} }- W_{u} \right)\\
&=& \eta_i.
\end{eqnarray*}
Let $\varepsilon>0$ and $\delta>0$. Using the assumptions and Chebyshev's inequality, we have that $(\widehat{W}_{u}^{N_O})_{N_O,N}$ is consistent, thus there exists $N_{O1}$ and $N_1$ such that for all $N_O\geq N_{O1}$ and all $N \geq N_1$,
$$
\PP\left( \left|\widehat{W}_u^{N_O}-W_u \right|>\delta\right) <\frac{\varepsilon}{2}.
$$
Moreover, 
\begin{eqnarray*}
\PP(\tilde{N}_{u,M}\leq N_{O1})
\underset{M\rightarrow+\infty}{\longrightarrow}0.
\end{eqnarray*}
Thus, there exists $M_1$ such that for all $M\geq M_1$,
$$
\PP(\tilde{N}_{u,M}\leq N_{O1})<\frac{\varepsilon}{2}.
$$
Thus, there exists $N_{tot1}$ such that for all $N_{tot}\geq N_{tot1}$,
$$
\PP(\tilde{N}_{u,M}\leq N_{O1})<\frac{\varepsilon}{2}.
$$
Finally, for all $N_{tot}\geq N_{tot1}$ and $N\geq N_1$, we have
\begin{eqnarray*}
\PP\left( \left|\widehat{W}_u^{\tilde{N}_{u,M}}-W_u \right|>\delta\right)&\leq & \PP\left( \left|\widehat{W}_u^{\tilde{N}_{u,M}}-W_u \right|>\delta,\;\;\tilde{N}_{u,M}\geq N_{O1}\right)+\PP(\tilde{N}_{u,M}\leq N_{O1})\\
&<&\varepsilon.
\end{eqnarray*}
That proves that the estimator $\omega\mapsto\widehat{W}_{u}^{\tilde{N}_{u,i,M}(\omega)}(\omega)$ converges to $W_u$ in probability when $N$ and $N_{tot}$ go to $+\infty$.
\end{proof}

\textbf{Proof of Corollary \ref{corol_MC} and Corollary \ref{corol_PF}}

We do the proof for Corollary \ref{corol_MC}. The proof of Corollary \ref{corol_PF} uses the same idea.
\begin{proof}
Let $\delta>0$. Thanks to Theorem \ref{thrm_vitesse_MC}, with $\varepsilon'=\delta$, we have
\begin{eqnarray*}
\PP\left(N^{\frac{1}{2(p-|u|)}-\delta}\left| \widehat{E}_{u,MC}-E_{u} \right|>\varepsilon\right)& \leq & \frac{C_{\sup}(\delta)N^{\frac{1}{p-|u|}-2\delta}}{\varepsilon^2 N^{\frac{1}{p-|u|}-\delta}}\underset{N\rightarrow+\infty}{\longrightarrow}0.
\end{eqnarray*}
That concludes the proof of Corollary \ref{corol_MC}. 
\end{proof}

\textbf{Proof of Proposition \ref{prop_consis_shap_inconnue}}
\begin{proof}
If we use the subset $W$-aggregation procedure, we just have to use the consistency of $\widehat{W}_u$ from Theorems \ref{thrm_consis_MC} and \ref{thrm_consis_PF} and to use Proposition \ref{prop_consistance_shapley1}.

If we use the subset $W$-aggregation procedure, the consistency of the estimators of the Shapley effects comes from the second part of Proposition \ref{prop_consistance_random-permutation $W$-aggregation procedure}. We just have to verify Assumption \ref{assum_consis_perm}. Let $\widehat{W}_{u}(m)$ of Proposition \ref{prop_consistance_random-permutation $W$-aggregation procedure} be $\widehat{E}_{u,s(m),MC}$ or $\widehat{V}_{u,s(m),PF}$ defined in Section \ref{section_Wu_inconnue}, where $(s(m))_{m}$ are independent and uniformly distributed on $[1:N]$. Then, following the end of the proof of Theorems \ref{thrm_consis_MC} and \ref{thrm_consis_PF}, we obtain
$$
\frac{1}{M^2} \sum_{m,m'=1}^{M} cov\left( \widehat{W}_{u}{(m)},\widehat{W}_{u}{(m')} \right) \underset{N,M\rightarrow+\infty}{\longrightarrow}0,
$$
and, by Proposition \ref{prop_asymp_sans_bias}, we have
$$
\E\left( \widehat{W}_u{(1)} \right)=\E\left( \widehat{W}_u^{(1)} \right) \underset{N\rightarrow +\infty}{\longrightarrow} W_u.
$$
Thus, Assumption \ref{assum_consis_perm} holds.
\end{proof}

\bibliographystyle{alpha}
\bibliography{biblio}

\end{document}